\documentclass[arxiv,reqno,bibliography=totoc,twoside,a4paper,12pt]{amsart}

\usepackage[utf8]{inputenc}

\usepackage{amsmath,bbm}
\usepackage{amsfonts}
\usepackage{amssymb}
\usepackage{amsthm}
\usepackage{float}
\usepackage{mathtools}
\usepackage{csquotes}
\usepackage{upgreek,esint,enumerate}
\usepackage[numbers]{natbib}
\usepackage{dsfont,ednotes}
\usepackage{enumitem}
\setlist{nosep}
\usepackage{graphicx}

\usepackage[colorlinks=true,linkcolor=blue,citecolor=blue, urlcolor=blue]{hyperref}
\usepackage{cleveref}

\linespread{1.05}
\usepackage[scaled]{helvet} 
\usepackage{courier} 
\usepackage[mathbf]{euler}
\usepackage[dvipsnames]{xcolor}

\usepackage[driver=pdftex,margin=3cm,heightrounded=true,centering]{geometry}

\usepackage{tikz-cd}
\usetikzlibrary{arrows,matrix,shapes,decorations.text, mindmap,shapes.misc,decorations.markings}

\usepackage{ifthen}
\usepackage{metalogo}
\usepackage{Styles/commands}
\usepackage{sepfootnotes}			
\newfootnotes{b}				





\bnotecontent{f1}{Since each level set is an embedded hypersurface, the inclusion maps $\inclus_t:\, \Sigma_t\,\hookrightarrow\,M$ are smooth embeddings, i.e. $\inclus_t$ are diffeomorphisms between $\Sigma_t$ and $\inclus_t(\Sigma_t)$.}

\bnotecontent{f2}{The term \textit{inward/outward pointing normal vector field} can be rephrased in the Lorentzian setting: if the hypersurface is spacelike, one fixes one of the two timelike normal directions as normal field (here past-directed) and calls it outward pointing. The other time-oriented normal vector is thus the inward pointing normal vector (here future-directed).
}
\bnotecontent{f1c}{Here and in the following if the manifold is clear from the context or from the bundle $E$, we write $C^\infty(E)$ instead of $C^\infty(M,E)$ and the same for other sections.}

\bnotecontent{f3}{Here and in the following we choose the sesquilinearity in such a way that the first entry is anti-linear.}

\bnotecontent{f4}{Another way of introducing a $\Psi$DO is presented in the appendix \ref{chap:app2} as a Fourier integral operator.}

\bnotecontent{f5}{A more general version of Seeleys result in 1964 can be found in \cite{Melbook} as Theorem and Lemma 1.4.1. in chapter 1.}

\bnotecontent{f11}{In comparison, $L^2_\comp(M,E)$ denotes those sections in $L^2_\loc(M,E)$, having compact support. In order to stress a fixed compact support in $K \Subset M$, I write $$L^2_{K,\loc}(M,E):=\SET{u \in L^2_\loc(M,E)\,\vert\, \supp{u} \subset K}\quad. $$ Then one can define $L^2_\comp(M,E)$ as union of $L^2_{K,\loc}(M,E)$ over all $K$ compact in $M$. }

\bnotecontent{f12}{This follows from the essentially self-adjointness of the identity map and the Bochner Laplacian $(\Nabla{E}{})^\ast\Nabla{E}{}$, see \cite{BraMilShub} or \cite{yoshida}. It holds true for any compact manifold since they are complete by the Hopf-Rinow theorem.}

\bnotecontent{f13}{If one has $k+\alpha < s-\frac{\dim(M)}{2}$ with $\alpha \in \intervallo{}{0}{1}$, one has even more a continuous embedding into the Hölder spaces $C^{k,\alpha}(M,E)$.}

\bnotecontent{f27}{In comparison the \textit{domain of dependence} (also known as \textit{"causal diamond"}) is defined as $\domdep{}(A)=\domdep{+}(A)\cup \domdep{-}(A)$ for $A \subset M$ such that no two contained points can be connected by a timelike curve (i.e. $A$ is achronal) where
\begin{equation*}
\domdep{\pm}(A):=\SET{p \in M\,\vert\,\text{every past}(+)\setminus\,\text{future}(-)\,\text{inextendible causal curve through}\,p\,\text{meets}\,A}\quad.
\end{equation*}
In particular $A \subset \domdep{\pm}(A)$, but also $\domdep{\pm}(A)\subset\Jlight{\pm}(A)$. }

\bnotecontent{fc1}{Later on $\mathcal{M}=\Sigma$ or $\Sigma_t$ for a $t \in \timef(M)$.}


\bnotecontent{f6}{Thus, it is an oriented manifold together with a $\group{Spin}_0(1,n)$-principal bundle where 
\begin{equation*}
\group{Spin}_0(1,n):=\SET{e_1,...,e_{2k}\in \group{Cl}^0_{1,n}\,\vert\,v_j \in \R^{n+1}\, , \, \idscal{1}{}{v_j}{v_j}=\pm 1 \quad \text{and} \quad \prod_{j=0}^{2k}\idscal{1}{}{v_j}{v_j}=1}
\end{equation*}
with $\group{Cl}^0_{1,n}$ as even elements of the Clifford-algebra w.r.t. the symmetric bilinearform $\idscal{1}{}{\cdot}{\cdot}$ is the connected component of the identity in the spin group; see \cite{BaerGauMor} and related literature for more details.}

\bnotecontent{f10}{Sections of $\spinb^{\pm}(\Sigma_t)$ for each hypersurface can be interpreted as spinor fields on the hypersurface with positive or negative chirality, even though I used it as $\spinb^{\pm}(\Sigma_{t})=\spinb^{\pm}(M)\vert_{\Sigma_t}=\spinb(\Sigma_t)$ for both signs. But since I won't apply a chirality decomposition on the hypersurfaces, this shouldn't bother here.}

\bnotecontent{f14}{See \cref{remsdiracop} (iii).}

\bnotecontent{f28}{Manifolds with bounded geometry are complete which follows from the injectitivity radius condition; see \cite[Def.1.1]{shubinspec} and following. Thus, the results from the former sections can be applied.}


\bnotecontent{f7}{This circumstand might lead to an obstruction in using more general boundary conditions as e.g. given in \cite{§4.2}{BaerHan} where the multiplicities induce an additional term in the index formula for the compact case if one is using this kind of generalized Atiyah-Patodi-Singer boundary conditions.}

\bnotecontent{f8}{The same procedure will be applied for the flip metric of $\met$ in order to compute the geometric index formula.}

\bnotecontent{f9}{It is also possible to achieve that every leave $\mathcal{T}=const$ in a globally hyperbolic manifold is complete and a Cauchy hypersurface by modifying the lapse function $N$ to be bounded on every leave; this shown in \cite[Prop.15]{mueller}.}

\bnotecontent{f15}{See appendix \ref{chap:app2} for the notations.}

\bnotecontent{f20}{This can be seen in Definition \ref{defevop} since the Dirac wave evolution operator maps compactly supported Sobolev sections to compactly supported Sobolev sections and is a topological isomorphism, so its adjoint does as well.}



\bnotecontent{f16}{The boundness on $\kernel{B}$ is inherited from the boundness on $\mathscr{H}$.}

\bnotecontent{f17}{Equivalently, the orbit of each point in $M$ is $M$ itself.}

\bnotecontent{f18}{Another reference for the Dirac operator on a Riemannian $\Upgamma$-manifold is \cite[Prop.3.1]{atiyahellvn}.}

\bnotecontent{f19}{$B\ydo{\ast}{(\mathsf{cl})}$ is closed under composition.}

\bnotecontent{f29}{Shubin introduced in \cite{shub} this class to consist of those pseudo-differential operators which differ from a properly supported pseudo-differential operator in a smoothing operator in $\ydo{-\infty}{\Upgamma}$. But the definition given here coincides with the manner, how these operators have been actually used in this reference.}

\bnotecontent{f21}{See e.g. \cite[Thm.3.2]{Joshi} for the more general statement that the tensor product of two local Sobolev sections is again a local section.}

\bnotecontent{f22}{Since $\upnu$ is chosen to be globally past-directed, the temporal projection of $\zeta_{+}$ is anti-parallel to $\upnu$ whereas the projection of $\zeta_{-}$ in temporal direction is parallel; the same holds for $\varsigma_{\pm}$.}



\bnotecontent{f23}{Projections on a Banach space are bounded iff their kernels and range are closed subspaces of this Banach space.}

\bnotecontent{f24}{See \cite[App.A]{benroy} for the unitary identification of $\mathscr{K}_\Upgamma(\mathscr{H})$ and $\mathscr{N}_{r}(\Upgamma)\otimes\mathscr{K}(\mathcal{H})$.}

\bnotecontent{f25}{Here and in the following of this proof every appearing topological isomorphism between Hilbert $\Upgamma$-modules can be considered as unitary isomorphism; see \cite[Prop.2.16]{shub}. Thus, every appearing isomorphism between Hilbert $\Upgamma$-modules implies that their $\Upgamma$-dimensions coincide.}

\bnotecontent{f26}{The mixed terms can be estimated with the polarization identity of the form
\begin{equation*}
4\Rep{\dscal{1}{\mathcal{H}}{x}{y}}=\norm{x+y}{\mathcal{H}}^2-\norm{x-y}{\mathcal{H}}^2 \leq \norm{x+y}{\mathcal{H}}^2+\norm{x-y}{\mathcal{H}}^2 = 2\norm{x}{\mathcal{H}}^2+2\norm{y}{\mathcal{H}}^2
\end{equation*}
for $x,y\in \mathcal{H}$ Hilbert space. This explains the factor 2 in the third line which will be restored into the constant in front.}

\bnotecontent{f30}{We can check that the closedness of $S_t$ implies the closedness of $\hat{S}_t$ for each $t$ fixed, due to the fact that the wave evolution operator is bounded. The condition $\kernel{\hat{S_t}^\ast \pm \Imag \Iop{\mathscr{H}_2(a)}}=\SET{0}$ is implied by the preassumption $\kernel{S_t^\ast\pm\Imag \Iop{\mathscr{H}_2(t)}}=\SET{0}$ since $Q(t,a)$ and $Q(a,t)$ are isomorphisms.}

\bnotecontent{f31}{This becomes clear if we apply the chain and product rule for operators with respect to a scalar variable dependence: Let $u \in \dom{}{A^2}\subset \dom{}{A}$ and $A$ time dependent with derivative $\dot{A}$, then
\begin{equation*}
2A\dot{A}u=\frac{\differ }{\differ t}(A^2u)=\dot{A}Au+A\dot{A}u \quad \leftrightarrow\quad A\dot{A}u=\dot{A}Au \quad .
\end{equation*}
}

\bnotecontent{f32}{$f^\ast \mathtt{c}(E,\nabla)=\mathtt{c}(f^\ast E,f^\ast \nabla)$ for any continuous function $f$. The curvature satisfies $\Omega(f^\ast\nabla)(X,Y)f^\ast u=f^\ast\left[\Omega(f_\ast X,f_\ast Y)u\right]$ for $X,Y$ vector fields on $M'$ and $u$ as in the text.}



\bnotecontent{af1}{See \cite{nazaetal} and \cite{sipa} for more details.}

\bnotecontent{af2}{Certain wavefront set conditions have to hold for both cases in addition and for each case a transversality condition: Let 
\begin{equation*}
N^\ast\mathrm{graph}(f):=\SET{(x,\xi,y,\eta)\in T^\ast(X\times Y)\,\vert\, y=f(x) \quad\text{and}\quad \xi\in(\differ f\vert_x)^{\dagger}(\SET{\eta})}\subset T^\ast X \times T^\ast Y
\end{equation*}
denote the conormal bundle of the graph of $f$, then $\mathsf{s}^\ast(N^\ast\mathrm{graph}(f))\pitchfork T^\ast X \times \mathsf{\Lambda}_Y$ for the pullback and \\
$\mathsf{s}^\ast(N^\ast\mathrm{graph}(f))\pitchfork \mathsf{\Lambda}_X \times T^\ast Y$ for the pushforward. $\mathsf{\Lambda}_X$ and $\mathsf{\Lambda}_Y$ are Lagrangian submanifolds of the mapped Lagrangian distributions and $(\differ f\vert_p)^{\dagger}\,:\,T^\ast_{f(p)}Y\,\rightarrow\,T^\ast_p X$ is the adjoint map of the pushforward which is a priori defined for $f$ being a diffeomorphism, but used here for the preimage.}

\bnotecontent{F1}{Dirac operators and the Gauss-Bonnet operator are examples of those geometric operators.}

\bnotecontent{F2}{Calculating the $\Upgamma$-(co-)dimensions shows
\begin{equation*}
\begin{split}
\dim_\Upgamma\left(\range{P_{> 0}}\cap\left(\range{P_{\geq 0}}\right)^\perp\right)&= \dim_\Upgamma\left(\range{P_{> 0}}\cap\range{P_{< 0}}\right)=\dim_\Upgamma(\range{P_{\emptyset}})=0 \\
-\Index_\Upgamma(\overline{P}^{>0}_{\geq 0}(\tau))&=\codim_\Upgamma\left(\range{P_{> 0}}\cap\range{P_{\geq 0}} \right)= \dim_\Upgamma\left(\quotspace{\range{P_{\geq 0}}}{\range{P_{\geq 0}}\cap\range{P_{>0}}}\right)\\
&=\dim_\Upgamma\range{P_{0}}=\dim_\Upgamma\kernel{A_\tau}< \infty \quad.
\end{split}
\end{equation*} 
}

\bnotecontent{F3}{E.g. the spin-Dirac, signature- or Gauss-Bonnet operator.}

\bnotecontent{F4}{In comparison: a \textit{geometric product space} is the product space $\R\times \Sigma$ with product metric \clef{metricspinorbundle} where the hypersurface metric $\met_\Sigma$ is fixed for all $t\in \R$.}

\bnotecontent{F5}{The group $\group{SO}$ in the subscript of the frame bundle is an abbreviation for $\group{SO}(r,s-1)$.}

\bnotecontent{F6}{It is common to discern the orbits spaces for left and right actions with $\group{G}/M$ respectively $M/\group{G}$. We don't distinguish between these two designations since we have restricted ourselves to left actions.}

\bnotecontent{F7}{Hereby we mean from now on a $\Upgamma$-manifold with smooth $\Upgamma$-invariant density with a possible $\Upgamma$-vector bundle, equipped with a $\Upgamma$-invariant bundle metric.}

\bnotecontent{F8}{We also suggest the introduction of \cite{Taylor1976} for more details as well as some analysis about the types if one replaces $\Upgamma$ with a suitable locally compact, but not necessarily discrete groups.}

\bnotecontent{F9}{The term unbounded means that the operator is not necessarily bounded} 

\bnotecontent{F10}{Other authors use the terminology $\Upgamma$-equivariant operator to stress that the operator intertwines the two left translation representations.}

\bnotecontent{F11}{See e.g. \cite[Prop.4.1]{vaill}, \cite[Lem.6.5]{vaill} or \cite[Thm.6.21]{schickl2} and \cite[App.A]{benroy} for the identification.}

\bnotecontent{F12}{See \cite[Thm.B]{PMIHES1969} or one recalls \cite[Prop.1]{phillips1996}.}

\bnotecontent{F13}{For $p,q\in [1,\infty)$ we call $(p,q)$ a \textit{conjugated (number) pair} if $1=1/p+1/q$.}

\bnotecontent{F14}{This is the \textit{east-coast convention} which we are going to use in this thesis. Another common way to define a Lorentzian metric is the choice $s=n-1$ which is the so-called \textit{west-coast convention}.}

\bnotecontent{F15}{In other contexts where we need to distinguish between spacelike and timelike properties, we also use the synonyms spatial respectively temporal.}

\bnotecontent{F16}{The case of such coverings is proven in \cite{Cheeger1990ChoppingRM}.}

\bnotecontent{F17}{Later, we will clarify that the $\Upgamma$-modules needs to be free or even more projective.}

\bnotecontent{F18}{This concept of Breuer- or $L^2$-Fredholmness in a von Neumann setting will be explained in \cref{chap:galois-sec:vneumann3} where we will introduce the terminology $\Upgamma$-Fredholmness.}

\bnotecontent{F19}{The time coordinate is treated as complex variable, constraint to the real axis. The Wick rotation rotates the real time axis to a imaginary time axis. In this way, the Lorentzian spacetime can be interpreted as Euclidean spacetime with an imaginary time coordinate.}

\bnotecontent{F20}{\protect\label{tempcompfn}\textit{Temporal compactness} means that the time domain $\timef(M)$ of the globally hyperbolic manifold $M$ is a compact time interval, i.e. there exists $t_1,t_2 \in \R$ such that $\timef(M)=[t_1,t_2]$. We will also use this terminology to express that we restrict the possibly non-compact time domain of $M$ to any compact time interval $[t_1,t_2]$. Thus, any restriction $M\vert_{[t_1,t_2]}$ becomes temporal compact in the original sense.}

\bnotecontent{F21}{\protect\label{secopfn}Some authors consider the sector to be contained in the resolvent set $\uprho(A)$ which is just a change of the point of view.}

\bnotecontent{F22}{As we will only consider connections on a vector bundle, we will parallely use the term connection for covariant derivative.}

\bnotecontent{F23}{Next to orientability (i.e the vanishing of the first Stiefel-Whitney class of $TM$) also its second Stiefel-Whitney class has to vanish. 
}

\bnotecontent{F24}{Since $M$ is time- and space-oriented, its tangent bundle obeys the decomposition \clef{decomptangentpseudoriem}. If the subbundles $T^{\pm}M$ admit themselves spin structures, $TM$ has vanishing second Stiefel-Whitney class. For more details, see \cite[Sec.2.1]{baum1981spin}.}

\bnotecontent{F25}{In order to apply one of these functional calculi, one needs to make sure that the operators of interest are self-adjoint with respect to a positive definite Hermitean bilinear form on the vector bundle. Moreover, the zero eigenvalue has to lie outside the essential spectrum.}

\bnotecontent{F26}{No two points in $\Sigma$ can be connected with a timelike curve.}

\bnotecontent{F27}{To be more precise, $\inclus_\ast Y$ is a section of the pullback bundle $\inclus^\ast(TM)$ along $\Sigma$.}

\bnotecontent{F28}{We will use this convention for any other occuring space of operators, acting between sections on one and the same manifold.}

\bnotecontent{F29}{$K_1$ has to be chosen in such a way that its boundary $\bound K_1$ is totally geodesic with normal vector field $\mathfrak{n}$ such that all normal derivatives of $\curvcon(X,\mathfrak{n})\mathfrak{n}$ vanish; see \cite{Mori} for details}

\bnotecontent{F30}{$\group{Mat}(n,\C)=\group{Mat}_{n\times n}(\C)$ is the algebra of complex $(n\times n)$-matrices.}

\bnotecontent{F31}{$(r,s)=(0,n)$ implies a negative definite metric which becomes Riemannian after rescaling with $(-1)$; we only refer to $(r,s)=(n,0)$ as the Riemannian case as the other possibility differs in a global sign.}

\bnotecontent{F32}{This is still justified in the non-compact case because the Cauchy hypersurfaces are assumed to be complete.}

\bnotecontent{F33}{\protect\label{sconv}Here and henceforth we use the convention that of objects ($A,B,gf$ elements, $C$ as set) like $A_{\pm}B_{\mp}$ or $g,f_{\pm}\in C^{\pm}$ are meant for the upper and lower signs separately if not otherwise stated, i.e. $A_{\pm}B_{\mp}$ is either $A_{+}B_{-}$ or $A_{-}B_{+}$ whereas $A_{\pm}B_{\pm}$ stands for either $A_{+}B_{+}$ or $A_{-}B_{-}$. In a similar manner we write $g,f_{\pm}\in C^{\pm}$ for either $g,f_{+}\in C^{+}$ or $g,f_{-}\in C^{-}$.}

\bnotecontent{F34}{We write $Q_{\pm\pm}$ for either $Q_{++}$ or $Q_{--}$ while $Q_{\pm\mp}$ is either $Q_{+-}$ or $Q_{-+}$. We apply this to all coming quantities with similar subscripts.}

\bnotecontent{F35}{Any properly discontinuous group actions on a metric space is strongly properly discontinuous; see \cite[Thm.1]{deovara}. This holds true for ordinary manifolds which comply the properties of being paracompact and Hausdorff. This implication follows as for discrete groups properness is equivalent to strong proper discontinuity; see \cite[Cor.3.22]{TomDieckTammo1938}.}

\usepackage{standalone}				

\theoremstyle{plain}
\newtheorem{theo}{Theorem}[section]

\newtheorem{lem}[theo]{Lemma}
\newtheorem{cor}[theo]{Corollary}
\newtheorem{prop}[theo]{Proposition}
\newtheorem{defi}[theo]{Definition}
\newtheorem{rem}[theo]{Remark}

\usepackage[toc,page]{appendix}

\setcounter{tocdepth}{1}
\numberwithin{equation}{section}

\tolerance=2000 
\emergencystretch=20pt

\begin{document}

\title[$L^2$-Gamma-Fredholmness for spacetimes]
{$L^2$-Gamma-Fredholmness for spacetimes}

\author{Orville Damaschke}
\address{Universit\"at Oldenburg,
26129 Oldenburg,
Germany}
\email{orville.damaschke@math.uni-oldenburg.de}

\date{\today}

\begin{abstract}
{Let $M$ be a temporal compact globally hyperbolic manifold with Cauchy hypersurface $\Sigma$ which is a Galois covering with respect to a discrete group $\Gamma$ of automorphisms such that the quotient $\Sigma/\Gamma$ is compact without boundary. We will show Fredholmness of a (spatial) $\Gamma$-invariant Lorentzian Dirac operator under (anti) Atiyah-Patodi-Singer boundary conditions in the von Neumann sense, known as ($L^2$-)$\Gamma$-Fredholmness.}
\end{abstract}

\maketitle
\tableofcontents

\section{Introduction and statement of the main result}\label{chap:Intro}

We continue our analysis of Atiyah-Singer Dirac operators on globally hyperbolic spacetimes with non-compact Cauchy boundary which we started with a well-posedness result of the Cauchy problem for the twisted Dirac operator in \cite{OD1}. We show $\Gamma$-Fredholmness of this Dirac operator in the setting where the Cauchy hypersurface $\Sigma$ is a Galois covering with respect to a (discrete) group $\Gamma$ such that the base $\Sigma_\Gamma:=\Sigma/\Gamma$ of the covering is a closed manifold. The globally hyperbolic manifold $M$ with such a Cauchy hypersurface then becomes a Galois covering on the spatial domain which we call spatial $\Gamma$-manifold. The spinor bundle $\spinb(M)$ becomes a $\Gamma$-vector bundle in this setting. We focus on an even dimensional manifold $M$ such that the spinor bundle decomposes into spinor bundles $\spinb^{\pm}(M)$ for positive respectively negative chirality. These are again $\Gamma$-vector bundles. \\ 
\\
We assume in this situation that the Dirac operator $\Dirac^{}:C^\infty(\spinb_{}(M))\rightarrow C^\infty(\spinb_{}(M))$ commutes with left action representation of the spatial action of $\Gamma$. The chiral decomposition then implies that the Dirac operators $D^{}_{\pm}:C^\infty(\spinb^{\pm}_{}(M))\rightarrow C^\infty(\spinb^{\mp}_{}(M))$ also commute or more precisely intertwine between the left action respresentation of $\Gamma$. This allows us to consider these operators as lifts of the corresponding Dirac operators on the compact base $M_\Gamma:=[t_1,t_2]\times \Sigma/\Gamma$.\\
\\
The $L^2$-$\Gamma$-Fredholmness to show is based on the theory of von Neumann algebras in association with $\Gamma$ which come with a regularised dimension function $\dim_\Gamma$, called $\Gamma$-\textit{dimension}. We will show that the kernel and the cokernel of the Dirac operators $D^{}_{\pm}$ do have finite $\Gamma$-dimensions which imply $\Gamma$-\textit{Fredholmness}. In order to do so, we impose (anti) Atiyah-Patodi-Singer boundary conditions (APS or rather aAPS) at the upper and lower boundary hypersurface $\Sigma_1:=\SET{t_1}\times \Sigma$ and $\Sigma_2:=\SET{t_2}\times\Sigma$ of the temporal compact manifold $M$. The main result of this paper is then stated as follows. 
\begin{theo}\label{fredDpmgaAPStwist}
Let $M$ be a temporal compact, globally hyperbolic spatial $\Gamma$-manifold with compact base $M_\Gamma$ and time domain $[t_1,t_2]$, $\spinb^{\pm}_{}(M)\rightarrow M$ the $\Gamma$-spin bundles of positive respectively negative chirality.
The $\Gamma$-invariant Dirac operators 
\begin{equation*}
D^{}_{\pm,\mathrm{APS}}\,:\,FE^0_{\Gamma,\mathrm{APS}}(M,[t_1,t_2],D^{}_{\pm})\,\rightarrow L^2(\spinb^{\mp}_{}(M))
\end{equation*}
and
\begin{equation*}
D^{}_{\pm,\mathrm{aAPS}}\,:\,FE^0_{\Gamma,\mathrm{aAPS}}(M,[t_1,t_2],D^{}_{\pm})\,\rightarrow L^2(\spinb^{\mp}_{}(M))
\end{equation*}
as lifts of Dirac operators on the base manifold $M_\Gamma$ are $\Gamma$-Fredholm.
\end{theo}
The spaces $FE^0_{\Gamma,\mathrm{APS}}(M,[t_1,t_2],D^{}_{\pm})$ are modified spaces of finite energy spinors for the Dirac operator $D^{}_{+}$ respectively $D^{}_{-}$ (see \cite[Def.4.3]{OD1}) for which the local Sobolev spaces are replaced by $\Gamma$-Sobolev spaces and are equipped with APS boundary conditions; $FE^0_{\Gamma,\mathrm{aAPS}}(M,[t_1,t_2],D^{}_{\pm})$ is the corresponding space with respect to aAPS boundary conditions. $L^2(\spinb^{\pm}_{}(M))$ stands for $L^2([t_1,t_2],L^2(\spinb^{\pm}(\Sigma_\bullet)))$ where $L^2(\spinb^{\pm}_{}(\Sigma_\bullet))$ denotes square-integrable sections of $\spinb^{\pm}_{}(\Sigma_t)$ at each $t$ in the time domain with respect to a $\Gamma$-invariant density and $\Gamma$-invariant Hermitian sesquilinear bundle metric. These and the $\Gamma$-Sobolev spaces are examples of projective respectively free Hilbert $\Gamma$-modules in the von Neumann algebra of bounded $\Gamma$-invariant operators. The corresponding $\Gamma$-\textit{indices} of $D_{\pm,\mathrm{(a)APS}}$ can be related to each other. 
\begin{theo}\label{indexDpmgaAPStwist}
The $\Gamma$-indices in the situation of \Cref{fredDpmgaAPStwist} satisfy
\begin{equation}\label{indexDgaAPSformintro}
\Index_{\Gamma}(D^{}_{\pm,\mathrm{APS}})=-\Index_{\Gamma}(D^{}_{\pm,\mathrm{aAPS}})\quad.
\end{equation}
\end{theo}
These two statements have been proven as Theorem 7.6 and Theorem 7.9 in \cite{OD}. We recall the proof from \cite{OD} and fix some issues. \Cref{fredDpmgaAPStwist} is shown as \Cref{indexDgaAPStwist} and \Cref{indexDneggaAPS}. \clef{indexDgaAPSformintro} follows in the progress of the proof where we are going to relate the question of $\Gamma$-Fredholmness to the one of certain spectral parts of the corresponding Dirac-wave evolution operator, which enables a connection to the well-posedness result.\\
\\
After a brief recapitulation of the important results from \cite{OD1} in \Cref{chap:Recap}, we introduce the technical framework of (spatial) Galois coverings which leads to (spatial) $\Gamma$-manifolds, von Neumann algebras associated to $\Gamma$, Hilbert $\Gamma$-modules, and $\Gamma$-Fredholm theory in \Cref{chap:Galois}. We clarify in \Cref{chap:gammasob} how the von Neumann setting is realised for Galois coverings of our interest. As we will need it in the following proof, which is heavy influenced from the original treatment in \cite{BaerStroh}, we clarify a $\Gamma$-version of Seeleys theorem for complex powers of $\Gamma$-(pseudo-)differential operators. We collect in \Cref{chap:setting} all made assumptions which we are going to use in this and the following paper. We start the proof in \Cref{chap:wellposgamma} by specifying the results from \Cref{chap:Recap} to the $\Gamma$-setting. These well-posedness results imply the existence of $\Gamma$-invariant wave evolution operators. \Cref{chap:gammaevol} is focused on the decomposition of the wave evolution operators for $D^{}_{\pm}$ with respect to the orthogonal splitting of $L^2$-spaces, induced by the chosen boundary conditions. We investigate how unitarity, hence $\Gamma$-Fredholmness, and the Fourier integral operator character carry over to these spectral parts. We show in \Cref{chap:Atiyah} how the $\Gamma$-Fredholmness of these spectral parts implies the $\Gamma$-Fredholmness of the Dirac operators of interest. We recall some facts about manifolds of bounded geometry in \Cref{chap:manbound} which will be needed for the $\Gamma$-version of Seeley's theorem.\\
\\
We will calculate the $\Gamma$-indices in the third paper of this series where we will introduce the notion of spectral flow in $\Gamma$-setting.   

\section{Recapitulation of paper I}\label{chap:Recap}

Concentrating on the untwisted Dirac operator, we have shown in the setting of \cite{OD1} that the inhomogeneous Cauchy problem
\begin{equation}\label{Diraceq}
D_{\pm} u = f\quad\text{with}\quad u\vert_{\Sigma_t}=g
\end{equation}
at initial time $t$ in a time domain $\timef(M)$ for initial values $g\in H^s_\comp(\spinb^{\pm}_{}(\Sigma_t))$ and inhomogeneities $f\in L^2_{\loc,\scomp}(\timef(M),H^s_\loc(\spinb^{\mp}_{}(\Sigma_\bullet)))$ is well-posed such that the solutions $u$ are in $FE^s_\scomp(M,\timef,D^{}_{\pm})$. The latter two spaces have been defined in \cite[Sec.2.3/4.1]{OD1}. The well posedness results have been stated with the restriction map $\mathsf{res}_t: M \rightarrow \Sigma_t$ as follows.
\begin{theo}[cf. Theorem 4.9 in \cite{OD1}]\label{inivpwell}
For a fixed $t \in \timef(M)$ and $s \in \R$ the maps 
\begin{equation}\label{inivpmap}
\mathsf{res}_t \oplus D^{}_{\pm} \,\,:\,\, FE^s_\scomp(M,\timef,D^{}_{\pm}) \,\,\rightarrow\,\, H^s_\comp(\spinb^{\pm}_{}(\Sigma_t))\oplus L^2_{\loc,\scomp}(\timef(M),H^s_\loc(\spinb^{\mp}_{}(\Sigma_\bullet)))\, 
\end{equation}
are isomorphisms of topological vector spaces.
\end{theo}
The corresponding homogeneous Cauchy problems are also well-posed.
\begin{cor}[cf. Corollary 4.10 in \cite{OD1}]\label{homivpwell}
For a fixed $t \in \timef(M)$ and $s \in \R$ the map
\begin{equation*}
\mathsf{res}_t  \,\,:\,\, FE^s_\scomp\left(M,\timef,\kernel{D^{}_{\pm}}\right) \,\,\rightarrow\,\, H^s_\comp(\spinb^{\pm}_{}(\Sigma_t))
\end{equation*}
are isomorphisms of topological vector spaces.
\end{cor}
These results imply the existence of wave evolution operators for positive and negative chirality:
\begin{eqnarray*}
Q^{}(t_2,t_1):=\rest{t_2}\circ(\rest{t_1})^{-1} &:& H^s_\comp(\spinb^{+}_{}(\Sigma_{1}))\,\,\rightarrow\,\,H^s_\comp(\spinb^{+}_{}(\Sigma_{2}))\,\, , \\
\tilde{Q}^{}(t_2,t_1):=\rest{t_2}\circ(\rest{t_1})^{-1}&:&H^s_\comp(\spinb^{-}_{}(\Sigma_{1}))\,\,\rightarrow\,\,H^s_\comp(\spinb^{-}_{}(\Sigma_{2}))\,\, . 
\end{eqnarray*}
These are isomorphisms of topological vector spaces which satisfy the properties of ordinary evolution operators. They are continuous in both time arguments and furthermore unitary for $s=0$ if one considers any fixed bounded time interval (see \cite[Lem.5.2]{OD1}). The most important property is their character as properly supported Fourier integral operator of order zero with canonical relation $\mathsf{C}_{1\rightarrow 2}$, given as graph of a canonical transformation, induced from the lightlike (co-)geodesic flow from $\Sigma_1$ to $\Sigma_2$.
\begin{theo}[cf. Theorem 5.4 in \cite{OD1}]\label{Qfourier} 
For all $s \in \R$ the operators $Q$ and $\tilde{Q}$ satisfy
\begin{eqnarray*}
&&Q(t_2,t_1)\in \FIO{0}_{\mathsf{prop}}(\Sigma_{1},\Sigma_{2};\mathsf{C}'_{1\rightarrow 2};\Hom(\spinb^{+}(\Sigma_1),\spinb^{+}(\Sigma_2)))\\
&& \tilde{Q}(t_2,t_1)\in \FIO{0}_{\mathsf{prop}}(\Sigma_{1},\Sigma_{2};\mathsf{C}'_{1\rightarrow 2};\Hom(\spinb^{-}(\Sigma_1),\spinb^{-}(\Sigma_2)))
\end{eqnarray*}
for any fixed time interval $[t_1,t_2]\subset \timef(M)$ and with canonical graphs 
\begin{equation}\label{Qcanrel}
\begin{split}
\mathsf{C}_{1\rightarrow 2} &=\mathsf{C}(\inclus^\ast)\circ\mathcal{C}=\mathsf{C}_{1\rightarrow 2\vert +}\sqcup \mathsf{C}_{1\rightarrow 2\vert -} \quad\text{where}\\
\mathsf{C}_{1\rightarrow 2\vert \pm} &=\SET{((x_{\pm},\xi_{\pm}),(y,\eta))\in \dot{T}^\ast \Sigma_2\times \dot{T}^\ast \Sigma_1\,\Big\vert\, (x_{\pm},\xi_{\pm})\sim (y,\eta)}  
\end{split}
\end{equation} 
with respect to the lightlike (co-)geodesic flow as canonical relation; their principal symbols are
\begin{equation}
{\sigma}_{0}(Q)(x,\xi_{\pm};y,\eta)=\pm\frac{\norm{\eta}{\met_{t_1}(y)}^{-1}}{2}\left(\mp\norm{\xi_{\pm}}{\met_{t_2}(x)}\upbeta+\Cliff{t_2}{\xi^{\sharp}_{\pm}}\right)\circ \mathpzc{P}^{\spinb(M)}_{(x,\varsigma_{\pm})\leftarrow (y,\zeta_{\pm})} \circ \upbeta\label{Qprincsymbol}
\end{equation}
and
\begin{equation}
{\sigma}_{0}(\tilde{Q})(x,\xi_{\pm};y,\eta)=\pm\frac{\norm{\eta}{\met_{t_1}(y)}^{-1}}{2}\left(\mp\norm{\xi_{\pm}}{\met_{t_2}(x)}\upbeta-\Cliff{t_2}{\xi^{\sharp}_{\pm}}\right)\circ \mathpzc{P}^{\spinb(M)}_{(x,\varsigma_{\pm})\leftarrow (y,\zeta_{\pm})} \circ \upbeta \label{Qnegprincsymbol}
\end{equation}
where $(y,\zeta_{\pm})\in T^\ast_{\Sigma_{1}}M$ and $(x,\varsigma_{\pm})\in T^\ast_{\Sigma_{2}}M$ restrict to $(y,\eta)$ and $(x,\xi_{\pm})$ respectively.
\end{theo}
$\mathpzc{P}^{\spinb(M)}_{(x,\varsigma_{\pm})\leftarrow (y,\zeta_{\pm})}$ denotes the parallel transport from $(y,\zeta_{\pm})$ to $(x,\varsigma_{\pm})$ with respect to the spinorial covariant derivative, $\sharp$ the sharp-isomorphism and $\norm{\bullet}{\met_{t}}$ a norm for covectors, induced by the dual metric of $\met_t$ for fixed $t$, and $T^\ast_{\Sigma_{t}}M:=T^\ast_pM$ for $p \in \Sigma_t$. $\FIO{0}_{\mathsf{prop}}$ denotes the space of properly supported Fourier integral operators of order $0$. We refer to Section 5 and Appendix A of \cite{OD1} for details and notations. The properly supportness allows to extend these operators as maps between $H^s_\loc$-spaces for any $s\in \R$.

\section{Geometry and functional analysis in $\Gamma$-setting}\label{chap:Galois}

This section relies heavily on \cite[Chap.2/3]{shub} which is supported with additional material from \cite{vaill} as well as \cite{schickl2}. We recall the basic important concepts of functional algebra and analysis on Galois coverings with respect to a group $\Gamma$. The first part deals with topological properties of Galois coverings with a special view on pseudo-Riemannian manifolds. The main objects are \textit{von Neumann algebras} and Hilbert modules with respect to $\Gamma$. We introduce the basic background of these concepts and how Fredholm theory is implemented.

\subsection{Galois Coverings and $\Gamma$-manifolds}\label{chap:galois-sec:manifold}

Let $M$ be any connected manifold. We denote with $\Gamma$ a discrete group which acts freely, freely discontinuous and cocompactly on a non-compact manifold $M$ by diffeomorphisms; we denote with $\upepsilon$ the neutral element. The quotient map $M\rightarrow M/\Gamma$ then becomes a Galois covering with compact base manifold $M/\Gamma$ and $\Gamma$ acts as deck transformation. We will also use the term $\Gamma$-\textit{manifold} and write $M_\Gamma$ for $M/\Gamma$. Examples of such manifolds $M$ are universal coverings of compact manifolds.\\
\\
If the $\Gamma$-manifold $M$ is equipped with a smooth and complex vector bundle $E$ with projection $\pi\,:\,E\rightarrow\,M$, one can define an action of $\Gamma$ on the vector bundle $E$ through the map $\pi_\Gamma\vert_p\,:\,E_p\,\rightarrow\,E_{\gamma p}$ for all $p \in M, \gamma \in \Gamma$. We presuppose that this map is a linear isomorphism, such that the action on $E$ covers the action on $M$, and that the projection $\pi$ is $\Gamma$-\textit{equivariant}, i.e. $\pi(\gamma p,v_{\gamma p})=\gamma \pi(p,v_p)$ for all $p,\gamma$, and $v_p$ a vector at $p$. Like in our main reference \cite{shub}, we call $E$ a $\Gamma$-\textit{vector bundle} if it is a vector bundle over a $\Gamma$-manifold such that $E$ is the pullback bundle of a vector bundle $E_\Gamma:=E/\Gamma$ over the compact base $M_\Gamma$ with respect to the covering map. We can furthermore extend $\pi_\Gamma\vert_p$ to an isometry by introducing an inner product $\idscal{1}{E_{p}}{\cdot}{\cdot}$ on each fibre $E_p$ for $p \in M$ which is either already given on $E$ or is induced as pullback of a bundle metric on $E/\Gamma$. In order to extend $\pi_\Gamma$ to an isometry, we have to require that the bundle metric is $\Gamma$-invariant:
\begin{equation}\label{gammabundlemetinv}
\idscal{1}{E_{\gamma p}}{\cdot}{\cdot}=\idscal{1}{E_{p}}{\cdot}{\cdot}\quad \forall\,\gamma\in\Gamma\,,\, p\in M\quad.
\end{equation} 
We also equip $M$ with a positive, $\Gamma$-invariant smooth density $\differ \mu$:
\begin{equation}\label{gammadenseinv}
\differ \mu (\gamma p)=\differ \mu(p) \quad\forall\,\gamma\in\Gamma\,,\, p\in M\quad.
\end{equation} 
This can be either directly obtained by lifting a positive smooth density from $M_\Gamma$ or with a metric on $T(M_\Gamma)$ which induces a $\Gamma$-invariant metric on $TM$ and thus a $\Gamma$-invariant smooth density on $M$.\\
\\
A fundamental domain $\mathcal{F}$ is an open subset in $M$ such that 
\begin{align*}
\text{(a)} & \quad(\gamma_1\mathcal{F})\cap(\gamma_2\mathcal{F})=\emptyset\quad\forall\,\gamma_1,\gamma_2\in \Gamma \quad\Rightarrow\quad \gamma_1\neq\gamma_2& \\
\text{(b)} & \quad M=\bigcup_{\gamma \in \Gamma}\gamma\left(\overline{\mathcal{F}}\right) &\text{(c)} \quad \overline{\mathcal{F}}\setminus\mathcal{F}\quad\text{is a null set.}
\end{align*}
Moreover, these properties imply that $M$ can be viewed as principal $\Gamma$-bundle, i.e. $M \simeq \Gamma \times \mathcal{F}$, and $\mathcal{F}$ as representative for the group action.\\
\\
The compactness of the orbit space allows to define a $\Gamma$-invariant locally finite open covering and a $\Gamma$-invariant partition of union subordinated to this open covering. Since $M_\Gamma$ is compact, one can take a finite covering of the base by open balls $\mathring{\mathbb{B}}_j$, $j \in J$ finite index set. These open balls can be lifted to $M$ which defines an infinite, but countable covering for $M$ due to the countability of $\Gamma$:
\begin{equation}\label{gammainvcover}
M=\bigcup_{\substack{j \in J \\ \gamma \in \Gamma}}\gamma(\mathring{\mathbb{B}}_j) \quad .
\end{equation}
A subordinated partition of unity $(\gamma \mathring{\mathbb{B}}_j, \phi_{j,\gamma})_{\substack{j \in J\\\gamma \in \Gamma}}$ can be constructed by taking compactly supported functions $\phi_{j,\gamma}\in C^\infty_\comp(\gamma \mathring{\mathbb{B}}_j,\R_{\geq 0})$ which satisfy $\phi_{j,\gamma}(p)=\phi_{j,\upepsilon}(\gamma^{-1}p)$ as well as
\begin{equation}\label{gammainvpartition}
\sum_{\substack{j \in J \\\gamma \in \Gamma}} \phi_{j,\gamma}=1 \quad .
\end{equation}
\begin{rem}\label{remfun}
\begin{itemize}
\item[]
\item[(i)]If we compare property (b) in the characterisation of the fundamental domain with \clef{gammainvcover}, we observe that the closure $\overline{\mathcal{F}}$ is the compact orbit space $M_\Gamma$. Thus, $\mathcal{F}$ is a dense subspace in $M_\Gamma$.
\item[(ii)] By comparing \clef{gammainvcover} with \clef{boundgeomcov}, we observe that $\Gamma$-manifolds are indeed examples of manifolds of bounded geometry. The radii-condition on the functions for the partition of unity in \Cref{manboundgeopartun} is replaced with the assumption on $\Gamma$ to act freely discontinuous.
\end{itemize}
\end{rem}

We have to require for pseudo-Riemannian manifolds that the covering map is a local isometry in order to obtain a pseudo-Riemannian metric on the covering via pullback with the same signature. If $M$ is simply connected, time- and space orientability are also preserved under pullback. Hence we don't need to impose further modifications than the assumption that the covering maps are local isometries. We refer to \cite[Sec.2.3]{baum1981spin} and \cite[p.191]{ONeill} for more informations. The preservation of the global hyperbolicity property under covering maps has been proven in \cite[Lem.4.1.]{garhar} if one assumes that the base of the covering is globally hyperbolic. Thus, the covering of Cauchy hypersurfaces are again Cauchy hypersurfaces.\\
\\
We are going to focus on the following situation: let $\Sigma_\Gamma$ be a spatial Cauchy hypersurface, such that $M_\Gamma=\timef\times\Sigma_\Gamma$ is globally hyperbolic with time domain $\timef$, and suppose that $\Sigma$ is its Galois covering with respect to $\Gamma$. We denote with $M=\timef\times \Sigma$ the covering in spatial direction which is again a globally hyperbolic manifold. We call $M$ \textit{spatial $\Gamma$-manifold} as $\Gamma$ only acts on the spatial part of the spacetime.

\subsection{Von Neumann algebra associated to a $\Gamma$-action}\label{chap:galois-sec:vneumann}

Let $\mathcal{H}$, $\mathcal{H}_1$, $\mathcal{H}_2$ be Hilbert spaces and denote with $\mathscr{B}(\mathcal{H}_1,\mathcal{H}_2)$  the set of bounded linear operators $A:\mathcal{H}_1\rightarrow \mathcal{H}_2$; we write $\mathscr{B}(\mathcal{H})$ for $\mathcal{H}_{1}=\mathcal{H}=\mathcal{H}_2$. The space $\mathscr{B}(\mathcal{H})$ with the operator norm is the most prominent example of a unital $C^\ast$-algebra with identity $\Iop{\mathcal{H}}$, composition as multiplication and involution is given by adjoining the operator. The \textit{commutant} of a subspace $\mathscr{N}\subset \mathscr{B}(\mathcal{H})$ is 
\begin{equation}\label{commutant}
\mathscr{N}':=\SET{A\in \mathscr{B}(\mathcal{H})\,\vert\, AB=BA\quad\forall\,B\in \mathscr{N}} \quad.
\end{equation}
It is a closed and unital subalgebra of $\mathscr{B}(\mathcal{H})$. A subalgebra $\mathscr{N}\subset\mathscr{B}(\mathcal{H})$ with $\Iop{\mathcal{H}}\in \mathscr{N}$ and $A^\ast \in \mathscr{N}$ for all $A\in \mathscr{N}$ is a \textit{von Neumann algebra} if $(\mathscr{N}')'=\mathscr{N}$. 
A von Neumann algebra $\mathscr{N}$ is a \textit{factor} if
$$ \mathscr{N}\cap\mathscr{N}'=\SET{B\in \mathscr{N}\,\vert\, AB=BA\quad\forall\, A\in \mathscr{N}}=\SET{\lambda \Iop{\mathcal{H}}\,\vert\, \lambda \in \C}\quad. $$
We now want to introduce a von Neumann algebra which we will use in the $\Gamma$-setting. Let $u,v$ be two functions $\Gamma\rightarrow \C$; we define an inner product
\begin{equation}\label{l2gammainnerprod}
\dscal{1}{\ell^2(\Gamma)}{u}{v}:=\sum_{g\in\Gamma}\overline{u(g)}v(g)  
\end{equation}
and a corresponding Hilbert space of $l^2$-functions 
\begin{equation}\label{l2gamma}
\ell^2(\Gamma):=\SET{u\,:\,\Gamma\,\rightarrow\,\C\,\Big\vert\,\dscal{1}{\ell^2(\Gamma)}{u}{u} <\infty} \quad. 
\end{equation}
An orthonormal basis in $\ell^2(\Gamma)$ is $\SET{\delta_\gamma}_{\gamma \in \Gamma}$ which satisfies $\delta_{\gamma_1}(\gamma_2)=1$ if $\gamma_1=\gamma_2$ and zero otherwise. The group action $\Gamma$ induces \textit{left and right translation operators} on $\ell^2(\Gamma)$ which acts on this basis as
\begin{equation}\label{leftrightonONB}
l_{\gamma_1}\delta_{\gamma_2}=\delta_{\gamma_1\gamma_2}\quad \text{and}\quad r_{\gamma_1}\delta_{\gamma_2}=\delta_{\gamma_2\gamma_1^{-1}} 
\end{equation}
for $\gamma_1,\gamma_2 \in \Gamma$. These are unitary operators such that $\gamma\,\mapsto\,l_\gamma$ and $\gamma\,\mapsto\,r_\gamma$ become unitary representations of the discrete group in \clef{l2gamma}. We denote with $\mathscr{N}_r(\Gamma)$ the smallest von Neumann algebra in $\mathscr{B}(\ell^2(\Gamma))$ which contains the set $\SET{r_\gamma\,\vert\,\gamma \in \Gamma}$ and $\mathscr{N}_l(\Gamma)$ as the smallest von Neumann algebra which contains $\SET{l_\gamma\,\vert\,\gamma \in \Gamma}$. It is proven in \cite[Thm.2.10]{shub} that the von Neumann algebras $\mathscr{N}_l(\Gamma)$ and $\mathscr{N}_r(\Gamma)$ take the form 
\begin{equation}\label{vNalr}
\begin{split}
\mathscr{N}_l(\Gamma) &=\SET{A\in \mathscr{B}(\ell^2(\Gamma))\,\vert\, Ar_\gamma=r_\gamma A\quad\forall\,a,b,\Gamma \in \Gamma} \\
\mathscr{N}_r(\Gamma) &=\SET{A\in \mathscr{B}(\ell^2(\Gamma))\,\vert\, Al_\gamma=l_\gamma A\quad\forall\,a,b,\Gamma \in \Gamma} \quad,
\end{split}
\end{equation}
hence $(\mathscr{N}_l(\Gamma))'=\mathscr{N}_r(\Gamma)$ respectively $(\mathscr{N}_r(\Gamma))'=\mathscr{N}_l(\Gamma)$, and that these are factors if and only if all conjugacy classes in $\Gamma/\SET{\upepsilon}$ are infinite. Those groups $\Gamma$, for which this is true, are called \textit{i.c.c-groups} (infinite-conjugacy-class). One can define for an operator $A$ in either $\mathscr{N}_l(\Gamma)$ or $\mathscr{N}_l(\Gamma)$ a \textit{trace} via
\begin{equation}\label{tracevN}
\uptau_\Gamma(A):=\dscal{1}{\ell^2(\Gamma)}{A \delta_{\gamma}}{\delta_{\gamma}}=\dscal{1}{\ell^2(\Gamma)}{A \delta_{\upepsilon}}{\delta_{\upepsilon}} 
\end{equation}
for $\gamma \in \Gamma$; the definition is independent of the choice of $\Gamma$ which explains the second equivalence.
This trace shares the same properties on $\mathscr{N}_{r,l}(\Gamma)$ with the trace $\Tr{}{\cdot}$ on $\mathscr{B}(\mathcal{H})$. 
This trace gives rise to a dimension function $\dim_{\uptau}$:
$$ \dim_\uptau(V):=\uptau_\Gamma(P_V) \quad $$
for any $V\subset \ell^2(\Gamma)$ and $P_V:\ell^2(\Gamma)\rightarrow V$ as orthogonal projection. It is shown in \cite[Prop.6.7.4]{kadrich} that $\dim_\tau(\mathscr{N}_{r,l}(\Gamma))\in [0,1]$ if $\Gamma$ is an i.c.c.-group, i.e. $\mathscr{N}_r(\Gamma)$ and $\mathscr{N}_l(\Gamma)$ are factors.\\
\\
\clef{vNalr} can be extended to bounded operators on the tensor product $\ell^2(\Gamma)\otimes\mathcal{H}$. The resulting tensor products of algebras $\mathscr{N}_{r,l}(\Gamma)\otimes \mathscr{B}(\mathcal{H})$ are bounded operators on the tensor product Hilbert space which commute with any right respectively left translation operation. $\mathscr{N}_{r,l}(\Gamma)\otimes \mathscr{B}(\mathcal{H})$ thus also become von Neumann algebras. If $\mathscr{N}_{l}(\Gamma)$ and $\mathscr{N}_{r}(\Gamma)$ are factors, these tensor products are also factors.
The associated $\Gamma$-\textit{trace} is 
\begin{equation}\label{tracevNtensor}
\Tr{\Gamma}{A}:=(\uptau_\Gamma \otimes \Tr{}{A})\quad.
\end{equation}
$l^2(\Gamma)$ as well as $\ell^2(\Gamma)\otimes\mathcal{H}$ are special cases of a \textit{(general) Hilbert $\Gamma$-module}, i.e. Hilbert spaces with a unitary \textit{left action representation} $L_\Gamma$ of $\Gamma$. The following classifications are important for the upcoming analysis.
\begin{defi}\label{defihilbertmodules}
Let $\mathscr{H}$ be a general Hilbert $\Gamma$-module and $\mathcal{H}$ a Hilbert space;
\begin{itemize}
\item[(a)] $\mathscr{H}$ is called \textit{free} if it is unitarily isomorphic to a $\Gamma$-module $\ell^2(\Gamma)\otimes \mathcal{H}$ and the representation of $\Gamma$ is given by $\gamma \,\mapsto\,l_\gamma\otimes\Iop{\mathcal{H}}$ for any $\gamma \in \Gamma$.
\item[(b)] $\mathscr{H}$ is called \textit{projective} if it is unitarily isomorphic to a closed $\Gamma$-invariant subspace in $\ell^2(\Gamma)\otimes \mathcal{H}$.
\end{itemize}
\end{defi}
The unitary isomorphisms are understood as unitary maps, commuting with the action of $\Gamma$. We collect some facts about these Hilbert $\Gamma$-modules.
\vspace{0.25cm}
\begin{prop}[cf. Proposition 2.16, Corollary 2.16.1, Corollary 2.16.2 in \cite{shub}]~\label{prophilbertgammamodules}
\vspace*{-1em}
\begin{itemize}
\item[(1)] Let $V_1,V_2$ be general Hilbert $\Gamma$-modules and $A:V_1 \rightarrow V_2$ a linear topological $\Gamma$-isomorphism of $\Gamma$-modules, then there exists a unitary $\Gamma$-isomorphism $U:V_1 \rightarrow V_2$ of Hilbert $\Gamma$-modules.
\item[(2)] If $V$ is a general Hilbert $\Gamma$-module and there exists a topological $\Gamma$-isomorphism onto a free Hilbert $\Gamma$-module or a closed $\Gamma$-submodule of a free Hilbert $\Gamma$-module, then $V$ is a free respectively projective Hilbert $\Gamma$-module.
\end{itemize}
\end{prop}
The notion \textit{topological} $\Gamma$-\textit{isomorphism} means that the topological isomorphism commutes with the left action representation. Both properties (1) and (2) indicate that it is enough to consider topological isomorphisms in \Cref{defihilbertmodules}. 
Certain subspaces carry the structure of a Hilbert $\Gamma$-module which we will show for later purposes.
\begin{lem}\label{helpinglemma1}
Let $\mathscr{H},\mathscr{H}_1,\mathscr{H}_2$ be projective respectively free Hilbert $\Gamma$-modules; the spaces $\mathscr{H}_1\oplus \mathscr{H}_2$, $W^\perp$ and $\mathscr{H}/{W}$ for any closed, $\Gamma$-invariant subspace $W\subset \mathscr{H}$ are projective respectively Hilbert $\Gamma$-(sub)modules.
\end{lem}
\begin{proof}
We consider the projective case as it covers the case of free Hilbert $\Gamma$-modules. Let $\gamma$ be any element in $\Gamma$. One needs to check that these spaces, well-known to be Hilbert spaces, are themselves $\Gamma$-invariant Hilbert spaces.\\
\\
Let $\SET{L_\gamma \vert\,\gamma \in \Gamma}$ denote the left action representation on $\mathscr{H}_1$ and $\SET{\mathcal{L}_\gamma \vert\,\gamma \in \Gamma}$ the left action representation on $\mathscr{H}_2$. The left action representation on the direct sum is the direct sum of the left action representation, i.e. $\Gamma$ induces a diagonal action on both Hilbert $\Gamma$-modules. As the ranges of the left action representations are contained in their belonging Hilbert $\Gamma$-module, we have
\begin{equation*}
(L_{\Gamma}\oplus\mathcal{L}_\Gamma)(\mathscr{H}_1\oplus\mathscr{H}_2)=L_{\Gamma}(\mathscr{H}_1)\oplus \mathcal{L}_{\Gamma}(\mathscr{H}_2)\subseteq \mathscr{H}_1\oplus\mathscr{H}_2 
\end{equation*}
for any $\gamma\in \Gamma$. Hence the direct sum is $\Gamma$-invariant and becomes a general Hilbert $\Gamma$-module. The orthogonal complement is also a $\Gamma$-invariant subpace: let $\SET{L_\gamma \vert\,\gamma \in \Gamma}$ now denote the left action representation on $\mathscr{H}$; suppose $v \in W^\perp$, i.e. $v\in \mathscr{H}$ such that $\dscal{1}{\mathscr{H}}{v}{u}=0$ for all $u \in W$. Then $L_\gamma v \in W^\perp$ since the action of $\Gamma$ is unitary and thus $L_\gamma(W^\perp)\subseteq W^\perp$, implying $W^\perp$ to be a general Hilbert-$\Gamma$-submodule. The quotient Hilbert space consists of equivalence classes for each element in $\mathscr{H}$ where two Hilbert vectors are equivalent if the difference is an element in $W$. We first check that this equivalence relation is also true for transformed Hilbert vectors: as $\mathscr{H}$ is a Hilbert $\Gamma$-module, the elements $L_\gamma v_1, L_\gamma v_2$ are in $L_\gamma(\mathscr{H})\subseteq \mathscr{H}$ for $v_1,v_2 \in \mathscr{H}$ and $L_\gamma(v_1-v_2)\in L_\gamma(W)\subseteq W$ since $W$ is a $\Gamma$-submodule. As the group action is linear, we have $L_\gamma (v_1-v_2) \in W$. Thus, each equivalence class is $\Gamma$-invariant: $L_\gamma\left(\mathscr{H}/{W}\right)\subseteq \mathscr{H}/{W}$. \\
\\
It is left to check that all these Hilbert $\Gamma$-modules are projective. Given some Hilbert spaces $\mathcal{H},\mathcal{H}_1$ and $\mathcal{H}_2$ such that $\mathscr{H},\mathscr{H}_1$ and $\mathscr{H}_2$ are 
unitarily related to closed submodules of $\ell^2(\Gamma)\otimes\mathcal{H}, \ell^2(\Gamma)\otimes\mathcal{H}_1$ and respectively $\ell^2(\Gamma)\otimes\mathcal{H}_2$. As these unitary isomorphisms are closed and commute with the $\Gamma$-action, we only need to check that the subspaces in the claim are isomorphic to a closed subset in $\ell^2(\Gamma)\otimes(\mathcal{H}_1\oplus\mathcal{H}_2)$ or rather $\ell^2(\Gamma)\otimes\mathcal{H}$.\\
\\
The direct sum of the two unitary $\Gamma$-isomorphisms implies a unitary $\Gamma$-isomorphism on the direct sum of the closed $\Gamma$-submodules which is again a closed $\Gamma$-submodule of $\ell^2(\Gamma)\otimes(\mathcal{H}_1\oplus\mathcal{H}_2)$. Projectivity of $W^\perp$ follows from restricting the unitary $\Gamma$-isomorphism on $W^\perp$ which again maps to a closed subspace in $\ell^2(\Gamma)\otimes \mathcal{H}$. The commuting of the $\Gamma$-isomorphisms with the group action implies $\Gamma$-invariance of this closed subspace such that $W^\perp$ becomes a projective Hilbert $\Gamma$-submodule. Since $W^\perp$ and $\mathscr{H}/{W}$ are $\Gamma$-invariant, the known isomorphy $W^\perp \cong \mathscr{H}/{W}$ due to closedness of $W$ is $\Gamma$-invariant and induces a unitary $\Gamma$-isomorphism according to \Cref{prophilbertgammamodules} (1). As $W^\perp$ is a projective Hilbert $\Gamma$-module, $\mathscr{H}/{W}$ becomes projective due to \Cref{prophilbertgammamodules} (2).\qedhere
\end{proof}

\noindent We consider the space of bounded operators on a general Hilbert $\Gamma$-module $\mathscr{H}$ which commute with the left action representation $L_\gamma$ for all $\gamma \in \Gamma$:
\begin{equation}\label{vNcomm}
\mathscr{B}_\Gamma(\mathscr{H}):=\SET{A \in \mathscr{B}(\mathscr{H})\,\vert\,AL_\gamma=L_\gamma A \quad \forall\, \gamma \in \Gamma}\quad.
\end{equation} 
This is again a von Neumann algebra as $\mathscr{B}_\Gamma(\mathscr{H})$ is the commutant of the space $\SET{L_\gamma\,\vert\,\gamma \in \Gamma}$. If we replace $\mathscr{H}$ with the tensor product $\ell^2(\Gamma)\otimes \mathcal{H}$, we gain
\begin{equation}\label{vNcommspecI}
\mathscr{B}_\Gamma(\ell^2(\Gamma)\otimes \mathcal{H})=\SET{A \in \mathscr{B}(\ell^2(\Gamma)\otimes \mathcal{H})\,\vert\,AL_\gamma=L_\gamma A \quad \forall\, \gamma \in \Gamma}
\end{equation}
and it coincides with $\mathscr{N}_r(\Gamma)\otimes \mathscr{B}(\mathcal{H})$. Thus, if $\mathscr{H}$ is in fact a free Hilbert $\Gamma$-module which is unitarily isomorphic to $\ell^2(\Gamma)\otimes \mathcal{H}$ with Hilbert space $\mathcal{H}$, then this unitary $\Gamma$-isomorphism induces a unitary $\Gamma$-isomorphism between \clef{vNcomm} and \clef{vNcommspecI}:
\begin{equation}\label{vNcommfree}
\mathscr{B}_\Gamma(\mathscr{H})\cong \mathscr{N}_r(\Gamma)\otimes \mathscr{B}(\mathcal{H}) \quad.
\end{equation} 
Let $V$ be a Hilbert $\Gamma$-module which is unitarily isomorphic to a closed $\Gamma$-invariant subspace $W \subset \ell^2(\Gamma)\otimes \mathcal{H}$. The orthogonal projection $P_W$ commutes with $l_\gamma\otimes \Iop{\mathcal{H}}$ for every $\gamma \in \Gamma$ such that $P_W \in \mathscr{B}_\Gamma(\ell^2(\Gamma)\otimes \mathcal{H})$ and thus $P_V \in \mathscr{B}_\Gamma(\mathscr{H})$. The trace on the von Neumann algebra \clef{vNcommfree} is given by \clef{tracevNtensor}. From \cite[Thm.2.16]{shub} we recall that the trace does not depend on the inclusion of $V$ in $\ell^2(\Gamma)\otimes \mathcal{H}$ via $W$ such that the $\Gamma$-\textit{dimension}
\begin{equation}\label{gammadimdefi}
\dim_\Gamma(V):=\Tr{\Gamma}{P_V}=\Tr{\Gamma}{P_W} 
\end{equation}
is well-defined and has the following properties.
\begin{prop}\label{propgammadim} Let  $\mathscr{H}$,  $\mathscr{H}_1$ and $\mathscr{H}_2$ be (free or projective) Hilbert $\Gamma$-modules; the following properties are satisfied by \clef{gammadimdefi}:
\begin{itemize}
\item[(1)] $\dim_\Gamma(\mathscr{H})$ is independent of the inclusion $\mathscr{H}\subset \ell^2(\Gamma)\otimes\mathcal{H}$;
\item[(2)] $\dim_\Gamma(\ell^2(\Gamma))=1$, $\dim_\Gamma(\SET{0})=0$ and $\dim_\Gamma(\ell^2(\Gamma)\otimes\mathcal{H})=\dim(\mathcal{H})$;
\item[(3)] $\dim_\Gamma(\mathscr{H}_1\oplus\mathscr{H}_2)=\dim_\Gamma(\mathscr{H}_1)+\dim_\Gamma(\mathscr{H}_2)$ for two orthogonal $\Gamma$-modules $\mathscr{H}_1,\mathscr{H}_2$;
\item[(4)] $\mathscr{H}_1\subset\mathscr{H}_2\,\Rightarrow\,\dim_\Gamma(\mathscr{H}_1)\leq\dim_\Gamma(\mathscr{H}_2)$ and $\dim_\Gamma(\mathscr{H}_1)=\dim_\Gamma(\mathscr{H}_2)$ if and only if $\mathscr{H}_1=\mathscr{H}_2$;
\item[(5)] $\dim_\Gamma(\mathscr{H}_1)=\dim_\Gamma(\mathscr{H}_2)\,\Leftrightarrow$ $\mathscr{H}_1$ and $\mathscr{H}_2$ are unitarily $\Gamma$-isomorphic to each other. 
\item[(6)] If $\Gamma$ is finite with cardinality $\absval{\Gamma}$, then $\absval{\Gamma}\dim_\Gamma$ coincides with $\dim$; if $\Gamma=\SET{\upepsilon}$, \\ then $\dim_\Gamma$ coincides with $\dim$. 
\end{itemize}
\end{prop}
The von Neumann algebra \clef{vNcomm} motivates to define bounded operators between two different Hilbert $\Gamma$-modules. We designate with $\mathscr{L}(\mathscr{H}_1,\mathscr{H}_2)$ the space of linear operators from a domain $\dom{}{A}\subset \mathscr{H}_1$ with range in the codomain $\mathscr{H}_2$.

\begin{defi}\label{defigammaopmorph}
Given two Hilbert $\Gamma$-modules $\mathscr{H}_1$, $\mathscr{H}_2$ with left action representations $\SET{L_\gamma \vert\,\gamma \in \Gamma}$ respectively $\SET{\mathcal{L}_\gamma \vert\,\gamma \in \Gamma}$ and $A\in \mathscr{L}(\mathscr{H}_1,\mathscr{H}_2)$, then we call $A$ a
\begin{itemize}
\item[(a)] $\Gamma$-\textit{operator}\bnote{F10} if $AL_\gamma=\mathcal{L}_\gamma A$ for all $\gamma \in \Gamma$;
\item[(b)] $\Gamma$-\textit{morphism} if $A$ is an $\Gamma$-operator and $A \in \mathscr{B}(\mathscr{H}_1,\mathscr{H}_2)$.
\end{itemize}
\end{defi}
We will denote the space of $\Gamma$-morphisms with $\mathscr{B}_\Gamma(\mathscr{H}_1,\mathscr{H}_2)$. We designate with $\mathscr{L}_\Gamma(\mathscr{H}_1,\mathscr{H}_2)$ the set of linear, $\Gamma$-invariant operators from $\dom{}{A}\subset\mathscr{H}_1$ to $\mathscr{H}_2$ with graph $\Graph{A}$. Murray and von Neumann have shown in \cite[Theorem XV]{murvN} that the subset of closed and densely defined operators is a $\ast$-algbra which contains a von Neumann algebra.
Coming back to our $\Gamma$-setting, we define 
\begin{equation*}
\mathscr{C}_\Gamma(\mathscr{H}):=\SET{A \in \mathscr{L}_\Gamma(\mathscr{H})\,\vert\,A\,\text{closed and densely defined}} 
\end{equation*}
which contains the von Neumann algebra $\mathscr{B}_\Gamma(\mathscr{H})$. We define moreover the set
\begin{equation*}
\mathscr{C}_\Gamma(\mathscr{H}_1,\mathscr{H}_2):=\SET{A \in \mathscr{L}_\Gamma(\mathscr{H}_1,\mathscr{H}_2)\,\vert\,\dom{}{A}\subset \mathscr{H}_1\,\text{dense}\,,\,\Graph{A}\,\text{closed}}.
\end{equation*}
The intertwining of the left action representation shows in addition that the domain and the graph of closed and densely defined operators are in fact projective Hilbert $\Gamma$-modules. 
\begin{lem}[Lemma 3.8.4 in \cite{shub}]\label{unboundeddomhilbertmod}
Let $\mathscr{H}_1,\mathscr{H}_2$ be Hilbert $\Gamma$-modules and $A\in \mathscr{C}(\mathscr{H}_1,\mathscr{H}_2)$ then
\begin{itemize}
\item[(1)] $A\in \mathscr{C}_\Gamma(\mathscr{H}_1,\mathscr{H}_2)$ if and only if $\Graph{A}$ is $\Gamma$-invariant with respect to the diagonal action of $\Gamma$;
\item[(2)] If $A\in \mathscr{C}_\Gamma(\mathscr{H}_1,\mathscr{H}_2)$, the domain $(\dom{\mathscr{H}_1}{A},\norm{\bullet}{\Graph{A}})$ becomes a projective Hilbert $\Gamma$-module. 
\end{itemize}  
\end{lem}
Certain subspaces with respect to $\Gamma$-operators do carry a Hilbert $\Gamma$-module structure which we will show for later purposes.
\begin{lem}\label{helpinglemma2}
Let $\mathscr{H}_1,\mathscr{H}_2$ be Hilbert $\Gamma$-modules and $A \in \mathscr{C}_\Gamma(\mathscr{H}_1,\mathscr{H}_2)$, then
\begin{itemize}
\item[(1)]
$\kernel{A}$ is a projective Hilbert $\Gamma$-submodule of $\mathscr{H}_1$. 
\item[(2)]
If $\range{A}$ is closed, then it is a projective Hilbert $\Gamma$-submodule of $\mathscr{H}_2$. 
\end{itemize}
\end{lem}
\begin{proof}
Let $\gamma$ be any element in $\Gamma$ and assume for simplicity that $A\in \mathscr{B}_\Gamma(\mathscr{H}_1,\mathscr{H}_2)$. We proceed as in the proof of \Cref{helpinglemma1}. The kernel is always closed in $\mathscr{H}_1$ and the range of $A$ is closed in $\mathscr{H}_2$ by assumption such that both are Hilbert spaces. We check that they are $\Gamma$-invariant due to the $\Gamma$-invariance of the operator $A$: let $L_\gamma$ denote the left action representation on $\mathscr{H}_1$ and $\mathcal{L}_\gamma$ the left action representation on $\mathscr{H}_2$: 
\begin{itemize}
\item[(1)]
suppose $u \in \kernel{A}$, then the intertwining of $A$ and the left action representation shows that also $L_\gamma u \in \kernel{A}$:
\begin{equation*}
Au=0 \quad \Rightarrow \quad AL_\gamma u = \mathcal{L}_\gamma A u =0 \quad \Rightarrow \quad L_\gamma(\kernel{A})\subseteq\kernel{A}\quad.
\end{equation*}
\item[(2)]
if $v \in \range{A}$, then there exists a $u \in \mathscr{H}_1$ such that $v=Au$. Applying $\mathcal{L}_\gamma$ from the left and using the commutation property of $A$ gives
\begin{equation*}
\quad\quad\mathcal{L}_\gamma v = \mathcal{L}_\gamma A u = A L_\gamma u \, \Rightarrow \, \mathcal{L}_\gamma(\range{A})\subseteq \SET{v \in \mathscr{H}_2\,\vert\, \exists\,w\in \range{L_\gamma}\,:\,v=Aw}
\end{equation*} 
and, as $\mathscr{H}_1$ is already a Hilbert $\Gamma$-module, we gain
\begin{equation*}
\mathcal{L}_\gamma(\range{A})\subseteq \SET{v \in \mathscr{H}_2\,\vert\, \exists\,u\in \mathscr{H}_1\,:\,v=Au}=\range{A}\quad.
\end{equation*} 
\end{itemize}
The projectivity follows from the fact that the unitary $\Gamma$-isomorphisms from $\mathscr{H}_i$ to $\ell^2(\Gamma)\otimes \mathcal{H}_i$ ($i\in \SET{1,2}$) are closed maps like in the proof of \Cref{helpinglemma1}.\\
\\
The arguments carry over to $\Gamma$-operators in $\mathscr{C}_\Gamma(\mathscr{H}_1,\mathscr{H}_2)$ with the little modification that according to \Cref{unboundeddomhilbertmod} (2), the domain of $A$ in $\mathscr{H}_1$ is a projective Hilbert $\Gamma$-module and the shown $\Gamma$-invariance holds true as the domain is $\Gamma$-invariant.
\end{proof}
%

\subsection{$\Gamma$-Fredholm operators}\label{chap:galois-sec:gammaop}

We already introduced the three spaces of operators $\mathscr{B}_\Gamma(\mathscr{H}_1,\mathscr{H}_2)$, $\mathscr{L}_\Gamma(\mathscr{H}_1,\mathscr{H}_2)$, and $\mathscr{C}_\Gamma(\mathscr{H}_1,\mathscr{H}_2)$ for two Hilbert $\Gamma$-modules $\mathscr{H}_1$ and $\mathscr{H}_2$. Fredholm operators in $\Gamma$-setting are introduced with the help of $\Gamma$-\textit{ideals}.
\begin{defi}[Definition 3.10.4, 3.10.5 in \cite{shub}, Definition 2.5 in \cite{vaill}]\label{gammaideals}
Given an operator $A \in \mathscr{B}_\Gamma(\mathscr{H}_1,\mathscr{H}_2)$; 
\begin{itemize}
\item[(a)] $A$ is a \textit{finite $\Gamma$-rank operator} if $\dim_\Gamma(\range{A})<\infty$ (we write $A \in \mathscr{R}_\Gamma(\mathscr{H}_1,\mathscr{H}_2)$).
\item[(b)] $A$ is a $\Gamma$-\textit{Hilbert-Schmidt operator} if $\Tr{\Gamma}{A^\ast A} < \infty$ (we write $A \in \mathscr{S}^2_\Gamma(\mathscr{H}_1,\mathscr{H}_2)$).
\item[(c)] $A$ is a $\Gamma$-\textit{trace class operator} if $\Tr{\Gamma}{\absval{A}} < \infty$ (we write $A \in \mathscr{S}^1_\Gamma(\mathscr{H}_1,\mathscr{H}_2)$). 
\item[(c)] $A$ is a $\Gamma$-\textit{compact operator} if $A$ lies in the norm closure of $\mathscr{S}^1_\Gamma(\mathscr{H}_1,\mathscr{H}_2)$ in $\mathscr{B}_\Gamma(\mathscr{H}_1,\mathscr{H}_2)$ (we write $A \in \mathscr{K}_\Gamma(\mathscr{H}_1,\mathscr{H}_2)$).
\end{itemize} 
\end{defi}
\begin{rem}\label{remsgammaideals}
\begin{itemize}
\item[]
\item[(i)] (\cite[Lem.3.10.13(a)]{shub}, \cite[Lem.2.6]{vaill}) All introduced operators are right-ideals over $\mathscr{B}_\Gamma(\mathscr{H}_1)$ and left-ideals over $\mathscr{B}_\Gamma(\mathscr{H}_2)$ and become two-sided ideals for $\mathscr{H}_1=\mathscr{H}_2$. We also have the following inclusions:
$$\mathscr{R}_\Gamma \subset  \mathscr{S}^1_\Gamma \subset  \mathscr{S}^2_\Gamma \subset \mathscr{K}_\Gamma \subset \mathscr{B}_\Gamma \subset \mathscr{C}_\Gamma \subset \mathscr{L}_\Gamma\quad.$$
\item[(ii)] (\cite[Lem.3.10.13(b)]{shub}) An alternative definition for $\Gamma$-trace class operators is the representation as a finite sum of two $\Gamma$-Hilbert Schmidt operators. If $A\in \mathscr{B}_\Gamma(\mathscr{H}_1,\mathscr{H}_2)$, $B \in \mathscr{B}_\Gamma(\mathscr{H}_2,\mathscr{H}_1)$ and either one of them is $\Gamma$-trace class or both $\Gamma$-Hilbert-Schmidt, then $AB\in \mathscr{S}^1_\Gamma(\mathscr{H}_2)$ as well as $BA\in \mathscr{S}^1_\Gamma(\mathscr{H}_1)$ and $\Tr{\Gamma}{AB}=\Tr{\Gamma}{BA}$ holds.
\item[(iii)] $\Gamma$-compact operators are in general not compact in the usual sense.
\end{itemize}
\end{rem}
Let $\mathscr{H}_1(=\mathscr{H})$ and $\mathscr{H}_2$ be free Hilbert $\Gamma$-modules which are isomorphically related to $\ell^2(\Gamma)\otimes \mathcal{H}_1$ and respectively $\ell^2(\Gamma)\otimes \mathcal{H}_2$ with Hilbert spaces $\mathcal{H}_1(=\mathcal{H})$ and $\mathcal{H}_2$.
\begin{figure}[H]
\centering
\begin{tikzpicture}

\draw (-2,1) node {$\mathscr{H}_1$};
\draw (-2,-1) node {$\mathscr{H}_2$};
\draw (2,1) node {$\ell^2(\Upgamma)\otimes\mathcal{H}_1$};
\draw (2,-1) node {$\ell^2(\Upgamma)\otimes\mathcal{H}_2$};

\draw[<->] (-1.65,1)--(1,1);
\draw[<->] (-1.65,-1)--(1,-1);

\draw[->] (-2,0.7)--(-2,-0.7);
\draw[->] (2.2,0.7)--(2.2,-0.7);

\draw (-2.3,0) node {$A$};
\draw (3.2,0) node {$\Iop{\ell^2(\Upgamma)}\otimes\underline{A}$};

\end{tikzpicture}
\caption{Depiction of the commutative diagrams for \clef{commdiagammabound}.}\label{diagramisomorphismshilbertmod}
\end{figure}
\clef{vNcommfree} indicates that an operator $A \in \mathscr{B}_{\Gamma}(\mathscr{H})$ corresponds to an operator of the form $\Iop{\ell^2(\Gamma)}\otimes \underline{A}$ with $\underline{A}\in \mathscr{B}(\mathcal{H})$. As one can view any operator from $\mathscr{H}_1$ to $\mathscr{H}_2$ as an operator on the direct sum $\mathscr{H}_1\oplus\mathscr{H}_2$, we have in fact
\begin{equation}\label{commdiagammabound}
\mathscr{B}_\Gamma(\mathscr{H}_1,\mathscr{H}_2) \cong \mathcal{N}_{r}(\Gamma)\otimes\mathscr{B}(\mathcal{H}_1,\mathcal{H}_2) \quad.
\end{equation}
\clef{vNcommfree} and thus \clef{commdiagammabound} carry over to something similar for $\Gamma$-compact operators which is a more practical characterisation\bnote{F11} of this operator class. Following \cite[pp.9-10]{lance1995hilbert}, one can prove 
\begin{equation}\label{gamma1rankiso}
\mathscr{R}_\Gamma(\mathscr{H}) \cong \mathcal{N}_{r}(\Gamma)\otimes\mathscr{R}(\mathcal{H})
\end{equation} 
and with the same reasoning as applied for the $\Gamma$-morphisms they furthermore implicate
\begin{equation}\label{gammaNrankiso}
\mathscr{R}_\Gamma(\mathscr{H}_1,\mathscr{H}_2) \cong \mathcal{N}_{r}(\Gamma)\otimes\mathscr{R}(\mathcal{H}_1,\mathcal{H}_2)\quad.
\end{equation} 
These isomorphisms are preserved under the norm closure with respect to the norm on $\mathscr{B}_\Gamma(\mathscr{H})$ respectively $\mathscr{B}_\Gamma(\mathscr{H}_1,\mathscr{H}_2)$ such that we get
\begin{equation}\label{gammacompiso}
\mathscr{K}_\Gamma(\mathscr{H}) \cong \mathcal{N}_{r}(\Gamma)\otimes\mathscr{K}(\mathcal{H})
\end{equation}
and also
\begin{equation}\label{commdiagammacomp}
\mathscr{K}_\Gamma(\mathscr{H}_1,\mathscr{H}_2) \cong \mathcal{N}_{r}(\Gamma)\otimes\mathscr{K}(\mathcal{H}_1,\mathcal{H}_2) \quad.
\end{equation}
$\Gamma$-Fredholm operators are defined via parametrices. 
\begin{defi}[Definition 3.10.3 in \cite{shub}]\label{defigammafred}
Let $\mathscr{H}_1,\mathscr{H}_2$ be Hilbert $\Gamma$-modules and $A \in \mathscr{B}_\Gamma(\mathscr{H}_1,\mathscr{H}_2)$; $A$ is a $\Gamma$-\textit{Fredholm operator} (denoted with $A \in \mathscr{F}_\Gamma(\mathscr{H}_1,\mathscr{H}_2))$ if there exists a $B\in \mathscr{B}_\Gamma(\mathscr{H}_2,\mathscr{H}_1)$ such that $(\Iop{\mathscr{H}_1}-BA) \in \mathscr{S}^1_\Gamma(\mathscr{H}_1)$ and $(\Iop{\mathscr{H}_2}-AB) \in \mathscr{S}^1_\Gamma(\mathscr{H}_2)$; this $B$ is the $\Gamma$-\textit{Fredholm parametrix} and the $\Gamma$-\textit{index} is defined by the Atiyah-Bott formula
\begin{equation}\label{vNindex}
\Index_\Gamma(A):=\Tr{\Gamma}{\Iop{\mathscr{H}_1}-BA}-\Tr{\Gamma}{\Iop{\mathscr{H}_2}-AB}\quad .
\end{equation}
\end{defi}
The existence of one parametrix in this definition can be replaced with the existence of left- and right-parametrices for $A$ as these parametrices differ in a $\Gamma$-trace class operator. $\Gamma$-Fredholm operators and the $\Gamma$-index have almost the same properties as Fredholm operators in the ordinary Hilbert space setting apart from the following facts.

\begin{prop}[Section 3.10 in \cite{shub}]\label{propgammafred}
\begin{itemize}
\item[]
\item[(1)] $\Index_\Gamma(A)\,:\,\mathscr{F}_\Gamma(\mathscr{H}_1,\mathscr{H}_2)\,\rightarrow\,\R$ is locally constant.
\item[(2)] If $A \in \mathscr{F}_\Gamma(\mathscr{H}_1,\mathscr{H}_2)$, then $\dim_\Gamma\kernel{A} <\infty$ and $\dim_\Gamma\kernel{A^\ast} <\infty$ such that $\Index_\Gamma(A)=\dim_\Gamma\kernel{A}-\dim_\Gamma\kernel{A^\ast}$.
\item[(3)] $A \in \mathscr{F}_\Gamma(\mathscr{H}_1,\mathscr{H}_2)$ if and only if $\dim_\Gamma\kernel{A} <\infty$ and there exists a closed set $W \in \mathscr{H}_2$ such that $W \subset \range{A}$ and $\codim_\Gamma(W):=\dim_\Gamma\left(\quotspace{\mathscr{H}_2}{W}\right)<\infty$.
\end{itemize} 
\end{prop}
\begin{rem}\label{remsfredprop}
It is required in \Cref{defigammafred} that the remainders are $\Gamma$-trace class operators such that the $\Gamma$-index is well-defined. However, it is sufficient if we replace the ideal of $\Gamma$-trace class operators with $\Gamma$-compact operators and thus any $\Gamma$-ideal in \Cref{gammaideals}. The $\Gamma$-index is then defined with \Cref{propgammafred} (2).
\end{rem}

\section{Sobolev spaces and operators on $\Gamma$-manifolds, $\Gamma$-Projections}\label{chap:gammasob}

We combine the geometric and functional analytic facts from the last section and consider operators which act between sections of $\Gamma$-vector bundles over $\Gamma$-manifolds. We will introduce $L^2$- and Sobolev spaces in this setting which turn out to be free Hilbert $\Gamma$-modules. After this, we take a closer look on spectral projections of certain operators and illustrate a $\Gamma$-version of Seeley's theorem of complex powers.

\subsection{Differential operators in $\Gamma$-setting}\label{chap:specflow-sec:gammaproj}

Let $E\rightarrow M$ be a $\Gamma$-vector bundle over a $\Gamma$-manifold $M$. The left action representation of $\Gamma$ on a smooth section $u$ of $E$ is then described by
\begin{equation}\label{gammaactionsection}
(L^E_\gamma u)(p)=(\pi^E_\Gamma u)(\gamma^{-1}p) \quad
\end{equation}
where $\pi^E_\Gamma$ is an isometry which maps the section at $p$ to $(\gamma\cdot p)$. We equip $M$ with a $\Gamma$-invariant smooth density $\differ \mu$, which is either defined by a $\Gamma$-invariant pseudo-Riemannian metric or otherwise, and $E$ with a $\Gamma$-invariant inner product $\idscal{1}{E_{p}}{\cdot}{\cdot}$ on each fiber $E_p$. We introduce the inner product

\begin{equation}\label{gammainnerprod}
\dscal{1}{L^2(M,E)}{u}{v}=\int_M \dscal{1}{E_{p}}{u(p)}{v(p)} \differ \mu(p)\quad.
\end{equation}
The norm closure of $C^\infty_\comp(M,E)$ with respect to this inner product is the space of square-integrable sections of $E$ on $M$ which we also denote with $L^2(M,E)$. The $\Gamma$-invariance of the density and the bundle metric ensure that the left action representation \clef{gammaactionsection} is unitary with respect to \clef{gammainnerprod}. Other function spaces of interest next to $L^2(M,E)$ are Sobolev spaces. We recall from \cite[Sec.3.9]{shub} that the appropriate norm is defined with a $\Gamma$-invariant partition of unity \clef{gammainvpartition}:
\begin{equation}\label{soboloevgammanorm}
\norm{u}{H^s_\Gamma(M,E)}^2:=\sum_{\substack{j \in J\\ \gamma \in \Gamma}}\norm{\phi_{j,\gamma}u}{H^s(\supp{\phi_{j,\gamma}},E)}^2\quad .
\end{equation} 
The corresponding \textit{$\Gamma$-Sobolev spaces} are then defined for any $s\in \R$ via
\begin{equation}\label{sobolevgamma}
H^s_\Gamma(M,E):=\SET{u\in H^s_\loc(M,E)\,\Big\vert\,\norm{u}{H^s_\Gamma(M,E)} < \infty}\quad.
\end{equation}
We note from \cite{shub} that the definition does not depend on the choice of the $\Gamma$-invariant partition of unity and the choice of the discrete group; they are a priori projective Hilbert $\Gamma$-modules. We also want to point out the resemblance of the $\Gamma$-Sobolev spaces with the Sobolev spaces from \clef{sobolevboundedgeonorm} on manifolds with bounded geometry. \\ 
\\
Let $\mathcal{F}$ be the fundamental domain of the $\Gamma$-action. Because $M$ is a principal $\Gamma$-bundle, the action of the group on the $\Gamma$-manifold becomes isomorphically related to an action on $\Gamma\times \mathcal{F}$. This induces a unitary isomorphism between $L^2$-sections,
\begin{equation}\label{l2freemoduleI}
L^2(M,E) \cong \ell^2(\Gamma)\otimes L^2(\mathcal{F},E\vert_\mathcal{F}) \quad,
\end{equation}
which is given by
\begin{equation}\label{unitarymapl2}
u\,\mapsto\,\sum_{\gamma \in \Gamma} \delta_\gamma \otimes (L^E_\gamma)^\ast u\vert_{\mathcal{F}} \quad.
\end{equation}
for a $u \in L^2(M,E)$. The left action representation $L^E_\gamma$ becomes $l_\gamma \otimes \Iop{\mathcal{F}}$ on $\ell^2(\Gamma)\otimes L^2(\mathcal{F},E\vert_\mathcal{F})$. Furthermore, \clef{l2freemoduleI} implies that $L^2(M,E)$ is a free Hilbert $\Gamma$-module and hence any closed $\Gamma$-invariant subset of $L^2(M,E)$ becomes a projective Hilbert $\Gamma$-module. We recall \Cref{remfun} and replace $L^2(\mathcal{F},E\vert_\mathcal{F})$ with $L^2(\overline{\mathcal{F}},E\vert_{\overline{\mathcal{F}}})=L^2(M_\Gamma,E_\Gamma)$ with $E_\Gamma$ as vector bundle over $M_\Gamma$ since $\mathcal{F}$ and $M_\Gamma$ differ in a set of measure zero and, due to the $\Gamma$-invariance of the density on $M$, the densities on $M$ and $\mathcal{F}$ coincide such that \clef{l2freemoduleI} becomes
\begin{equation}\label{l2freemoduleII}
L^2(M,E) \cong \ell^2(\Gamma)\otimes L^2(M_\Gamma,E_\Gamma) \quad.
\end{equation}
The von Neumann algebra of operators in $\mathscr{B}_\Gamma(L^2(M,E))$ is naturally isomorphic to $\mathscr{N}_r(\Gamma)\otimes \mathscr{B}(L^2(M_\Gamma,E_\Gamma))$.\\
\\
Given two  $\Gamma$-vector bundles $E'\rightarrow M'$ and $E\rightarrow M$ over $\Gamma$-manifolds $M,M'$ and an operator $A:C^\infty_\comp(M',E')\rightarrow C^{-\infty}(M,E)$ with Schwartz-Kernel $K_A$. $A \in \mathscr{B}_\Gamma(L^2(M',F),L^2(M,E))$ implies that the Schwartz kernel is $\Gamma$-equivariant and becomes a distribution on the compact orbit $(M \times M')/\Gamma$ under the diagonal action of $\Gamma$. If we take an operator $A \in \mathscr{S}^1_\Gamma(L^2(M,E))$, the Schwartz-Kernel satisfies $K_A \in L^2(M\times M,E\otimes E^\ast)$. Suppose for simplicity that there are two operators $C\in \mathscr{S}^2_\Gamma(L^2(M,E))$ and $B\in \mathscr{S}^2_\Gamma(L^2(M,E))$ with Schwartz kernels $K_C$ and $K_B$ such that $A=B\circ C$. The analytic expression of the $\Gamma$-trace for $A$ in terms of these Schwartz kernels becomes  
\begin{equation}\label{analyticgammatraceI}
\Tr{\Gamma}{A}=\int_{\mathcal{F}\times M} \tr{E_p}{K_B(p,q)\circ K_C(q,p)} \differ\mu_{M\times M}(p,q) 
\end{equation}
with $\tr{E_p}{\cdot}$ as fibrewise $\Gamma$-invariant trace along $E_p, p \in \mathcal{F}$, and $\differ\mu_{M\times M}$ $\Gamma$-invariant density on $M\times M$. If the Schwartz kernel of $A$ is already continuous at the diagonal, the Schwartz kernels of $B$ and $C$ do as well and vice versa such that the inner integral can be performed, giving $K_A$ and consequently
\begin{equation}\label{analyticgammatraceII}
\Tr{\Gamma}{A}=\int_{\mathcal{F}} \tr{E_p}{K_A(p,p)} \differ\mu_{M}(p) \quad.
\end{equation}
We want to consider differential operators, pseudo-differential and Fourier integral operators in the $\Gamma$-setting. We denote with $\Diff{m}{\Gamma}(M,\Hom(E,F))$, $\ydo{s}{\Gamma}(M,\Hom(E,F))$, and $\FIO{s}_{\Gamma}(M,N;\mathsf{C}';\Hom(E,G))$ the sets of $\Gamma$-differential, $\Gamma$-pseudo-differential and $\Gamma$-Fourier integral operators where $M,N$ are $\Gamma$-manifolds, $E,F\rightarrow M$ and $G\rightarrow N$ are $\Gamma$-vector bundles, $\mathsf{C}$ a canonical relation, and $m\in \N_0$, $s\in\R$. These are in general unbounded operators. If an operator in one of these operator spaces enjoys the property of being properly supported, its Schwartz kernel is compactly supported with support in $(M\times M)/\Gamma$ respectively $(M\times N)/\Gamma$. Thus, any $\Gamma$-invariant differential operator has compactly supported Schwartz kernel on the orbit space.
A certain subclass of $\Gamma$-invariant pseudo-differential operators are \textit{classical} operators, denoted with $\ydo{m}{\mathsf{cl},\Gamma}(M,\Hom(E,F))$, if an $A\in \ydo{m}{\Gamma}(M,\Hom(E,F))$ can be decomposed into a sum of a $\Gamma$-invariant operator with classical symbol and a smoothing $\Gamma$-invariant remainder.
We collect some known results.
\begin{prop}[Corollary 3.9.2, Theorem 3.9.1/4(a) in \cite{shub}]\label{propgammapseudoprop}
\begin{itemize}
\item[]
\item[(1)] If $A\in\ydo{m}{\mathsf{prop},\Gamma}(M,\Hom(E,F))$, then $A \in \mathscr{B}_\Gamma(H^s_\Gamma(M,E),H^{s-m}_\Gamma(M,F))$.
\item[(2)] If $A\in\mathscr{B}_\Gamma(L^2(M,E),L^2(M,F))$ and $A\in\ydo{m}{\mathsf{prop},\Gamma}(M,\Hom(E,F))$ with \\$m< - \dim(M)$, then $A\in \mathscr{S}^1_\Gamma(L^2(M,E),L^2(M,F))$.
\end{itemize}
\end{prop}
The requirement of being properly supported is always fulfilled for differential operators, but too restrictive for a general operator class. We first introduce a wider set of $\Gamma$-operators such that they are smoothing in the sense that they map $\Gamma$-Sobolev spaces into $\Gamma$-Sobolev spaces.
\begin{defi}\label{ssmoothing}
Given $E\rightarrow M$ and $F\rightarrow M'$ $\Gamma$-vector bundles over $\Gamma$-manifolds $M,M'$ with respect to the same $\Gamma$-action and let $A\,:\,C^\infty_\comp(M,E)\,\rightarrow\,C^{-\infty}(M',F)$ be a $\Gamma$-operator; $A$ is said to be \textit{s-smoothing} if it extends to a continuous linear operator between $\Gamma$-Sobolev spaces for any orders $r,p$: 
\begin{equation*}
A\,:\,H^r_\Gamma(M,E)\,\rightarrow\,H^p_\Gamma(M',F)\quad.
\end{equation*} 
\end{defi}
This is a slight generalisation of s-smoothing operators, introduced in \cite[Sec.3.11]{shub}; if $M=M'$ and $A$ a pseudo-differential operator, we write $A\in S\ydo{-\infty}{\Gamma}(M,\Hom(E,F))$. It is proven in \cite[Lem. 3.11.1/2]{shub} that any element in $S\ydo{-\infty}{\Gamma}(M,\Hom(E,F))$ has smooth Schwartz kernel on $M\times M$ and is a $\Gamma$-trace class operator. This new kind of operators motivates another important class of operators.
\begin{defi}[Definition 3.11.2 in \cite{shub}]
Let $E,F$ be $\Gamma$-vector bundles over the $\Gamma$-manifold $M$; an operator $A \in \ydo{m}{\Gamma}(M,\Hom(E,F))$ is called \textit{s-regular} if there is an operator $\hat{A}\in \ydo{m}{\Gamma,\mathsf{prop}}(M,\Hom(E,F))$ such that $(A-\hat{A}) \in S\ydo{-\infty}{\Gamma}(M,\Hom(E,F))$. If $\hat{A}\in \ydo{m}{\mathsf{cl},\Gamma}(M,\Hom(E,F))$ and is properly supported, they are defined as \textit{classical s-regular} operators.
\end{defi}
We denote these spaces of operators with $S\ydo{m}{\Gamma}(M,\Hom(E,F))$ and $S\ydo{m}{\mathsf{cl},\Gamma}(M,\Hom(E,F))$ respectively. We also introduce the notation $S\FIO{m}_\Gamma$ for $m\in \R$ to stress that an operator is the sum of a properly supported $\Gamma$-invariant Fourier integral operator and a s-smoothing remainder. We list some important properties. 
\begin{prop}[Corollary 3.11.1/5, Lemma 3.11.3, Proposition 3.11.2/3/4, Theorem 3.11.2 in \cite{shub}]\label{remarkssmoothing}
Let $E,F\rightarrow M$ be $\Gamma$-vector bundles over the $\Gamma$-manifold $M$ and let $A\,:\,C^\infty_\comp(M,E)\,\rightarrow\,C^{-\infty}(M,F)$ be an $\Gamma$-operator.
\begin{itemize}
\item[(1)] $\Diff{m}{\Gamma}(M,\Hom(E,F))\subset \ydo{m}{\mathsf{prop},\Gamma}(M,\Hom(E,F)) \subset S\ydo{m}{\Gamma}(M,\Hom(E,F))$. 
\item[(2)] $S\ydo{-\infty}{\Gamma}(M,\Hom(E,F))=\bigcap_{m \in \R}S\ydo{m}{\Gamma}(M,\Hom(E,F))$. 
\item[(3)] If $A\in\mathscr{B}_\Gamma(L^2(M,E),L^2(M,F))$ and $A\in S\ydo{m}{\Gamma}(M,\Hom(E,F))$ with \\$m< -\dim(M)$, then $A\in \mathscr{S}^1_\Gamma(L^2(M,E),L^2(M,F))$.
\item[(4)] If $A\in S\ydo{m}{\Gamma}(M,\Hom(E,F))$, then $A\in\mathscr{B}_\Gamma(H^s_\Gamma(M,E),H^{s-m}_\Gamma(M,F))$ for all $s\in \R$.
\item[(5)] For any $s\in \R$ there exists a formally self-adjoint operator $\Lambda^s\in S\ydo{s}{\mathsf{cl},\Gamma}(M,E)$ which is elliptic and maps from $H^s_\Gamma(M,E)\rightarrow L^2(M,F)$ as a topological isomorphism of Hilbert $\Gamma$-modules.
\item[(6)] If $A \in S\ydo{m}{\Gamma}(M,E)$ is elliptic with $m>0$ and formally self-adjoint, then $A$ is essentially self-adjoint and $\Chi_I(A)\in S\ydo{-\infty}{\Gamma}(M,E)$ for any bounded Borel set $I \subset \R$.
\item[(7)] If $E,F$ are Hermitian $\Gamma$-vector bundles and $A \in S\ydo{m}{\mathsf{cl},\Gamma}(M,\Hom(F,E))$ elliptic, then $A\in \mathscr{F}_\Gamma(H^s_\Gamma(M,F),H^{s-m}_\Gamma(M,E))$ for any $s\in \R$ with an $s$-independent $\Gamma$-index
\vspace*{-1em}
\begin{equation*}
\Index_\Gamma(A)=\dim_\Gamma\kernel{A}-\dim_\Gamma\kernel{A^{\ast}}
\end{equation*}
where $A^{\ast}\in S\ydo{m}{\mathsf{cl},\Gamma}(M,\Hom(E,F))$ is the formal adjoint of $A$ with respect to \clef{gammainnerprod}.
\end{itemize}
\end{prop}
Properties (1) and (7) together imply that elliptic $\Gamma$-differential and elliptic, properly supported $\Gamma$-pseudo-differential operators are $\Gamma$-Fredholm. A generalisation of property (6) to unbounded intervals is presented in the following subsection. (5) implies together with \Cref{prophilbertgammamodules} (1) that the $\Gamma$-Sobolev spaces are in fact \textit{free Hilbert $\Gamma$-modules} for all $s\in \R$.
%

\subsection{Projectors in the $\Gamma$-setting}\label{chap:specflow-sec:gammaproj}

Given an elliptic and self-adjoint operator $A\in \ydo{m}{\Gamma,\mathsf{prop}}(M,E)$ of order $m>0$ which acts between smooth sections of the (Hermitian) $\Gamma$-vector bundle $E$ over the (Riemannian) manifold $M$. The characteristic functions of such an operator can be defined by means of any unbounded functional calculus for self-adjoint operators, but this approach lacks of analytic informations. To bypass this, we rewrite the characteristic functions $x \mapsto\Chi_{(0,\infty)}(x)$ and $x\mapsto\Chi_{(-\infty,0)}(x)$ with the signum-function which in turn can be expressed as quotient $x/\absval{x}$:
\begin{equation}\label{chireps}   
\begin{split}
\Chi_{>0}(x)&:=\Chi_{(0,\infty)}=\frac{1}{2}\left(1+\frac{x}{\absval{x}}-\Chi_{\SET{0}}(x)\right)\\
\Chi_{<0}(x)&:=\Chi_{(-\infty,0)}=\frac{1}{2}\left(1-\frac{x}{\absval{x}}-\Chi_{\SET{0}}(x)\right)\quad.
\end{split}
\end{equation}
The task is to show that the expressions $A\circ \absval{A}^{-1}$ and $\absval{A}^{-1}=(A^\ast\circ A)^{-1/2}=(A^2)^{-1/2}$ in particular are meaningful. If $M$ is closed, Seeley's theorem for complex powers states that under further conditions on $A$ any of its complex powers become meaningful as pseudo-differential operator and their principal symbols can be calculated from the one of $A$. Such a result is known for manifolds of bounded geometry; in \Cref{chap:manbound} we present and explain this version of Seeley's theorem as \Cref{kordyuresults}. We transfer the result to $\Gamma$-manifolds.
\begin{cor}\label{kordyuresultsgamma}
Let $B \in \ydo{m}{\Gamma,\mathsf{prop,cl}}(M,E)$ be an elliptic and positive operator with scalar positive definite principal symbol ${\sigma}_m(B)(p,\xi)$ such that $({\sigma}_m(B)(p,\xi)-\lambda)$ is not vanishing for $\lambda$ in a sector $\Lambda_{\theta}\subset \uprho(B)$ for $\theta \in (0,\uppi/2)$ and $\xi\in \dot{T}_pM$. Then the following holds: 
\begin{itemize}
\item[(1)] there exists a value $r>0$ such that for $\lambda \in \Lambda_{\theta,r}\subset \C$ the operator $B$ becomes sectorial and the resolvent satisfies $\mathsf{R}(B,\lambda)\in S\ydo{-m}{\Gamma}(M,E)$ with principal symbol $({\sigma}_m(B)(p,\xi)-\lambda)^{-1}$; moreover $\mathsf{R}(B,\lambda)$ becomes a bounded operator from $H^s_\Gamma(M,E)$ to $H^{s+m}_\Gamma(M,E)$ for any $s\in \R$.
\item[(2)] $B$ generates a holomorphic semigroup $\expe{-zB}$ for $\Rep{z}>0$ in $L^2(M,E)$.
\item[(3)] $B^z \in S\ydo{m\Rep{z}}{\Gamma}(M,E)$ for all $z\in \C$ with principal symbol \clef{princsymbcomplexpower}.
\end{itemize}
\end{cor}
The third point in this result is thereby an analogue of Seeley's theorem for complex powers in the setting of Galois coverings. It is only left to show that the resolvent commutes with the left action representation if $B$ does: from the commuting $(\lambda-B)L^E_\gamma=L^E_\gamma (\lambda-B)$ for all $\lambda \in \C$ and any $\gamma \in \Gamma$, we can conclude for $\lambda \in \uprho(B)$ the commuting
of the resolvent with the left action representation:

\begin{equation*}
(\lambda-B)^{-1}L^E_\gamma =(\lambda-B)^{-1}L^E_\gamma (\lambda-B)(\lambda-B)^{-1}=L^E_\gamma(\lambda-B)^{-1}\quad.
\end{equation*}
Thus, $f(B)$ in \clef{cauchyoperator} commutes with $L^E_\Gamma$ for all $\gamma$ and allowable functions $f$.\\
\\
\noindent We now come back to the (essentially) self-adjoint and elliptic operator $A\in \ydo{m}{\Gamma,\mathsf{prop,cl}}(M,E)$, $m >0$; the operator $B:=A^\ast\circ A=A^2 \in \ydo{2m}{\Gamma,\mathsf{prop,cl}}(M,E)$ is again self-adjoint and its spectrum lies on the positive half-line. Hence $B$ becomes positive and altogether sectorial. The ellipticity of $B$ is also inherited from the ellipticity of $A$ and the principal symbol of $B$ is positive definite. Under this assumptions, the expression $\absval{A}^{-1}$ becomes well-defined as s-regular pseudo-differential operator of order $(-m)$ and finally one can consider projectors as s-regular $\Gamma$-pseudo-differential operators.
\begin{cor}\label{projsreggenop}
Let $A\in \ydo{m}{\Gamma,\mathsf{prop,cl}}(M,E)$ be elliptic and (essentially) self-adjoint with positive order $m>0$; the spectral projections onto eigenspaces of (the closure of) $A$ in the spectral range $(0,\infty)$
\begin{equation}\label{apsprojectorspos}
P_{>0}(A):=P_{+}(A):=\Chi_{(0,\infty)}(A)=\frac{1}{2}\left(\Iop{}+ A\circ\vert A \vert^{-1} - P_{0}(A)\right)
\end{equation}
and in the spectral range $(-\infty,0)$
\begin{equation}\label{apsprojectorsneg}
P_{<0}(A):=P_{-}(A):=\Chi_{(-\infty,0)}(A)=\frac{1}{2}\left(\Iop{}- A\circ\vert A \vert^{-1} - P_{0}(A)\right)
\end{equation}
are well-defined and satisfy $P_{\pm}(A)\in S\ydo{0}{\Gamma}(M,E)$ with principal symbols
\begin{equation}\label{princsymprojgen}
{\sigma}_0(P_{\pm})(p,\xi)=\frac{1}{2}\left(1\pm{\sigma}_m(A)(p,\xi){\sigma}_{-m}\left(\absval{A}^{-1}\right)(p,\xi)\right)\\
\end{equation}
for all $p\in M$ and $\xi\in \dot{T}^\ast M$.
\end{cor}
\begin{rem}\label{remslabel1}
\begin{itemize}
\item[]
\item[(i)] One has $P_{\pm}=p_{\pm}+r_{\pm}$ where $p_{\pm} \in \ydo{0}{\Gamma,\mathsf{prop,cl}}(M,E)$ and $r_{\pm}$ is s-smoothing. As the principal symbol is defined modulo smoothing terms, one has ${\sigma}_m(P_{\pm})={\sigma}_m(p_{\pm})$.
\item[(ii)] $P_{\geq 0}(A):=\Chi_{[0,\infty)}(A)$ and $P_{\leq }(A):=\Chi_{(-\infty,0]}(A)$ differ from $P_{+}(A)$ respectively $P_{-}(A)$ in $P_{0}:=\Chi_{\SET{0}}$ which is s-smoothing. Hence, we also have $P_{\geq 0}(A),P_{\leq 0}(A)\in S\ydo{0}{\Gamma}(M,E)$ with principal symbols \clef{princsymprojgen}.
\item[(iii)] Since $A$ has been chosen to be self-adjoint, the complex power $z$ in the construction of $(A^\ast\cdot A)^z$ does not depend on the angle of the rays in the Hankel-like contour of \clef{cauchyoperator}. The same treatment could be applied to non-self-adjoint operators $A$ since 
$A^\ast A$ is self-adjoint, but now each complex power $z$ depends on the choice of the ray by fixing an angle via $\mathrm{arg}(z)$. This observation has been used in \cite{BaerStroh2} for the index theorem with compact Cauchy boundaries, but a-priori non-self-adjoint Dirac operators.
\end{itemize}
\end{rem}
If we cut the spectral range at another point $a\in \R$ than zero, we will also write 
\begin{eqnarray*}
P_{>a}(A)&:=&P_{(a,\infty)}(A)=\Chi_{(a,\infty)}(A)\quad\text{and}\quad P_{<a}(A):=P_{(-\infty,a)}(A)=\Chi_{(-\infty,a)}(A)\\ 
P_{\geq a}(A)&:=&P_{>a}(A)+P_{\SET{a}}(A)\quad\text{and}\quad P_{\leq a}(A):=P_{<a}(A)+P_{\SET{a}}(A)\,.
\end{eqnarray*}
These projectors are also s-regular $\Gamma$-pseudo-differential operators: depending on whether the spectral cutpoint $a \in \R$ is bigger or smaller then $0$, they can be related to $P_{\lesseqgtr 0}$ and differ in projectors with bounded intervals in the spectrum of $A$. Consequently,
\begin{equation}
P_{\gtrless b}(A), P_{\geq b}(A), P_{\leq b}(A), P_I(A) \in S\ydo{0}{\Gamma}(M,E)
\end{equation}
We also introduce the notation $P^\perp_I(A):=\Iop{E}-P_I(A)=P_{I^\complement}(A)$ which is an $s$-regular operator of order $0$ for all $I\subset \R$. 
\section{Setting for $L^2$-Gamma-Fredholmness}\label{chap:setting}

From now on, $M$ will always denote a globally hyperbolic spin manifold of even dimension $(n+1)$ ($n$ odd). The spacelike Cauchy hypersurface $\Sigma\subset M$ is assumed to be a Riemannian Galois covering with respect to a discrete, freely discontinuous, freely and cocompactly acting group $\Gamma$ with infinite conjugacy class. Thus, $\Sigma$ is from now on a $\Gamma$-manifold. $M$ becomes a spatial $\Gamma$-manifold with globally hyperbolic base such that it is isometric to $\timef(M)\times \Sigma/\Gamma$. We always assume that the group does not vary with the time parameter $t$. We will also specify $M$ to be temporal compact, i.e. there are times $t_1,t_2\in \R$ such that the time domain is $\timef(M)=[t_1,t_2]$. The base $M_\Gamma$ then becomes compact with boundary  $\Sigma_{1}/\Gamma\sqcup\Sigma_2/\Gamma$. This boundary lifts to two disjoint boundary hypersurfaces $\Sigma_{1}$ and $\Sigma_2$ on $M$, induced by the Cauchy hypersurface $\Sigma$ at times $t_1$, $t_2$: $\Sigma_{1}:=\Sigma_{t_1}=\SET{t_1}\times\Sigma$ and $\Sigma_2:=\SET{t_2}\times\Sigma$. These boundary hypersurfaces are $\Gamma$-manifolds with respect to the same group $\Gamma$. The metric $\met$ on $M$ is given by 
$$ \met=-N^2\differ^{\otimes 2}t+ \met_t$$
from \cite[Thm.2.1/Eq.2.8]{OD1}. The smooth one-parameter family of hypersurface metrics $\SET{\met_t}$ on the slices $\Sigma_t$ is assumed to be $\Gamma$-invariant which implies the spatial $\Gamma$-invariance of $\met$ and thus a spatial $\Gamma$-invariant volume form.\\
\\
The assumptions on $M$ to be globally hyperbolic and spin as well as on $\Sigma$ to be a spin Cauchy hypersurface is justified as long these to properties are true on the compact bases $M_\Gamma$ respectively $\Sigma_\Gamma$. The tangent bundles $T\Sigma$ and $TM$ are $\Gamma$-vector bundles with respect to the covering $\Sigma \rightarrow \Sigma_\Gamma$ respectively $M\rightarrow M_\Gamma$. As the spinor bundles over $\Sigma$ or $M$ are naturally related to the tangent bundles over $\Sigma$ respectively $M$, the bundles $\spinb^{\pm}(M)$, $\spinb(M)$, $\spinb^{\pm}(\Sigma)=\spinb(\Sigma)$ become $\Gamma$-vector bundles. The left action representation on $M$ acts as $L^{\spinb(\Sigma_t)}_\gamma$ in spatial direction and as identity in temporal direction, i.e. $(L^{\spinb(M)}_\gamma u)(t,p)=\pi^{\spinb(\Sigma_t)}_\Gamma u(t,\gamma^{-1}\cdot p)$ for $u\in C^\infty(\spinb(\Sigma_t))$.\\
\\
We will focus on $L^2$- and Sobolev-sections of the spinor bundles $\spinb^{\pm}_{}(\Sigma_t)=\spinb_{}(\Sigma_t)$ in the $\Gamma$-setting. Due to the diffeomorphy of $\Sigma$ to $\Gamma\times \mathcal{F}$ with $\mathcal{F}$ as fundamental domain of the $\Gamma$-action, the $\Gamma$-Sobolev spaces and $L^2$-spaces are free Hilbert $\Gamma$-modules:
\begin{equation*}
\begin{split}
L^2(\spinb^{\pm}_{}(\Sigma_t))&\,\cong\, \ell^2(\Gamma)\otimes L^2(\mathcal{F},\spinb^{\pm}_{}(\Sigma_t)\vert_{\mathcal{F}})\,\cong\, \ell^2(\Gamma)\otimes L^2(\spinb^{\pm}_{}(\Sigma_t/\Gamma))\\
H^s_\Gamma(\spinb^{\pm}_{}(\Sigma_t))&\,\cong\, \ell^2(\Gamma)\otimes H^s(\mathcal{F},\spinb^{\pm}_{}(\Sigma_t)\vert_{\mathcal{F}})\,\cong\, \ell^2(\Gamma)\otimes H^s(\spinb^{\pm}_{}(\Sigma_t/\Gamma))
\end{split}
\end{equation*}
for all $s\in \R_{>0}$ and $t\in \timef(M)$; the right isomorphisms in each line come from the density of $\mathcal{F}$ in the compact base. As $\Gamma$ is assumed to have infinite conjugacy class, these Hilbert $\Gamma$-modules imply that
$$\mathscr{B}_\Gamma(H^s_\Gamma(\spinb^{\pm}_{}(\Sigma_t)),H^{s-m}_\Gamma(\spinb^{\pm}_{}(\Sigma_t)))\,\cong\,\mathscr{N}_r(\Gamma)\otimes \mathscr{B}(H^s(\spinb^{\pm}_{}(\Sigma_t/\Gamma)),H^{s-m}(\spinb^{\pm}_{}(\Sigma_t/\Gamma))) $$
are von Neumann algebra factors for any $s\geq m$.\\
\\
Assuming that the hypersurface Dirac operator $A_t$ is a $\Gamma$-invariant Riemannian Dirac operator, it becomes an elliptic, essentially self-adjoint and s-regular $\Gamma$-operator. It follows from \Cref{propgammapseudoprop} (1) that it is a $\Gamma$-morphism from $H^s_\Gamma(\spinb^{\pm}_{}(\Sigma_t))$ to $H^{s-1}_\Gamma(\spinb^{\pm}_{}(\Sigma_t))$. Properties (7) and (8) of \Cref{remarkssmoothing} furthermore imply that the spectral projections onto $L^2$-subsets with respect eigenvalues in a bounded Borel set as well as the projection onto the kernel of each $A_t$ in $L^2(\spinb^{\pm}_{}(\Sigma_t))$ are $\Gamma$-trace class operators. The most relevant conclusion is that each $\Gamma$-invariant hypersurface Dirac operator is $\Gamma$-Fredholm between $\Gamma$-Sobolev spaces and the $\Gamma$-index vanishes due to self-adjointness (\Cref{remarkssmoothing} (8)).\\   
\\
We consider the Lorentzian (Atiyah-Singer) Dirac operators $D_{\pm}$ and their decompositions into tangential and normal operators along a slice $\Sigma_t$, introduced in \cite[Sec.3.2]{OD1}. The $\Gamma$-invariance of the hypersurface metric implies that the tangential part, i.e. $A_t$, commutes with the left action representation $L^{\spinb(\Sigma_t)}_\gamma$ on $\Sigma_t$ at each time $t$ and $\gamma\in \Gamma$. Hence the Dirac operators become $\Gamma$-invariant with respect to this induced $\Gamma$-action and can be viewed as lift of Dirac operators $\underline{D}_{\pm}$ on the compact base.

\section{Well-posedness in $\Gamma$-setting}\label{chap:wellposgamma}

As a starting point in showing $\Gamma$-Fredholmness of $D^{}_{\pm}$, we provide a $\Gamma$-version of the well-posedness results \Cref{inivpwell} and \Cref{homivpwell}.

\subsection{Well-posedness}\label{chap:wellposgamma-sec:wellposgamma}

If $M$ is temporal compact, any spatially compact subset of $M$ is itself compact since it is a closed subset in $M$ and $\Jlight{}(K)\subset M$ for any $K\Subset M$ becomes compact. Thus, 
\begin{eqnarray*}
C^l_\scomp([t_1,t_2],H^s_\loc(\spinb^{\pm}_{}(\Sigma_\bullet))) &\rightarrow& C^l([t_1,t_2],H^s_\loc(\spinb^{\pm}_{}(\Sigma_\bullet))) \\
L^2_{\loc,\scomp}([t_1,t_2],H^s_\loc(\spinb^{\pm}_{}(\Sigma_\bullet))) &\rightarrow& L^2([t_1,t_2],H^s_\loc(\spinb^{\pm}_{}(\Sigma_\bullet)))
\end{eqnarray*} 
for $l \in \N_0$ and $s\in \R$. The spaces $C^0_\scomp$ and $L^2_{\loc,\scomp}$ have been introduced in \cite[Sec.2.3]{OD1}. We modify these spaces with $H^s_\Gamma \subset H^s_\loc$ for any $s\in \R$:
\begin{equation}\label{upgammafesections}
\begin{split}
\text{(a)}\quad & FE^s_\Gamma(M,[t_1,t_2],\spinb^{\pm}_{}(M)):= C^0([t_1,t_2],H^s_\Gamma(\spinb^{\pm}_{}(\Sigma_\bullet))) \quad; \\
\text{(b)}\quad & FE^s_\Gamma(M,[t_1,t_2],D^{}_{\pm}):=\Big\lbrace u \in FE^s_\Gamma(M,[t_1,t_2],\spinb^{\pm}_{}(M))\,\Big\vert \\ 
&\mathrel{\phantom{XXXXXXXXXXXXXXXXXXXXX}} D^{}_{\pm}u\in L^2([t_1,t_2],H^s_\Gamma(\spinb^{\mp}_{}(\Sigma_\bullet)))\Big\rbrace \\
\text{(c)}\quad & FE^s_\Gamma\left(M,\kernel{D^{}_{\pm}}\right):=\SET{u \in FE^s_\Gamma(M,[t_1,t_2],\spinb^{\pm}_{}(M))\,\Big\vert\,D^{}_{\pm}u=0} \quad .
\end{split}
\end{equation}
A seminorm on $FE^s_{\Gamma}(M,[t_1,t_2],\spinb^{\pm}_{}(M))$ is defined with the norm of $\Gamma$-Sobolev spaces:
\begin{equation}\label{snclkgamma}
\Vert u \Vert_{I,K,l,s,\Gamma}:= \max_{k \in [0,l]\cap \N_0} \max_{t \in I} \big\Vert (\nabla_t)^{l} u\big\Vert_{H^s_\Gamma(\spinb^{\pm}_{}(\Sigma_\bullet))} \quad.
\end{equation}
$L^2([t_1,t_2],H^s_\Gamma(\spinb^{\pm}_{}(\Sigma_\bullet)))$ is a bundle of free Hilbert $\Gamma$-modules; the left action representation on $H^s_\Gamma(\spinb^{\pm}_{}(\Sigma_t))$ for each time induces a left action on $L^2([t_1,t_2],H^s_\Gamma(\spinb^{\pm}_{}(\Sigma_\bullet)))$ at each time $t\in [t_1,t_2]$. Thus, $L^2([t_1,t_2],H^s_\Gamma(\spinb^{\pm}_{}(\Sigma_\bullet)))$ become free Hilbert $\Gamma$-modules for any $s$ on their own right. We point out that
\begin{equation*}
L^2([t_1,t_2],H^0_\Gamma(\spinb^{\pm}_{}(\Sigma_\bullet)))=L^2([t_1,t_2],L^2(\spinb^{\pm}_{}(\Sigma_\bullet)))=L^2(\spinb^{\pm}_{}(M)) \quad.
\end{equation*}
The $\Gamma$-versions of \Cref{inivpwell} and \Cref{homivpwell} are then reformulated as follows.
\begin{theo}\label{inivpwellgammatempcomp}
For a fixed $t \in [t_1,t_2]$ with $M$ a temporal compact globally hyperbolic spatial $\Gamma$-manifold and any $s\in \R$ the maps
\begin{equation}\label{inivpmapqtempcomp}
\rest{t} \oplus D^{}_{\pm} \,\,:\,\, FE^s_{\Gamma}(M,[t_1,t_2],D^{}_{\pm}) \,\,\rightarrow\,\, H^s_\Gamma(\spinb^{\pm}_{}(\Sigma_t))\oplus L^2([t_1,t_2],H^s_\Gamma(\spinb^{\mp}_{}(\Sigma_\bullet)))
\end{equation}
are $\Gamma$-isomorphisms of Hilbert $\Gamma$-modules.
\end{theo}
\begin{cor}\label{homivpwellgammatempcomp}
For a fixed $t \in [t_1,t_2]$ with $M$ a temporal compact globally hyperbolic spatial $\Gamma$-manifold and any $s\in \R$ the maps
\begin{equation}\label{homivpmapqtempcomp}
\rest{t}  \,\,:\,\, FE^s_{\Gamma}\left(M,\kernel{D^{}_{\pm}}\right) \,\,\rightarrow\,\, H^s_\Gamma(\spinb^{\pm}_{}(\Sigma_t))
\end{equation}
are $\Gamma$-isomorphisms of Hilbert $\Gamma$-modules.
\end{cor}
\begin{proof}
The continuity of $D_{\pm}$ as map from $FE^s_{\Gamma}(M,[t_1,t_2],D_{\pm})$ to $L^2([t_1,t_2],H^s_\Gamma(\spinb^{\mp}(\Sigma_\bullet)))$ follows by construction of the domain. The continuity of the restriction is given by the following modified argument: for all $t\in [t_1,t_2]$ the diffeomorphism between $M$ and the product manifold $[t_1,t_2] \times \Sigma$ implies that for each $\Gamma$-hypersurface one and the same time independent $\Gamma$-invariant partition of unity $\SET{\phi_{i,\gamma}}_{\substack{i \in I \\ \gamma \in \Gamma}}$, subordinated to a covering of $\Sigma$, can be chosen. Thus, every slice in $\SET{\Sigma_t}_{t \in [t_1,t_2]}$ has the same partition of unity. As in the the proof of \Cref{inivpwell} in \cite[Sec.4.3]{OD1} with $\mathcal{K} \cap \Sigma_t$ replaced by $K(t,i,\gamma):=\mathcal{K} \cap \Sigma_t \cap \supp{\phi_{i,\gamma}}$ for each $i \in I$, $\gamma \in \Gamma$ and $\mathcal{K}$ spatially compact, we gain 
\begin{eqnarray*}
\norm{\rest{t}u}{H^s_\Gamma(\spinb^{\pm}(\Sigma_t))} &\leq& \Vert u \Vert_{[t_1,t_2],\mathcal{K},0,s,\Gamma} \quad.
\end{eqnarray*}
The second estimate in the proof of \Cref{inivpwell} can be modified after introducing the $\Gamma$\textit{-s-energy} as square of the $\Gamma$-Sobolev norm:
\begin{equation*}
\mathcal{E}_{s,\Gamma}(u,\Sigma_t) = \norm{u}{H^s_\Gamma(\spinb(\Sigma_t))}^2=\sum_{\substack{i \in I\\\gamma \in \Gamma}}\norm{\phi_{i,\gamma} u}{H^s(K(t,i,\gamma),\spinb(\Sigma_t))}^2= \sum_{\substack{i \in I\\\gamma \in \Gamma}} \mathcal{E}_{s}(\phi_{i,\gamma}u,\Sigma_t)\quad.
\end{equation*} 
One observes that this energy can be rewritten as the usual $s$-energy from \cite[Sec.4.2]{OD1}. We abbreviate $\phi:=\phi_{i,\gamma}$ and $\tilde{u}=\phi u$ and perform the estimate of the energy $\mathcal{E}_{s}(\phi_{i,\gamma}u,\Sigma_t)$ like in the proof of \cite[Prop.4.6]{OD1} where one makes use of the time independence of the partition of unity: we abbreviate $B_{t,\pm}=\pm i A_t-n/2 H_t$ and observe
\begin{align*}
\frac{\differ}{\differ t}\mathcal{E}_s(\tilde{u},\Sigma_{t}) 
&\leq n\dscal{1}{L^2(\spinb(\Sigma_t))}{H_t \Lambda^s_t \tilde{u}}{\Lambda^s_t \tilde{u}}+c_2\norm{\tilde{u}}{H^s(\spinb(\Sigma_t))}^2-2 \Re\mathfrak{e}\left\lbrace\dscal{1}{H^s(\spinb(\Sigma_t))}{\tilde{u}}{\phi \nabla_{\mathsf{v}} u} \right\rbrace \\
&= n\dscal{1}{L^2(\spinb(\Sigma_t))}{H_t \Lambda^s_t \tilde{u}}{\Lambda^s_t \tilde{u}}+c_2\norm{\tilde{u}}{H^s(\spinb(\Sigma_t))}^2\\
&\quad\quad+2 \Re\mathfrak{e}\left\lbrace\dscal{1}{H^s(\spinb(\Sigma_t))}{\tilde{u}}{\upbeta \phi D_{\pm} u} \right\rbrace + 2 \Re\mathfrak{e}\left\lbrace\dscal{1}{H^s(\spinb(\Sigma_t))}{\tilde{u}}{\phi B_{t,\pm} u} \right\rbrace \\
&\stackrel{(\ast)}{\leq} n\absval{\dscal{1}{L^2(\spinb(\Sigma_t))}{H_t \Lambda^s_t \tilde{u}}{\Lambda^s_t \tilde{u}}}{}+c_2\norm{\tilde{u}}{H^s(\spinb(\Sigma_t))}^2\\
&\quad\quad+2 \Re\mathfrak{e}\left\lbrace\dscal{1}{L^2(\spinb(\Sigma_t))}{\Lambda^s_t\tilde{u}}{\upbeta \Lambda^s_t \phi D_{\pm} u} \right\rbrace + 2 \Re\mathfrak{e}\left\lbrace\dscal{1}{H^s(\spinb(\Sigma_t))}{\tilde{u}}{\phi B_{t,\pm} u} \right\rbrace \\
&\leq  n\norm{H_t \Lambda^s_t \tilde{u}}{L^2(\spinb(\Sigma_t))}^2+c_3\norm{\tilde{u}}{H^s(\spinb(\Sigma_t))}^2+ \norm{\upbeta \Lambda^s_t \phi D_{\pm} u}{L^2(\spinb(\Sigma_t))}^2 \\
&\quad\quad+ \norm{\Lambda^s_t \phi B_{t,\pm} u}{L^2(\spinb(\Sigma_t))}^2 \\
&\leq  c_4\norm{\tilde{u}}{H^s(\spinb(\Sigma_t))}^2+\norm{\phi D_{\pm} u}{H^s(\spinb(\Sigma_t))}^2 + \norm{\phi B_{t,\pm} u}{H^s(\spinb(\Sigma_t))}^2\quad;
\end{align*}
We used in $(\ast)$ the same isometry property from \cite[Eq.3.23/3.24]{OD1}. Since $H_t \Lambda^s_t$ is a $\Gamma$-pseudo-differential operator of order s, we have $H_t \Lambda^s_t \tilde{u} \in L^2(\spinb^{\pm}(\Sigma_t))$ and thus the estimate for Sobolev norms; the isometry property of $\upbeta$ has been used in the last step. Since $B_{t,\pm}$ are properly supported, we have $B_{t,\pm} u\vert_{\Sigma_t} \in H^s_\loc(\spinb^{+}(\Sigma_t))$ for $u\vert_{\Sigma_t} \in H^{s+1}_\loc(\spinb^{+}(\Sigma_t))$. Summing over all covering balls and $\Gamma$-actions leads to
\begin{equation*} 
\sum_{\substack{i \in I\\\gamma \in \Gamma}} \norm{\phi_{i,\gamma} B_{t,\pm} u}{H^s(\spinb(\Sigma_t))}^2 \leq c \sum_{\substack{i \in I\\\gamma \in \Gamma}} \norm{\phi_{i,\gamma} u}{H^{s+1}(\spinb(\Sigma_t))}^2 \leq \tilde{c} \mathcal{E}_{s,\Gamma}(u,\Sigma_{t}) \quad.
\end{equation*}
The second inequality is a result of the continuous inclusion of Sobolev spaces. One finally yields
\begin{equation*}
\frac{\differ}{\differ t}\mathcal{E}_s(u,\Sigma_{t}) \leq c_5\mathcal{E}_{s,\Gamma}(u,\Sigma_t)+\norm{D_{\pm} u}{H^s_\Gamma(\spinb(\Sigma_t))}^2
\end{equation*}
and repeating all further steps from the proof in the general case as well as from \cite[Cor.4.7]{OD1} shows that for $\tau \in [t_1,t_2]$, $K \subset M$ compact, and $s \in \R$ there exists a $C>0$ such that
\begin{equation}\label{enestgamma}
\mathcal{E}_{s,\Gamma}(u,\Sigma_{t})\leq C \left(\mathcal{E}_{s,\Gamma}(u,\Sigma_{\tau})+ \norm{D_{\pm} u}{[t_1,t_2],\Jlight{}(K),s,\Gamma}^2\right) 
\end{equation}
is valid for all $t$ in any $[t_1,t_2]$ and for all $u \in FE^{s+1}_{\Gamma}(M,[t_1,t_2],D_{\pm})$ with $D_{\pm}u \in FE^s_{\Gamma}(M,[t_1,t_2],\spinb^{\mp}(M))$ and $\supp{u}\subset \Jlight{}(K)$. In comparison to \cite[Prop.4.6]{OD1} and \cite[Cor.4.7]{OD1}, the spatial compact support has been ensured by the $\Gamma$-invariant partition of unity which has been used in order to carry over the proof up to these modifications. Hence, the constant $C$ depends on the projection of the support onto $\Sigma\mirror{\setminus}\Gamma$, but since this base is compact by our general preassumption, the constant $C$ is now independent of the support of $u$. The $\Gamma$-version of \cite[Cor.4.8]{OD1} then follows with identical arguments and can be used for the well-posedness of the Cauchy problem on $\Gamma$-manifolds: any finite energy section in $FE^s_{\Gamma}(M,[t_1,t_2],D_{\pm})$ can be estimated with an initial value $u_0 \in H^s_\Gamma(\Sigma_1,\spinb^{\pm}(\Sigma_1))$ and inhomogeneity $f=D_{\pm}u \in L^2([t_1,t_2],H^s_\Gamma(\spinb^{\mp}(\Sigma_\bullet)))$ with \clef{enestgamma} at initial time $t=t_1$:
\begin{equation*}
\begin{split}
\Vert u \Vert_{[t_1,t_2],\Jlight{}(K),0,s,\Gamma}^2 &\leq \max_{\tau \in [t_1,t_2]}\SET{\norm{u}{H^s_\Gamma(\spinb^{\pm}(\Sigma_\tau))}^2}=\max_{\tau \in [t_1,t_2]}\big\{\mathcal{E}_{s,\Gamma}(u,\Sigma_{\tau})\big\}\\
&\leq C \left(\mathcal{E}_{s,\Gamma}(u_0,\Sigma_{1})+ \norm{f}{[t_1,t_2],\Jlight{}(K),s,\Gamma}^2\right) \quad .
\end{split}
\end{equation*}
The rest of the proof works analogously such that one yields an isomorphism between the topological vector spaces $FE^s_{\Gamma}(M,[t_1,t_2],D_{\pm})$ and $H^s_\Gamma(\spinb^{\pm}(\Sigma_t))\oplus L^2([t_1,t_2],H^s_\Gamma(\spinb^{\pm}(\Sigma_\bullet)))$. $D_{\pm}$ are $\Gamma$-invariant by assumption and the restriction operator $\rest{t}$ just fixes a slice at time $t$, wherefore it intertwines the $\Gamma$-action. Hence, the isomorphisms in the claims are $\Gamma$-invariant. It is left to show, that $FE^s_\Gamma(M,[t_1,t_2],D^{}_{\pm})$ are Hilbert $\Gamma$-modules. We clarified beforehand that $L^2([t_1,t_2],H^s_\Gamma(\spinb^{\pm}_{}(\Sigma_\bullet)))$ are free Hilbert $\Gamma$-modules for all $s\in \R$. Since $H^s_\Gamma(\spinb^{\pm}_{}(\Sigma_t))$ is also a free Hilbert $\Gamma$-module, its direct sum in \clef{inivpmapqtempcomp} becomes a free Hilbert $\Gamma$-module due to \Cref{helpinglemma1}. $FE^s_\Gamma(M,[t_1,t_2],D^{}_{\pm})$ is a general Hilbert $\Gamma$-module because the $\Gamma$-isomorphism implies a Hilbert space structure and the space admits a left action representation from $H^s_\Gamma(\spinb^{\pm}_{}(\Sigma_t))$ for each time $t$. \Cref{prophilbertgammamodules} (2) finally ensures that $FE^s_\Gamma(M,[t_1,t_2],D^{}_{\pm})$ becomes a (free) Hilbert $\Gamma$-module. A similar reasoning can be applied to $FE^s_{\Gamma}\left(M,\kernel{D^{}_{\pm}}\right)$. 
\end{proof}

\subsection{Wave-evolution operator in $\Gamma$-setting}\label{chap:wellposgamma-sec:waveevolgamma}

As in the general setting we can extract several Dirac-wave evolution operators from \Cref{homivpwellgammatempcomp}.
\begin{defi}\label{defevopgamma}
For a globally hyperbolic manifold $M$ with time domain $[t_1,t_2]$ the (Dirac-)wave evolution operators for positive and negative chirality are the following isomorphisms between free Hilbert $\Gamma$-modules
\begin{eqnarray*}
Q(t_2,t_1):=\rest{t_2}\circ(\rest{t_1})^{-1} &:& H^s_\Gamma(\spinb^{+}(\Sigma_{1}))\,\,\rightarrow\,\,H^s_\Gamma(\spinb^{+}(\Sigma_{2}))\,\, , \\
\tilde{Q}(t_2,t_1):=\rest{t_2}\circ(\rest{t_1})^{-1}&:&H^s_\Gamma(\spinb^{-}(\Sigma_{1}))\,\,\rightarrow\,\,H^s_\Gamma(\spinb^{-}(\Sigma_{2}))\,\, . 
\end{eqnarray*}
\end{defi}

These operators share the same properties as the wave evolution operators, presented in \cite[Lem.5.2]{OD1}. We want to point out the $\Gamma$-version of \Cref{Qfourier}. 

\begin{lem}\label{propsofpropgamma}
For any $s \in \R$ and $t_1,t_2 \in [t_1,t_2]$ the following holds:
\begin{equation*}
\begin{split}
Q^{} &\in \FIO{0}_\Gamma(\Sigma_{1},\Sigma_{2};\mathsf{C}'_{1\rightarrow 2};\Hom(\spinb^{+}_{}(\Sigma_1),\spinb^{+}_{}(\Sigma_2)))\quad, \\
\tilde{Q}^{} &\in \FIO{0}_\Gamma(\Sigma_{1},\Sigma_{2};\mathsf{C}'_{1\rightarrow 2};\Hom(\spinb^{+}_{}(\Sigma_1),\spinb^{+}_{}(\Sigma_2)))
\end{split}
\end{equation*}
as in \Cref{Qfourier}.
\end{lem}
\begin{proof}
The main steps have been already proven for \Cref{Qfourier}. It is left to show that the Dirac-wave evolution operators are $\Gamma$-invariant, i.e. intertwine the $\Gamma$-actions on $H^s_\Gamma(\spinb^{\pm}_{}(\Sigma_{1}))$ and $H^s_\Gamma(\spinb^{\pm}_{}(\Sigma_{2}))$. As the restriction operator is $\Gamma$-invariant and bijective on $H^s_\Gamma(\spinb^{+}_{}(\Sigma_t))$ for each $t\in [t_1,t_2]$, it implies
\begin{equation*}
L^{\spinb^{+}(\Sigma_2)}_\gamma Q(t_2,t_1)\rest{t_1}=Q(t_2,t_1) L^{\spinb^{+}(\Sigma_1)}_\gamma \rest{t_1} \\
\end{equation*}
and thus the claim; the other case follows due to the same reasoning.
\end{proof}

\section{$\Gamma$-Fredholmness of wave evolution operators}\label{chap:gammaevol}

We present regularity and Fredholmness of $Q$ and $\tilde{Q}$ under boundary conditions which decompose the $L^2$-spaces of spinors on $\Sigma_1$ and $\Sigma_2$. We focus on $\Gamma$-Fredholmness under ordinary (a)APS boundary conditions. 

\subsection{Generalised (a)APS-boundary conditions}\label{chap:evolgammafred-sec:gammagaAPS}

In order to introduce boundary conditions, we need some kind of product structure near the boundary. The collar neighbourhood theorem states that the manifold is diffeomorphic to a product structure near each Cauchy boundary. The metric near the boundary $\bound M= \Sigma_{1}\sqcup \Sigma_2$ can be deformed in such a way that it becomes \textit{ultra-static}: $\met=-\differ t^{\otimes 2}+\met_{t_j}$ near $\Sigma_j$ for $j\in \SET{1,2}$; and each mean curvature $H_{t_j}$ of the spacelike boundary hypersurfaces is vanishing identically. Both Dirac operators are then given by
\begin{equation}\label{diracprodneigh}
D^{}_{\pm}\vert_{\Sigma_j}:= \upbeta\left(\partial_t \mp\Imag A_j \right)=-\upbeta\left(-\partial_t \pm\Imag A_j \right)
\end{equation}
along $\Sigma_j$ with past-directed timelike vector $\mathsf{v}=-\partial_t$, $\upbeta=\cliff{\mathsf{v}}$ and $A_j=A_{t_j}$. The Riemannian hypersurface Dirac operators $A_{1}$ and $A_{2}$ are essentially self-adjoint and the real-valued spectrum of their unique self-adjoint extensions, still denoted as $A_{1}$ and $A_{2}$, decomposes disjointly into a point and continuous spectrum. The eigenspaces for eigenvalues in the point spectrum are orthogonal to each other, but these spaces have in general infinitely many dimensions and the eigensections of the continuous spectrum are smooth, but not square-integrable. \\
\\
Recapitulating our general assumptions, the hypersurface Dirac operators $A_t$ along each slice are $\Gamma$-invariant differential operators of first order: $A_t\in \Diff{1}{\Gamma}(\spinb^{\pm}_{}(\Sigma_t))$ for each $t \in [t_1,t_2]$. Hence they can be viewed as properly supported $\Gamma$-invariant pseudo-differential operators of order $1$. Moreover, they are elliptic and essentially self-adjoint such that $A_t^2$ is an elliptic, positive and properly supported $\Gamma$-invariant pseudo-differential operator for each $t\in [t_1,t_2]$. 
In order to apply \Cref{projsreggenop}, we need to analyse the principal symbol of $A_t$, given by
\vspace*{-1em}
\begin{equation} \label{princsymbA}
{\sigma}_1(A_t)(w,\rho) = \mp \upbeta \Cliff{t}{\rho^\sharp}
\end{equation}
for $w\in \Sigma_t$ and $\rho\in T^\ast\Sigma_t$. The sign depends on the choice of the chirality and the sharp isomorphism is applied with respect to the Riemannian metric $\met_t$. The principal symbol of $A_t^2$ becomes 
\begin{equation}\label{princsymbAsquared}
{\sigma}_2(A^\ast_t\circ A_t)(w,\rho)={\sigma}_2(A^2_t)(w,\rho)=\upbeta \Cliff{t}{\rho^\sharp}\upbeta \Cliff{t}{\rho^\sharp}=\met_{t}(\rho^\sharp,\rho^\sharp)\Iop{\spinb(\Sigma_t)} 
\end{equation} 
which is positive definite. \Cref{projsreggenop} and \Cref{remarkssmoothing} (7) imply that $P_I(t):=\Chi_I(A_t)$ is well-defined for all $t\in [t_1,t_2]$ and all measurable intervals $I\subset \R$; we write  
\begin{equation}\label{gammaapsprojectors}
P_{\pm}(t):=\frac{1}{2}\left(\Iop{}\pm A_t\circ\vert A_t \vert^{-1} - P_{0}(A_t)\right) \in S\ydo{0}{\Gamma}(\spinb^{\pm}(\Sigma)) \quad;
\end{equation} 
if $I$ is bounded, we have $P_I(t)\in S\ydo{-\infty}{\Gamma}(\spinb^{\pm}(\Sigma_t))$.  
\clef{princsymbA}, \clef{princsymbAsquared}, \clef{princsymbcomplexpower} and \clef{princsymprojgen} then imply
\begin{equation*}
{\sigma}_{-1}\left(\absval{A_t}^{-1}\right)(w,\rho)=\left({\sigma}_2(A_t^\ast A_t)(w,\rho)\right)^{-\frac{1}{2}}=\left(\met_{t}\vert_{w}(\rho^\sharp,\rho^\sharp)\right)^{-\frac{1}{2}}=(\norm{\rho}{\met_{t}(w)})^{-1}\Iop{\spinb(\Sigma_t)}
\end{equation*}
and thus
\begin{equation}\label{princsymbproj}
{\sigma}_0(P_{\gtrless 0}(t))(w,\rho)=\frac{1}{2}\left(1\mp(\norm{\rho}{\met_{t}(w)})^{-1}\upbeta \Cliff{t}{\rho^\sharp}\right)\quad .
\end{equation}
The projectors $P_{\gtrless a}(t):=P_{\gtrless a}(A_t)$, $P_{\geq a}(t):=P_{\geq a}(A_tt)$, $P_{\leq a}(t):=P_{\leq a}(A_t)$ and $P^\perp_I(t):=P^\perp_I(A_t)$ are then also defined as in \Cref{chap:specflow-sec:gammaproj} for any $a\in \R$ as well as $I\subset \R$ and satisfy
\begin{equation}\label{sregprojAtanycut}
P_{\gtrless a}(t), P_{\geq a}(t), P_{\leq a}(t), P^\perp_I(t) \in S\ydo{0}{\Gamma}(\spinb^{\pm}(\Sigma_t))\quad.
\end{equation}
After clarifying these details beforehand, we are able to introduce the boundary conditions of our interest. For any $a_1,a_2 \in \R$ the \textit{generalised Atiyah-Patodi-Singer (gAPS) boundary conditions} are defined as follows:
\begin{equation}\label{gapsbc}
\begin{array}{ccl}
P_{\intervallro{a_1}{\infty}{}}(t_1)(u\vert_{\Sigma_{1}}) &=& 0 \\
P_{\intervalllo{-\infty}{a_2}{}}(t_2)(u\vert_{\Sigma_{2}}) &=& 0 \\
\text{for positive chirality} && 
\end{array}
\,\,\quad\text{and}\,\,\quad
\begin{array}{ccl}
P_{\intervallo{-\infty}{a_1}{}}(t_1)(u\vert_{\Sigma_{1}}) &=& 0 \\
P_{\intervallo{a_2}{\infty}{}}(t_2)(u\vert_{\Sigma_{2}}) &=& 0 \\
\text{for negative chirality} && 
\end{array}
\quad.
\end{equation}
Another set of boundary conditions are \textit{generalised anti Atiyah-Patodi-Singer (gaAPS) boundary conditions} which are orthogonal to the gAPS boundary conditions:
\begin{equation}\label{gaapsbc}
\begin{array}{ccl}
P_{\intervallo{-\infty}{a_1}{}}(t_1)(u\vert_{\Sigma_{1}}) &=& 0 \\
P_{\intervallo{a_2}{\infty}{}}(t_2)(u\vert_{\Sigma_{2}}) &=& 0 \\
\text{for positive chirality} &&
\end{array}
\,\,\quad\text{and}\,\,\quad
\begin{array}{ccl}
P_{\intervallro{a_1}{\infty}{}}(t_1)(u\vert_{\Sigma_{1}}) &=& 0 \\
P_{\intervalllo{-\infty}{a_2}{}}(t_2)(u\vert_{\Sigma_{2}}) &=& 0 \\
\text{for negative chirality} && 
\end{array}
\quad.
\end{equation}
The boundary conditions for the negative chirality are chosen to be the adjoint boundary conditions for positive chirality. These boundary conditions induce orthogonal splittings of $L^2$-spaces. We denote the range of the projections on $L^2$-spaces with $L^2_{I}(\spinb^{\pm}_{}(\Sigma_j)):=P_{I}(t_j)[L^2_{}(\spinb^{\pm}_{}(\Sigma_j))]$ and later also with $\range{P_I(t_j)}$; then we can decompose as follows:
\newpage
%
%
\begin{equation}\label{genL2splitgamma}
\begin{array}{ccl}
L^2(\spinb_{}^{+}(\Sigma_1))&=&L^2_{\intervallro{a_1}{\infty}{}}(\spinb_{}^{+}(\Sigma_1))\oplus L^2_{\intervallo{-\infty}{a_1}{}}(\spinb_{}^{+}(\Sigma_1)) \\
L^2(\spinb_{}^{+}(\Sigma_2))&=&L^2_{\intervallo{a_2}{\infty}{}}(\spinb_{}^{+}(\Sigma_2))\oplus L^2_{\intervalllo{-\infty}{a_2}{}}(\spinb_{}^{+}(\Sigma_2)) \\
L^2(\spinb_{}^{-}(\Sigma_1))&=&L^2_{\intervallro{a_1}{\infty}{}}(\spinb_{}^{-}(\Sigma_1))\oplus L^2_{\intervallo{-\infty}{a_1}{}}(\spinb_{}^{-}(\Sigma_1)) \\
L^2(\spinb_{}^{-}(\Sigma_2))&=&L^2_{\intervallo{a_2}{\infty}{}}(\spinb_{}^{-}(\Sigma_2))\oplus L^2_{\intervalllo{-\infty}{a_2}{}}(\spinb_{}^{-}(\Sigma_2)) 
\end{array}
\quad.
\end{equation} 
Because of $\range{P_I(t)}=\kernel{\Iop{}-P_I(t)}$, all subspaces in \clef{genL2splitgamma} are closed for all $a_1,a_2$. Since $P_I(t)$ is s-regular for all $I\subset \R$, it is a $\Gamma$-morphism on $L^2$-spaces in regards to \Cref{remarkssmoothing} (4) such that all the ranges and hence all subspaces in the orthogonal splittings are projective Hilbert $\Gamma$-modules. A further investigation shows that these spaces are even free Hilbert $\Gamma$-modules, i.e. there exists a unitary $\Gamma$-isomorphism $L^2_I(N,E)\cong\ell^2(\Gamma)\otimes L^2_I(N_\Gamma,E_\Gamma)$ for any Hermitian $\Gamma$-vector bundle $E \rightarrow N$ over a $\Gamma$-manifold $N$. This follows from the restriction of the isomorphism \clef{unitarymapl2} to $L^2_I(N,E)$, the $\Gamma$-invariance of the projections $P_I$, coinciding with projectors $\underline{P}_I$ on $L^2(N_\Gamma,E_\Gamma)$, and the density of the fundamental domain in the base manifold as the smallest representative for the $\Gamma$-action.\\
\\
We also introduce projections with domains on the range of another projector: let $a,b \in \R$ and $\blacktriangle, \blacktriangledown$ any symbol in $\SET{<,>,\leq,\geq}$; we define as \textit{restricted projectors} the maps
\begin{equation}\label{restrictedprojections}
\overline{P}^{\blacktriangle a}_{\blacktriangledown b}:=P_{\blacktriangledown b}\,:\,\range{P_{\blacktriangle a}}\rightarrow\range{P_{\blacktriangledown b}} .
\end{equation}
One is tempted to write $P_{[a,b]}$ for $\overline{P}^{\blacktriangle a}_{\blacktriangledown b}$ for $a<b$. However, the former is a bounded operator on $L^2$-spaces, but in general not necessarily Fredholm. The latter operator is the restriction of the former to a closed subspace of $L^2$. It turns out that \clef{restrictedprojections} is in fact $\Gamma$-Fredholm under certain conditions.
\begin{lem}\label{fredrestrictedprojections}
Let $a,b \in \R$, $A$ as in \Cref{projsreggenop} and $P_{\blacktriangle a}$, $P_{\blacktriangledown b}$ spectral projections of $A$ in $L^2$; if $ \codim_\Gamma\left(\range{P_{\blacktriangledown b}} \cap \range{P_{\blacktriangle a}}\right)$ and $\dim_\Gamma\left(\range{P_{\blacktriangle a}}\cap \range{(P_{\blacktriangledown b})^\perp}\right)$ are finite, then $\overline{P}^{\blacktriangle a}_{\blacktriangledown b} \in \mathscr{F}_\Gamma(\range{P_{\blacktriangle a}},\range{P_{\blacktriangledown b}})$ with
\begin{equation*}
\Index_\Gamma(\overline{P}^{\blacktriangle a}_{\blacktriangledown b})=
\dim_\Gamma\left(\range{P_{\blacktriangle a}}\cap \range{(P_{\blacktriangledown b})^\perp}\right) - \codim_\Gamma\left(\range{P_{\blacktriangledown b}} \cap \range{P_{\blacktriangle a}}\right)
\end{equation*}
where 
\begin{equation*}
\codim_\Gamma\left(\range{P_{\blacktriangledown b}}\right)=\dim_\Gamma\left(\quotspace{\range{P_{\blacktriangledown b}}}{\range{P_{\blacktriangledown b}}\cap \range{P_{\blacktriangle a}}}\right) \quad.
\end{equation*}
\end{lem}
\begin{proof}
The range and the kernel of $\overline{P}^{\blacktriangle a}_{\blacktriangledown b}$ can be directly calculated: 
\begin{equation*}
\kernel{\overline{P}^{\blacktriangle a}_{\blacktriangledown b}}=\range{(P_{\blacktriangledown b})^\perp}\cap \range{P_{\blacktriangle a}} \quad\text{and}\quad\range{\overline{P}^{\blacktriangle a}_{\blacktriangledown b}}=\range{P_{\blacktriangledown b}} \cap \range{P_{\blacktriangle a}} \quad.
\end{equation*}
The kernel has finite $\Gamma$-dimension by preassumption. Since all ranges are closed subsets of $L^2$, the range of $\overline{P}^{\blacktriangle a}_{\blacktriangledown b}$ as intersection of closed sets does as well and has finite $\Gamma$-codimension by assumption.\\
\\
In order to calculate the $\Gamma$-index, it is left to compute the kernel of the formal adjoint of $\overline{P}^{\blacktriangle a}_{\blacktriangledown b}$. As $\overline{P}^{\blacktriangle a}_{\blacktriangledown b}$ has closed range by construction, the closed range theorem implies that exactly this kernel is given by
$$ (\range{P_{\blacktriangledown b}}\cap \range{P_{\blacktriangle a}})^\perp $$
where the orthogonal complement has to be taken in $\range{P_{\blacktriangledown b}}$ as ambient Hilbert space. The intersection of Hilbert $\Gamma$-modules is in general not a Hilbert $\Gamma$-module. But since this intersection of ranges results again in a range of a spectral projection $P_{I(a,b)}:=\range{P_{\blacktriangledown b}}\cap \range{P_{\blacktriangle a}}, I(a,b)\subset \R$, $\range{P_{I(a,b)}}$ is again a projective Hilbert $\Gamma$-module. \Cref{helpinglemma1} implies further that the orthogonal complement of $\range{P_{I(a,b)}}$ is a projective Hilbert $\Gamma$-module. The same holds true for the quotient $\range{P}_{\blacktriangledown b}/\range{P_{I(a,b)}}$.  
As $\range{P_{I(a,b)}}$ is closed, the canonical isomorphy between the quotient and the orthogonal complement implies a unitary $\Gamma$-isomorphism such that
$$\dim_\Gamma\left(\range{P^\perp_{I(a,b)}}\right)=\dim_\Gamma\left(\quotspace{\range{P}_{\blacktriangledown b}}{\range{P_{I(a,b)}}}\right)$$
according to \Cref{propgammadim} (5). The right-hand side coincides with the $\Gamma$-codimension of $\range{P_{I(a,b)}}$ which shows the claimed $\Gamma$-index formula.
\end{proof}

\subsection{$\Gamma$-Fredholmness for (a)APS boundary conditions}\label{chap:evolgammafred-sec:gammafredQ}

We concentrate our analysis to (a)APS boundary conditions and consider the decomposition of $L^2$-spaces \clef{genL2splitgamma} for $a_1=0=a_2$. \\
\\
If we apply the splittings \clef{genL2splitgamma} due to APS and aAPS boundary conditions for positive chirality to $Q(t_2,t_1)$, it allows us to rewrite the operator as a ($2\times 2$)-matrix:
\begin{equation}\label{Qmatrix}
Q(t_2,t_1)=\left(\begin{matrix}
Q_{++}(t_2,t_1) & Q_{+-}(t_2,t_1) \\
Q_{-+}(t_2,t_1) & Q_{--}(t_2,t_1)
\end{matrix}\right) \quad .
\end{equation}
We define the entries as maps of the form 
\begin{equation}\label{Qmatrixpos}
\begin{array}{ccccc}
Q_{++}(t_2,t_1)&:& L^2_{}(\spinb^{+}(\Sigma_1)) &\rightarrow & L^2_{(0,\infty)}(\spinb^{+}(\Sigma_2))\\
&& u &\mapsto & P_{>0}(t_2)\circ Q(t_2,t_1)\circ P_{\geq 0}(t_1)u \\
Q_{--}(t_2,t_1)&:& L^2_{}(\spinb^{+}(\Sigma_1)) &\rightarrow & L^2_{(-\infty,0]}(\spinb^{+}(\Sigma_2))\\
&& u &\mapsto & P_{\leq 0}(t_2)\circ Q(t_2,t_1)\circ P_{< 0}(t_1)u \\
Q_{+-}(t_2,t_1)&:& L^2_{}(\spinb^{+}(\Sigma_1)) &\rightarrow & L^2_{(0,\infty)}(\spinb^{+}(\Sigma_2))\\
&& u &\mapsto & P_{>0}(t_2)\circ Q(t_2,t_1)\circ P_{< 0}(t_1)u \\
Q_{-+}(t_2,t_1) &:& L^2_{}(\spinb^{+}(\Sigma_1)) &\rightarrow & L^2_{(-\infty,0]}(\spinb^{+}(\Sigma_2))\\
&& u &\mapsto & P_{\leq 0}(t_2)\circ Q(t_2,t_1)\circ P_{\geq 0}(t_1)u
\end{array}
\quad.
\end{equation}
These matrix entries will be referred on as \textit{spectral} or rather \textit{matrix entries}. We call $Q_{\pm\pm}$ the \textit{diagonal}\bnote{F34} and $Q_{\pm\mp}$ the \textit{off-diagonal entries} of $Q$. We will also need these operators, acting as maps on spectral subspaces of $L^2(\spinb^{+}(\Sigma_1))$:
\begin{equation}\label{Qmatrixposrest}
\begin{array}{lcccl}
Q_{++}(t_2,t_1)= P_{> 0}(t_2)\circ Q(t_2,t_1)&:&  L^2_{[0,\infty)}(\spinb^{+}(\Sigma_1)) &\rightarrow & L^2_{(0,\infty)}(\spinb^{+}(\Sigma_2)) \\
Q_{--}(t_2,t_1)= P_{\leq 0}(t_2)\circ Q(t_2,t_1) &:& L^2_{(-\infty,0)}(\spinb^{+}(\Sigma_1)) &\rightarrow & L^2_{(-\infty,0]}(\spinb^{+}(\Sigma_2)) \\
Q_{+-}(t_2,t_1)= P_{> 0}(t_2)\circ Q(t_2,t_1) &:& L^2_{(-\infty,0)}(\spinb^{+}(\Sigma_1))  &\rightarrow & L^2_{(0,\infty)}(\spinb^{+}(\Sigma_2)) \\
Q_{-+}(t_2,t_1)= P_{\leq 0}(t_2)\circ Q(t_2,t_1) &:& L^2_{[0,\infty)}(\spinb^{+}(\Sigma_1)) &\rightarrow & L^2_{(-\infty,0]}(\spinb^{+}(\Sigma_2))
\end{array}
\, .
\end{equation}
Any spectral entry is a $\Gamma$-morphism. As the spectral subspaces of $L^2$ are closed free Hilbert $\Gamma$-modules, all entries in \clef{Qmatrixposrest} have closed range. The commuting with the left action representation is clear. 
The unitarity property on $L^2$-sections of $Q$ implies that the off-diagonal entries are isomorphisms between the kernels of the diagonal entries and their adjoints.
\begin{lem}\label{kernelisoQ}
The operators $Q_{+-}(t_2,t_1)$ and $Q_{-+}(t_2,t_1)$ restrict to $\Gamma$-isomorphisms 
\begin{equation*}
\begin{array}{ccccc}
Q_{+-}(t_2,t_2)&:&\kernel{Q_{--}(t_2,t_1)}&\rightarrow& \kernel{(Q_{++}(t_2,t_1))^{\ast}} \\[2pt]
Q_{-+}(t_2,t_2)&:&\kernel{Q_{++}(t_2,t_1)}&\rightarrow& \kernel{(Q_{--}(t_2,t_1))^{\ast}} 
\end{array}
\quad .
\end{equation*}
\end{lem}

The purely algebraic proof can be taken from \cite[Lem.2.5]{BaerStroh}.\\
\\
Because of its later importance, we write out $(Q(t_2,t_1))^\ast Q(t_2,t_1)=\Iop{L^2_{}(\spinb^{+}(\Sigma_{1}))}$ and $Q(t_2,t_1)(Q(t_2,t_1))^{\ast}=\Iop{L^2_{}(\spinb^{+}(\Sigma_{2}))}$ and get the following system of equations for the non-trivial entries:
\begin{equation}\label{Qsys1}
\begin{array}{ccl}
(Q_{++}(t_2,t_1))^\ast Q_{++}(t_2,t_1)+(Q_{-+}(t_2,t_1))^\ast Q_{-+}(t_2,t_1)&=& \Iop{} \\
(Q_{+-}(t_2,t_1))^\ast Q_{+-}(t_2,t_1)+(Q_{--}(t_2,t_1))^\ast Q_{--}(t_2,t_1)&=& \Iop{} 
\end{array}
\end{equation}
and  
\begin{equation}\label{Qsys2}
\begin{array}{ccl}
Q_{++}(t_2,t_1)(Q_{++}(t_2,t_1))^\ast+Q_{+-}(t_2,t_1))(Q_{+-}(t_2,t_1))^\ast &=& \Iop{} \\
Q_{-+}(t_2,t_1)(Q_{-+}(t_2,t_1))^\ast+Q_{--}(t_2,t_1))(Q_{--}(t_2,t_1))^\ast&=& \Iop{}
\end{array}
\quad.
\end{equation}
If we apply \clef{genL2splitgamma} for $a_1=0=a_2$ and negative chirality to $\tilde{Q}(t_2,t_1)$, we get a similar description as (2x2)-matrix:
\begin{equation}\label{Qnegmatrix}
\tilde{Q}(t_2,t_1)=\left(\begin{matrix}
\tilde{Q}_{++}(t_2,t_1) & \tilde{Q}_{+-}(t_2,t_1) \\
\tilde{Q}_{-+}(t_2,t_1) & \tilde{Q}_{--}(t_2,t_1)
\end{matrix}\right) \quad 
\end{equation}
where the matrix entries are analogously defined via
\begin{equation}\label{Qmatrixneg}
\begin{array}{ccccc}
\tilde{Q}_{++}(t_2,t_1)&:& L^2_{}(\spinb^{-}(\Sigma_1)) &\rightarrow & L^2_{(0,\infty)}(\spinb^{-}(\Sigma_2))\\
&& u &\mapsto & P_{>0}(t_2)\circ \tilde{Q}(t_2,t_1)\circ P_{\geq 0}(t_1)u \\
\tilde{Q}_{--}(t_2,t_1)&:& L^2_{}(\spinb^{-}(\Sigma_1)) &\rightarrow & L^2_{(-\infty,0]}(\spinb^{-}(\Sigma_2))\\
&& u &\mapsto & P_{\leq 0}(t_2)\circ \tilde{Q}(t_2,t_1)\circ P_{< 0}(t_1)u \\
\tilde{Q}_{+-}(t_2,t_1)&:& L^2_{}(\spinb^{-}(\Sigma_1)) &\rightarrow & L^2_{(0,\infty)}(\spinb^{-}(\Sigma_2))\\
&& u &\mapsto & P_{>0}(t_2)\circ \tilde{Q}(t_2,t_1)\circ P_{< 0}(t_1)u \\
\tilde{Q}_{-+}(t_2,t_1) &:& L^2_{}(\spinb^{-}(\Sigma_1)) &\rightarrow & L^2_{(-\infty,0]}(\spinb^{-}(\Sigma_2))\\
&& u &\mapsto & P_{\leq 0}(t_2)\circ \tilde{Q}(t_2,t_1)\circ P_{\geq 0}(t_1)u \quad.
\end{array}
\end{equation}
As in the case for positive chirality, these matrix entries are $\Gamma$-morphisms with closed ranges. Since $\tilde{Q}$ is equally a unitary $\Gamma$-morphism on $L^2$, its matrix representation \clef{Qnegmatrix} implies a similar set of equations in \clef{Qsys1} and \clef{Qsys2} with $Q$ replaced by $\tilde{Q}$. \Cref{kernelisoQ} for negative chirality can be proven equally.\\
\\  
We have seen how unitarity of the Dirac-wave evolution operators as $\Gamma$-morphisms on $L^2$ carry over to the matrix entries in \clef{Qmatrix} and \clef{Qnegmatrix}. We now want to clarify how their regularity as Fourier integral operator transfers to the spectral entries. We will show that these compositions are s-regular Fourier integral operators of the same order and with same canonical relation.

\begin{prop}\label{propQposneg}
The operators in \clef{Qmatrixpos} and \clef{Qmatrixneg} are $\Gamma$-invariant Fourier integral operators of order $0$ modulo s-smoothing remainders, i.e.
\begin{itemize}
\item[(1)] $Q_{\pm\pm}(t_2,t_1), Q_{\pm\mp}(t_2,t_1) \in S\FIO{0}_\Gamma(\Sigma_1,\Sigma_2;\mathsf{C}'_{1\rightarrow 2};\Hom(\spinb^{+}(\Sigma_1),\spinb^{+}(\Sigma_2)))$;
\item[(2)] $\tilde{Q}_{\pm\pm}(t_2,t_1), \tilde{Q}_{\pm\mp}(t_2,t_1) \in S\FIO{0}_\Gamma(\Sigma_1,\Sigma_2;\mathsf{C}'_{1\rightarrow 2};\Hom(\spinb^{-}(\Sigma_1),\spinb^{-}(\Sigma_2)))$. 
\end{itemize}
\end{prop} 
\begin{proof}
We prove (1) of this assertion as the following arguments can be transferred to the proof of (2) by replacing $Q$ with $\tilde{Q}$.\\
\\
$P_{\pm}(t_j)$ and $P_{>0}(t_j)$ are elements in $S\ydo{0}{\Gamma}(\Sigma_j,\spinb^{+}(\Sigma_j))$. Due to \Cref{remslabel1} (i), they can be decomposed as $P_{\pm}(t_j)=p_{\pm}(t_j)+r_{\pm}(t_j)$ where $r_{\pm}(t_j)\in S\ydo{-\infty}{\Gamma}(\Sigma_j,\spinb^{+}(\Sigma_j))$ and $p_{\pm}(t_j)$ are properly supported classical $\Gamma$-pseudo-differential operators of order $0$. The other projections $P_{\geq 0}(t_j)$ and $P_{\leq 0}(t_j)$ differ from $P_{>0}(t_j)$ respectively $P_{<0}(t_j)$ in a s-smoothing projector $P_{0}(t_j)$. We write $\tilde{r}_{\pm}(t_j)$ for $r_{\pm}(t_j)+P_{0}(t_j)$ such that 
\begin{equation*}
P_{\geq 0}(t_j) =p_{+}(t_j)+\tilde{r}_{+}(t_j)\quad \text{and} \quad P_{\leq 0}(t_j) =p_{-}(t_j)+\tilde{r}_{-}(t_j) \quad.
\end{equation*}
Each spectral entry in \clef{Qmatrixpos} can be split up into a sum of a properly supported $\Gamma$-Fourier integral operator and a s-smoothing operator:  
\begin{align*}
Q_{\pm\pm}(t_2,t_1)&= q_{\pm\pm}(t_2,t_1)+ R_{\pm\pm}(t_2,t_1)\quad \text{with} \quad q_{\pm\pm}(t_2,t_1):=p_{\pm}(t_2)\circ Q\circ p_{\pm}(t_1) \\
Q_{\pm\mp}(t_2,t_1)&= q_{\pm\mp}(t_2,t_1)+ R_{\pm\mp}(t_2,t_1)\quad \text{with} \quad q_{\pm\mp}(t_2,t_1):=p_{\pm}(t_2)\circ Q\circ p_{\mp}(t_1)  .
\end{align*} 
The remainders $R_{\pm\pm}(t_2,t_1)$ and $R_{\pm\mp}(t_2,t_1)$ are sums of triple compositions between s-smoothing pseudo-differential operators and the properly supported Fourier integral operator $Q$. Compositions, containing the properly supported pseudo-differential operator $p_{\pm}$ of order $0$, are well-defined and s-smoothing because composing with $Q$ is well-defined and the second composition with either $r_{\pm}$ or $\tilde{r}_{\pm}$ gives a s-smoothing operator. The remaining triple composition of the form $\tilde{r}\circ Q \circ r$ is equally s-smoothing wherefore $R_{\pm\pm}$ and $R_{\pm \mp}$ are s-smoothing.\\
\\
The triple compositions $q_{\pm\pm}$ and $q_{\pm\mp}$ of properly supported operators are properly supported Fourier integral operators of order $0$ with canonical relation \clef{Qcanrel} since the composition of $N^\ast\mathrm{diag}(\Sigma_j)$ and $\mathsf{C}_{1\rightarrow 2}$ is proper, transversal, and results in $\mathsf{C}_{1\rightarrow 2}$:
\begin{equation}\label{specentrypropsup}
q_{\pm\pm}(t_2,t_1), q_{\pm\mp}(t_2,t_1) \in \FIO{0}_{\Gamma,\mathsf{prop}}(\Sigma_2;\mathsf{C}'_{1\rightarrow 2};\Hom(\spinb^{+}(\Sigma_1),\spinb^{+}(\Sigma_2)))\quad.
\end{equation}
The intertwining of the left action representations is clear since every part in the composition is a $\Gamma$-invariant operator on its own right.
\end{proof}
The principal symbols of the compositions can be calculated by multiplication since the character of $\mathsf{C}_{1\rightarrow 2}$, being a disjoint union of graphs of symplectomorphisms, has been preserved. In doing so, we observe that we can improve the result for the off-diagonal entries.
\begin{prop}\label{propQposnegfredimprov}
\begin{itemize}
\item[]
\item[(1)] $Q_{\pm\mp}(t_2,t_1) \in S\FIO{-1}_\Gamma(\Sigma_1,\Sigma_2;\mathsf{C}'_{1\rightarrow 2};\Hom(\spinb^{+}(\Sigma_1),\spinb^{+}(\Sigma_2)))$;
\item[(2)] $\tilde{Q}_{\pm\mp}(t_2,t_1) \in S\FIO{-1}_\Gamma(\Sigma_1,\Sigma_2;\mathsf{C}'_{1\rightarrow 2};\Hom(\spinb^{-}(\Sigma_1),\spinb^{-}(\Sigma_2)))$.
\end{itemize}
\end{prop} 
\begin{proof}
Let $q_{\pm\mp}(t_2,t_1)$ be the properly supported part of $Q_{\pm\mp}(t_2,t_1)$ from \clef{specentrypropsup}; their principal symbols are the same up to smoothing terms; the same holds for the symbol of the s-regular projectors. The fact, that the resulting canonical relation and thus its nature as graph of a symplectomorphism is preserved under the composition, allows us to compute the principal symbol of the compositions by just composing the principal symbol of each occuring operator (see \cite[Eq.A.1]{OD1}):     
\begin{equation*}
\begin{split}
\mathpzc{q}_{\pm\mp}(x,\xi_{\pm};y,\eta)&:={\sigma}_0(p_{\pm})(x,\xi_{\pm})\circ {\sigma}_0(Q)(x,\xi_{\pm};y,\eta) \circ {\sigma}_0(p_{\mp})(y,\eta)
\end{split}
\end{equation*}
for $(x,\xi_{\pm})\in T^\ast_x \Sigma_2$ and $(y,\eta) \in T^\ast_y\Sigma_2$. 
The calculation and vanishing of the principal symbols $\mathpzc{q}_{\pm\mp}(x,\xi_{\pm};y,\eta)$ can be performed like in the proof of \cite[Lem.2.6]{BaerStroh}. The calculation of the principal symbols with respect to the matrix entries of $\tilde{Q}$ can be performed likewise. Hence, the principal symbols of $Q_{\pm\mp}(t_2,t_1)$ and $\tilde{Q}_{\pm\mp}(t_2,t_1)$ are identically vanishing. The exact sequence property \cite[Lem.A.4]{OD1} then implies that the order of the properly supported part is $(-1)$. 
\end{proof} 

The $\Gamma$-Fredholm property carries over to their matrix entries for (a)APS boundary conditions. 
\begin{theo}\label{propQposnegfred}
$Q_{\pm\pm}$ and $\tilde{Q}_{\pm\pm}$ are $\Gamma$-Fredholm as maps from
\begin{equation}\label{Qfredmappingproperty}
\begin{split}
Q_{++}(t_2,t_1) &: L^2_{[0,\infty)}(\spinb^{+}(\Sigma_1)) \,\rightarrow\, L^2_{(0,\infty)}(\spinb^{+}(\Sigma_2))\\
Q_{--}(t_2,t_1) &: L^2_{(-\infty,0)}(\spinb^{+}(\Sigma_1)) \,\rightarrow\, L^2_{(-\infty,0]}(\spinb^{+}(\Sigma_2))\\
\tilde{Q}_{++}(t_2,t_1) &: L^2_{[0,\infty)}(\spinb^{-}(\Sigma_1)) \,\rightarrow\, L^2_{(0,\infty)}(\spinb^{-}(\Sigma_2))\\
\tilde{Q}_{--}(t_2,t_1) &: L^2_{(-\infty,0)}(\spinb^{-}(\Sigma_1)) \,\rightarrow\, L^2_{(-\infty,0]}(\spinb^{-}(\Sigma_2))\\
\end{split}
\end{equation}
with $\Gamma$-indices 
\begin{equation*}
\begin{split}
&\Index_{\Gamma}(Q_{++}(t_2,t_1))=-\Index_\Gamma(Q_{--}(t_2,t_1))\\
\text{and}\quad\quad &\Index_{\Gamma}(\tilde{Q}_{++}(t_2,t_1))=-\Index_\Gamma(\tilde{Q}_{--}(t_2,t_1))\quad.
\end{split}
\end{equation*}
\end{theo} 
\begin{proof}
We show that
\begin{equation*}
\begin{split}
Q_{+-}(t_2,t_1)&\in \mathscr{K}_\Gamma(L^2_{(-\infty,0)}(\spinb^{+}(\Sigma_1)),L^2_{}(\spinb^{+}(\Sigma_2)))\\
\text{and} \quad Q_{-+}(t_2,t_1)&\in \mathscr{K}_\Gamma(L^2_{[0,\infty)}(\spinb^{+}(\Sigma_1)),L^2_{}(\spinb^{+}(\Sigma_2)))\quad.
\end{split}
\end{equation*}
These imply that $Q^\ast_{\pm\mp}$ and the compositions $Q^\ast_{\pm\mp}Q_{\pm\mp}$ are equally $\Gamma$-compact such that $\Gamma$-Fredholmness of each $Q_{\pm\pm}$ directly follows from \clef{Qsys1} and \clef{Qsys2}. The $\Gamma$-Fredholm parametrices are then fully given by the adjoints of $Q_{\pm\pm}$.\\ 
\\
$Q_{\pm\mp}$ are unitarily related to $(\Iop{}\otimes \underline{Q_{\pm\mp}})$ according to the following commutative diagram.
\begin{figure}[H]
\centering
\begin{tikzpicture}

\draw (-2,1) node {$L^2_{I}(\mathpzc{S}^{+}(\Sigma_1))$};
\draw (-2,-2) node {$L^2_{}(\mathpzc{S}^{+}(\Sigma_2))$};
\draw (4.25,1) node {$\ell^2(\Upgamma)\otimes L^2_{I}(\mathpzc{S}^{+}(\Sigma_1/\Upgamma))$};
\draw (4.25,-2) node {$\ell^2(\Upgamma)\otimes L^2(\mathpzc{S}^{+}(\Sigma_2/\Upgamma))$};

\draw[->] (-0.75,1)--(2,1);
\draw[->] (-0.75,-2)--(2,-2);

\draw[->] (-2,0.7)--(-2,-1.7);
\draw[->] (4,0.7)--(4,-1.7);

\draw (-3,-0.5) node {$Q_{\pm\mp}(t_2,t_1)$};
\draw (5.75,-0.5) node {$\mathbbm{1}_{\ell^2(\Upgamma)}\otimes\underline{Q_{\pm\mp}}(t_2,t_1)$};

\draw (0.625,1.2) node {$\cong$};
\draw (0.625,-1.8) node {$\cong$};

\end{tikzpicture}
\caption{Commutative diagram for $Q_{\pm\mp}(t_2,t_1)$.}
\end{figure}
The proof in \cite[Lem.2.6]{BaerStroh} shows that $\underline{Q_{\pm\mp}}(t_2,t_1)$ are compact operators on the compact bases: 
\begin{equation*}
\begin{split}
\underline{Q_{+-}}(t_2,t_1) &\in \mathscr{K}(L^2_{(-\infty,0)}(\spinb^{+}(\Sigma_1/\Gamma)),L^2(\spinb^{+}(\Sigma_2/\Gamma))) \\
\underline{Q_{-+}}(t_2,t_1) & \in \mathscr{K}(L^2_{[0,\infty)}(\spinb^{+}(\Sigma_1/\Gamma)),L^2(\spinb^{+}(\Sigma_2/\Gamma)))\quad.
\end{split}
\end{equation*}
\clef{commdiagammacomp} then implies $\Gamma$-compactness of $Q_{\pm\mp}$. The $\Gamma$-indices then follow from \Cref{kernelisoQ} and \Cref{prophilbertgammamodules} (1) since they imply that there exist unitary $\Gamma$-isomorphisms such that
\begin{equation}\label{qfredstep2}
\dim_\Gamma\kernel{Q_{\pm\pm}}=\dim_\Gamma\kernel{Q^\ast_{\mp\mp}}\quad.
\end{equation}
The index formula in \Cref{propgammafred} (2) completes the proof. The argument carries over to $\tilde{Q}_{\pm\pm}(t_2,t_1)$ with the help of \clef{Qsys1}, \clef{Qsys2} where $Q$ is replaced by $\tilde{Q}$. 
\end{proof}
\begin{rem}
An alternative proof has been published in \cite{OD}. We briefly want to sketch the idea. Each first two equations in \clef{Qsys1} and \clef{Qsys2} show the important observation that $Q^\ast_{++}$ and $Q^\ast_{--}$ can be used as initial parametrices of $Q_{++}$ respectively $Q_{--}$ with remainders of the form $Q^\ast_{\mp\pm}\circ Q_{\mp\pm}$. Ellipticity would usually become an important property in constructing an initial parametrix, but becomes irrelevant here as the unitarity of $Q$ replaces this step. Based on \Cref{propQposneg} (1), a precise analysis of the compositions $Q^\ast_{\pm\mp}\circ Q_{\pm\mp}$ shows that these are $\Gamma$-invariant properly supported pseudo-differential operators of order $(-2)$ with a $\Gamma$-trace class remainder. We can construct an even better parametrix for each operator with a Neumann series argument. The order of the errors then become sufficiently negative. \\
\\
A careful analysis of the regularity and operator character of the constructed parametrices shows that they are properly supported $\Gamma$-Fourier integral operator of order $(-2N)$ with canonical relation $(\mathsf{C}_{1\rightarrow 2})^{-1}$ modulo $\Gamma$-trace class operator. Applying them on $Q_{\pm\pm}$ gives the identity plus a $\Gamma$-trace class remainder and a remainder $R\in \ydo{-2N}{\Gamma,\mathsf{prop}}(\spinb^{+}(\Sigma_{1}))$. If we choose $N > \dim(\Sigma)/2$ and recall $L^2$-boundness of the matrix entries, then \Cref{propgammapseudoprop} (2) already characterises $R$ as $\Gamma$-trace class operator such that the constructed parametrices are in fact suitable in showing $\Gamma$-Fredholmness. 
\end{rem}

\section{$\Gamma$-Fredholmness of the Lorentzian Dirac operator}\label{chap:Atiyah}

We prove Fredholmness of $D^{}_{\pm}$ by relating to the known $\Gamma$-Fredholmness of the diagonal spectral entries of $Q$ and $\tilde{Q}$. 

\subsection{An important Lemma}\label{chap:fredholm-sec:lemma}

We first have to prove a Hilbert $\Gamma$-module version of \cite[Lem.A.1]{BaerBall}.

\begin{lem}\label{funcanagammalemma}
Let $\mathscr{H},\mathscr{H}_1,\mathscr{H}_2$ be (projective) Hilbert $\Gamma$-modules, $A \in \mathscr{B}_\Gamma(\mathscr{H},\mathscr{H}_1)$, and $B \in \mathscr{B}_\Gamma(\mathscr{H},\mathscr{H}_2)$ which is onto; define $C=A\vert_{\kernel{B}}\oplus \Iop{\mathscr{H}_2}$, then
\begin{itemize}
\item[(1)] $\dim_\Gamma\kernel{C}=\dim_\Gamma\kernel{A\oplus B}\,\,;$
\item[(2)] $\range{C}$ is closed if and only if $\range{A\oplus B}$ is closed and 
\begin{equation*}
\codim_{\Gamma}(\range{C})=\codim_{\Gamma}(\range{A\oplus B})\quad;
\end{equation*} 
\item[(3)] $C$ is $\Gamma$-Fredholm if and only if $A\oplus B$ is $\Gamma$-Fredholm and $\Index_{\Gamma}(C)=\Index_{\Gamma}(A\oplus B)$.
\end{itemize}
\end{lem}
\begin{proof}
Since $A$ and $B$ are bounded, also $A\vert_{\kernel{B}}$, $C$ and $A\oplus B$ are bounded 
and their $\Gamma$-invariance follows trivially in each summand. 
One equally concludes from the proof in \cite{BaerBall} that there exist two $\Gamma$-isomorphisms $\mathcal{I}\in \mathscr{B}_\Gamma(\mathscr{H}^{\oplus 2})$ and $\mathcal{J}\in \mathscr{B}_\Gamma(\mathscr{H}_1\oplus \mathscr{H}_2)$ for the splitting $\mathscr{H}=\kernel{B}\oplus\kernel{B}^{\perp}$ such that
\begin{equation}\label{operatorisomorph}
C=\mathcal{J}\circ (A\oplus B)\circ \mathcal{I}=\mathcal{J}\circ
\left(
\begin{matrix}
A\vert_{\kernel{B}} & A\vert_{\kernel{B}^\perp}\\
\zerop{} & B\vert_{\kernel{B}^\perp}
\end{matrix}
\right)\circ \mathcal{I} \quad .
\end{equation}
These are given by 
\begin{equation*}
\mathcal{I}=\Iop{\kernel{B}}\oplus \left(B\vert_{\kernel{B}^\perp}\right)^{-1} \quad \text{and} \quad \mathcal{J}=\left(
\begin{matrix}
\Iop{\mathscr{H}_1} & -A\vert_{\kernel{B}^\perp}\circ \left(B\vert_{\kernel{B}^\perp}\right)^{-1} \\
\zerop{} & \Iop{\mathscr{H}_2}
\end{matrix}
\right) \quad.
\end{equation*}
\begin{itemize}
\item[(1)] The null-spaces of the left- and right-hand side of \clef{operatorisomorph} are projective Hilbert $\Gamma$-modules. \Cref{propgammadim} (4) implies
\begin{equation*}
\dim_\Gamma\kernel{A\vert_{\kernel{B}}}=\dim_\Gamma\kernel{C}=\dim_\Gamma\kernel{\mathcal{J}\circ (A\oplus B)\circ \mathcal{I}} \quad .
\end{equation*}
Because of the equivalence $\kernel{\mathcal{J}\circ (A\oplus B)\circ \mathcal{I}}=\mathcal{I}^{-1}(\kernel{A\oplus B})$ with the $\Gamma$-isomorphism $\mathcal{I}^{-1}$, the map between the projective Hilbert $\Gamma$-modules $\kernel{A\oplus B}$ and $\kernel{\mathcal{J}\circ(A\oplus B)\circ \mathcal{I}}$
%
becomes a $\Gamma$-isomorphism and \Cref{prophilbertgammamodules} (1) then guarantees the existence of a unitary $\Gamma$-isomorphism
$$\kernel{A\oplus B}\cong\kernel{\mathcal{J}\circ(A\oplus B)\circ \mathcal{I}}\quad.$$
The invariance of the $\Gamma$-dimension between unitary isomorphic spaces (\Cref{propgammadim} (5)) shows the claim.
\item[(2)] Since closedness does not involve any property with respect to the group $\Gamma$, the first part of the claim follows as in \cite[Lem.A.1]{BaerBall} from the closed range theorem.
\Cref{helpinglemma1} and closedness imply that 
\begin{equation*}
\quotspace{\left(\mathscr{H}_1\oplus\mathscr{H}_2\right)}{\range{C}}\quad \text{and} \quad \quotspace{\left(\mathscr{H}_1\oplus\mathscr{H}_2\right)}{\range{A\oplus B}}
\end{equation*}  
are projective Hilbert $\Gamma$-submodules. A unitary $\Gamma$-isomorphism between these two spaces needs to be found. Since both quotient spaces are again Hilbert spaces, they are isomorphic to their dual spaces by the Frechet-Riesz theorem. The closed range theorem implies again that the orthogonal complements of the ranges of $C$ and $A\oplus B$, both canonically isomorphic to their quotients in $\mathscr{H}_1\oplus\mathscr{H}_2$, are equal to the null spaces of their adjoint operators $\kernel{C^\ast}$ respectively $\kernel{(A\oplus B)^\ast}$. \\
\\
It is now left to show that the kernels are isomorphic to each other. 
Adjoining \clef{operatorisomorph} yields
\begin{equation*}
C^\ast=\mathcal{I}^\ast\circ (A\oplus B)^\ast \circ \mathcal{J}^\ast
\end{equation*}  
A similar argument as in (1) shows that the kernels of $C^\ast$ and $(A\oplus B)^\ast$ are indeed (unitarily) isomorphic to each other. Thus, the composition with the other isomorphisms, used to reduce the quotients, implies again a topological and thus a unitary $\Gamma$-isomorphism between projective $\Gamma$-submodules which finally leads to the claim and

\begin{equation*}
\begin{split}
\codim_\Gamma(\range{C})&=\dim_\Gamma\range{C}^\perp=\dim_\Gamma\kernel{C^\ast}=\dim_\Gamma\kernel{(A\oplus B)^\ast}\\
&=\dim_\Gamma\range{A\oplus B}^\perp=\codim_\Gamma(\range{A\oplus B})\quad .
\end{split}
\end{equation*}
\item[(3)] $\Gamma$-Fredholmness of $C$ implicates $\dim_\Gamma\kernel{C}<\infty$, $\range{C}$ is closed, and that the $\Gamma$-codimension satisfies $\codim_\Gamma(\range{C})<\infty$. (1) and (2) imply $\dim_\Gamma\kernel{A\oplus B}<\infty$, $\range{A\oplus B}$ is closed and $\codim_\Gamma(\range{A\oplus B})<\infty$ and thus $\Gamma$-Fredholmness of $A\oplus B$. We can on the other hand conclude the $\Gamma$-Fredholmness of $C$ from the assumed $\Gamma$-Fredholmness of $A\oplus B$ as the above arguments are symmetric in terms of both operators. The equivalence of the $\Gamma$-indices follows from the equivalence of the $\Gamma$-dimensions in this proof and (7) of \Cref{propgammafred}:
\begin{align*}
\Index_\Gamma(C)&=\dim_\Gamma\kernel{C}-\dim_\Gamma\kernel{C^\ast}\\
&=\dim_\Gamma\kernel{A\oplus B}-\dim_\Gamma\kernel{(A\oplus B)^\ast}=\Index_\Gamma(A\oplus B)\quad.\qedhere
\end{align*}
\end{itemize}
\end{proof}  

\subsection{$\Gamma$-Fredholmness of $D_{\pm}$ under (a)APS boundary conditions}\label{chap:fredholm-sec:gammafredDplus}

We are now in the position to prove Fredholmness of $D_{\pm}^{}$. We focus on $D^{}=D^{}_{+}$ first and give a detailed proof. Afterwards, we consider $\tilde{D}^{}=D^{}_{-}$.\\
\\ 
We first define finite energy spinors of the Dirac equation which satisfy either APS or aAPS boundary conditions for positive chirality in the $\Gamma$-setting: based on \clef{upgammafesections} (b), we define
\begin{multline*}
FE^s_{\Gamma,\mathrm{APS}}(M,[t_1,t_2],D^{})\\
:=\SET{u \in FE^s_{\Gamma}(M,[t_1,t_2],D^{})\,\Big\vert\,P^{}_{\intervallro{0}{\infty}{}}(t_1)\circ\rest{\Sigma_1}u=0=P^{}_{\intervalllo{-\infty}{0}{}}(t_2)\circ\rest{\Sigma_2}u}
\end{multline*}
\begin{multline*}
FE^s_{\Gamma,\mathrm{aAPS}}(M,[t_1,t_2],D^{})\\:=\SET{u \in FE^s_{\Gamma}(M,[t_1,t_2],D^{})\,\Big\vert\,P^{}_{\intervallo{-\infty}{0}{}}(t_1)\circ\rest{\Sigma_1}u=0=P^{}_{\intervallo{0}{\infty}{}}(t_2)\circ\rest{\Sigma_2}u} \,.
\end{multline*}
All these spaces are closed subspaces of $FE^s_{\Gamma}(M,[t_1,t_2],D^{})$. Since $FE^s_{\Gamma}(M,[t_1,t_2],D^{})$ is a free Hilbert $\Gamma$-module due to \Cref{inivpwellgammatempcomp}, its left action representation transfers to $FE^s_{\Gamma,\mathrm{(a)APS}}(M,[t_1,t_2],D^{})$; moreover, it is $\Gamma$-invariant as it is defined with spectral projectors and restrictions as $\Gamma$-invariant operators. \Cref{prophilbertgammamodules} (2) then implies that they are projective Hilbert $\Gamma$-submodules on their own right. The choice $s=0$ is of particular interest as the related spaces then appear to be the correct domains for $\Gamma$-Fredholmness. We thus define the \textit{Dirac operator under APS boundary conditions} to be
\begin{equation*}
D^{}_{\mathrm{APS}}\,:\,FE^0_{\Gamma,\mathrm{APS}}(M,[t_1,t_2],D^{})\,\rightarrow L^2(\spinb^{-}_{}(M))
\end{equation*}
and the \textit{Dirac operator under aAPS boundary conditions} to be
\begin{equation*}
D^{}_{\mathrm{aAPS}}\,:\,FE^0_{\Gamma,\mathrm{aAPS}}(M,[t_1,t_2],D^{})\,\rightarrow L^2(\spinb^{-}_{}(M)) \quad.
\end{equation*}
\newpage
\begin{theo}\label{indexDgaAPStwist}
Let $M$ be a temporal compact, even-dimensional globally hyperbolic spatial $\Gamma$-manifold with time domain $[t_1,t_2]$, $\spinb^{+}(M)\rightarrow M$ the $\Gamma$-spin bundle of positive chirality; the $\Gamma$-invariant Dirac operators $D_{\mathrm{APS}}$ and $D_{\mathrm{aAPS}}$ as lifts of Dirac operators on the base manifold are $\Gamma$-Fredholm 
with $\Gamma-$indices
\begin{equation*}
\Index_{\Gamma}(D_{\mathrm{APS}})= \Index_\Gamma\left(Q_{--}(t_2,t_1)\right) \quad\text{and}\quad
\Index_{\Gamma}(D_{\mathrm{aAPS}})=\Index_\Gamma\left(Q_{++}(t_2,t_1)\right) \quad.
\end{equation*}
\end{theo}
\begin{proof}
We denote with 
\begin{equation}\label{boundcondopgaps}
\mathbb{P}_{+}:=\left(P_{\intervallro{0}{\infty}{}}(t_1)\circ\rest{\Sigma_1}\right)\oplus \left(P_{\intervalllo{-\infty}{0}{}}(t_2)\circ\rest{\Sigma_{2}}\right) 
\end{equation}
the boundary condition operator for APS boundary conditions and with
\begin{equation}\label{boundcondopgaaps}
\mathbb{P}_{-}:=\left(P_{\intervallo{-\infty}{0}{}}(t_1)\circ\rest{\Sigma_1}\right)\oplus \left(P_{\intervallo{0}{\infty}{}}(t_2)\circ\rest{\Sigma_{2}}\right)
\end{equation}
the operator for aAPS boundary conditions. The main task is to show that 
\begin{multline*}
\mathbb{P}_{+}\oplus D :\\
FE^0_{\Gamma,\mathrm{APS}}(M,[t_1,t_2],D)\,\rightarrow \,\left[L^2_{\intervallro{0}{\infty}{}}(\spinb^{+}(\Sigma_1))\oplus L^2_{\intervalllo{-\infty}{0}{}}(\spinb^{+}(\Sigma_2))\right]\oplus L^2(\spinb^{-}(M))
\end{multline*}
and
\begin{multline*}
\mathbb{P}_{-}\oplus D: \\
FE^0_{\Gamma,\mathrm{aAPS}}(M,[t_1,t_2],D)\,\rightarrow \,\left[L^2_{\intervallo{-\infty}{0}{}}(\spinb^{+}(\Sigma_1))\oplus L^2_{\intervallo{0}{\infty}{}}(\spinb^{+}(\Sigma_2))\right]\oplus L^2(\spinb^{-}(M))
\end{multline*}
are $\Gamma$-Fredholm with the claimed $\Gamma$-indices. We want to apply \Cref{funcanagammalemma} with

\begin{align}
\mathscr{H}&=FE^0_{\Gamma}(M,[t_1,t_2],D) &\mathscr{H}_1&=L^2_{\intervallro{0}{\infty}{}}(\spinb^{+}(\Sigma_1))\oplus L^2_{\intervalllo{-\infty}{0}{}}(\spinb^{+}(\Sigma_2)) &\mathscr{H}_2&=L^2(\spinb^{-}(M))\nonumber \\
B&=D=D_{+} &  A&=\mathbb{P}_{+} \label{dataforlem1} 
\end{align}
to prove $\Gamma$-Fredholmness of $\mathbb{P}_{+}\oplus D$ by checking that $C=\mathbb{P}_{+}\vert_{\kernel{D}}\oplus \Iop{\mathscr{H}_2}$ is $\Gamma$-Fredholm. The use is legitimated, because $\mathscr{H}_1$ and $\mathscr{H}_2$ are Hilbert $\Gamma$-modules and $\mathscr{H}$ is a Hilbert $\Gamma$-module by \Cref{inivpwellgammatempcomp} and \Cref{prophilbertgammamodules} (2); boundness in the spririt of $\Gamma$-morphisms and surjectivity of $D_{+}$ follow from \Cref{inivpwellgammatempcomp}; the same counts for the restrictions to the hypersurfaces. The s-regularity of the spectral projections then imply that $\mathbb{P}_{+}$ is equally a $\Gamma$-morphism. We furthermore use the shorthand notations $\mathscr{H}_1(t_1)=L^2_{[0,\infty)}(\spinb^{+}(\Sigma_1))$ and $\mathscr{H}_1(t_2):=L^2_{(-\infty,0]}(\spinb^{+}(\Sigma_2))$ which imply 
\begin{align*}
&  &\mathscr{H}_1=\mathscr{H}_1(t_1)\oplus\mathscr{H}_1(t_2) & \\
\mathscr{H}^\perp_1(t_1)&=L^2_{(-\infty,0)}(\spinb^{+}(\Sigma_1)) & &\mathscr{H}^\perp_1(t_2)=L^2_{(0,\infty)}(\spinb^{+}(\Sigma_1)) \quad.
\end{align*}
The algebraic manipulations and arguments from \cite[Theorem 3.2.]{BaerStroh} can be applied here as well where one needs to express everything in terms of projective/free Hilbert $\Gamma$-modules, $\Gamma$-isomorphisms, $\Gamma$-dimensions, $\Gamma$-morphisms and $\Gamma$-compact operators:
\vspace*{1em}
\begin{itemize}
\item $\kernel{C}=\kernel{Q_{--}(t_2,t_1)}\oplus \SET{0_{\mathscr{H}_2}}$ has finite $\Gamma$-dimension;
\item $\range{C}$ is closed and has the form
\begin{equation*}
\Big\lbrace(u_1,Q_{-+}(t_2,t_1)u_1+Q_{--}(t_2,t_1)v) \in \mathscr{H}_1 \,\Big\vert\, v \in  \mathscr{H}_1^\perp(t_1)\Big\rbrace\oplus\mathscr{H}_2\,.
\end{equation*}
\item $\quotspace{\mathscr{H}_1\oplus\mathscr{H}_2}{\range{C}}\cong\range{C}^\perp$ with
\begin{equation*}
\range{C}^\perp=\left[-(Q_{-+}(t_2,t_1))^\ast\otimes\Iop{}\right]\left(\kernel{(Q_{--}(t_1,t_2))^\ast}\oplus\SET{0_{\mathscr{H}_2}}\right)\,.
\end{equation*}
\end{itemize}
\vspace*{0.5em}
\Cref{kernelisoQ} implies that $(Q_{-+}(t_2,t_1))^\ast$ is a $\Gamma$-isomorphism on $\kernel{(Q_{--}(t_1,t_2))^\ast}$ such that
$$ \range{C}^\perp \cong \kernel{(Q_{--}(t_1,t_2))^\ast}\oplus\SET{0_{\mathscr{H}_2}} $$
and the range of $C$ has finite $\Gamma$-codimension due to $\Gamma$-Fredholmness of $Q_{--}(t_2,t_1)$. We conclude that $C=\mathbb{P}_{+}\oplus D$ is $\Gamma$-Fredholmn with $\Gamma$-index
\begin{align*}
\Index_\Gamma(\mathbb{P}_{+}\oplus D)&=\dim_\Gamma\kernel{Q_{--}(t_2,t_1)} -\dim_\Gamma\kernel{(Q_{--}(t_2,t_1))^\ast}\\
&= \Index_\Gamma(Q_{--}(t_2,t_1)).
\end{align*}
The second statement follows in the same way by choosing
\begin{equation*}
\mathscr{H}_1=L^2_{\intervallo{-\infty}{0}{}}(\spinb^{+}(\Sigma_1))\oplus L^2_{\intervallo{0}{\infty}{}}(\spinb^{+}(\Sigma_2))\quad\text{and}\quad A=\mathbb{P}_{-}
\end{equation*}
in \clef{dataforlem1} with the same legimitations. We replace the shorthand notations from the upper case to $\mathscr{H}_1(t_1)=L^2_{(-\infty,0)}(\spinb^{+}(\Sigma_1))$ and $\mathscr{H}_1(t_2)=L^2_{(0,\infty)}(\spinb^{+}(\Sigma_2))$ such that $\mathscr{H}_1=\mathscr{H}_1(t_1)\oplus\mathscr{H}_1(t_2)$ and $\mathscr{H}^\perp_1(t_1)=L^2_{[0,\infty)}(\spinb^{+}(\Sigma_1))$. It turns out that the $\Gamma$-Fredholmness can be related to the $\Gamma$-Fredholm-property of $Q_{++}(t_2,t_1)$ with
\begin{eqnarray*}
\Index_\Gamma(\mathbb{P}_{-}\oplus D)&=& \Index_\Gamma(Q_{++}(t_2,t_1)) \quad.
\end{eqnarray*}
$\Gamma$-Fredholmness of $\mathbb{P}_{\pm}\oplus D$ implies $\Gamma$-Fredholmness of $D\oplus \mathbb{P}_{\pm}$ as this property does not depend on the order of the direct sum. The Dirac operators of interest are related to $\mathbb{P}_{\pm}$ via
\begin{equation}
D_{\mathrm{APS}}=D\vert_{\kernel{\mathbb{P}_{+}}}\quad \text{and} \quad D_{\mathrm{aAPS}}=D\vert_{\kernel{\mathbb{P}_{-}}}\quad.
\end{equation}
The main claim finally follows by applying \Cref{funcanagammalemma} with 
\begin{align*}
\mathscr{H}&=FE^0_{\Gamma}(M,[t_1,t_2],D)& \mathscr{H}_1&=L^2(\spinb^{-}(M)) & \mathscr{H}_2&=L^2_{\intervallo{-\infty}{0}{}}(\spinb^{+}(\Sigma_1))\oplus L^2_{\intervallo{0}{\infty}{}}(\spinb^{+}(\Sigma_2)) \\
A&=D=D_{+} & \quad B&=\mathbb{P}_{+} 
\end{align*}
and $D_{\mathrm{APS}}$ becomes $\Gamma$-Fredholm with index $\Index_{\Gamma}(Q_{--}(t_2,t_1))$; applying \Cref{funcanagammalemma} with 
\begin{align*}
\mathscr{H}&=FE^0_{\Gamma}(M,[t_1,t_2],D) &  \mathscr{H}_1&=L^2(\spinb^{-}(M)) & \mathscr{H}_2&=L^2_{\intervallo{-\infty}{0}{}}(\spinb^{+}(\Sigma_1))\oplus L^2_{\intervallo{0}{\infty}{}}(\spinb^{+}(\Sigma_2)) \\
A&=D=D_{+} & \quad B&=\mathbb{P}_{-} 
\end{align*}
proves on the other hand the $\Gamma$-Fredholmness of $D_{\mathrm{aAPS}}$. The applications are justified because $D$ and $\mathbb{P}_{\pm}$ are $\Gamma$-morphism and $\mathbb{P}_{\pm}$ are onto.
\end{proof}
We can prove in a similar fashion $\Gamma$-Fredholmness for $\tilde{D}^{}:=D^{}_{-}$. We define similarly
\begin{multline*}
FE^s_{\Gamma,\mathrm{APS}}(M,[t_1,t_2],\tilde{D}^{})\\
:=\SET{u \in FE^s_{\Gamma}(M,[t_1,t_2],\tilde{D}^{})\,\Big\vert\,P^{}_{\intervallo{-\infty}{0}{}}(t_1)\circ\rest{\Sigma_1}u=0=P^{}_{\intervallo{0}{\infty}{}}(t_2)\circ\rest{\Sigma_2}u}
\end{multline*}
\begin{multline*}
FE^s_{\Gamma,\mathrm{aAPS}}(M,[t_1,t_2],\tilde{D}^{})\\:=\SET{u \in FE^s_{\Gamma}(M,[t_1,t_2],\tilde{D}^{})\,\Big\vert\,P^{}_{\intervallro{0}{\infty}{}}(t_1)\circ\rest{\Sigma_2}u=0=P^{}_{\intervalllo{-\infty}{0}{}}(t_2)\circ\rest{\Sigma_1}u} \,.
\end{multline*}
These spaces are all Hilbert $\Gamma$-modules for the same reasons as for $D^{}$. Restricting $\tilde{D}^{}$ to these domains for $s=0$ defines the Dirac operators for (a)APS-boundary conditions with respect to negative chirality:
\begin{eqnarray*}
\tilde{D}^{}_{\mathrm{APS}}\,:\,FE^0_{\Gamma,\mathrm{APS}}(M,[t_1,t_2],\tilde{D}^{})\,\rightarrow L^2(\spinb^{+}_{}(M))\\
\tilde{D}^{}_{\mathrm{aAPS}}\,:\,FE^0_{\Gamma,\mathrm{aAPS}}(M,[t_1,t_2],\tilde{D}^{})\,\rightarrow L^2(\spinb^{+}_{}(M)) \quad.
\end{eqnarray*}
$\Gamma$-Fredholmness follows like in the proof of \Cref{indexDgaAPStwist}.
\begin{theo}\label{indexDneggaAPS}
Let $M$ be a temporal compact, even-dimensional globally hyperbolic spatial $\Gamma$-manifold with time domain $[t_1,t_2]$, $\spinb^{-}(M)\rightarrow M$ the $\Gamma$-spin bundle of negative chirality; the $\Gamma$-invariant Dirac operators $\tilde{D}_{\mathrm{APS}}$ and $\tilde{D}_{\mathrm{aAPS}}$ as lifts of Dirac operators on the base manifold are $\Gamma$-Fredholm  with $\Gamma$-indices
\begin{equation*}
\Index_{\Gamma}(\tilde{D}_{\mathrm{APS}})= \Index_\Gamma\left(\tilde{Q}_{++}(t_2,t_1)\right) \quad\text{and}\quad 
\Index_{\Gamma}(\tilde{D}_{\mathrm{aAPS}})= \Index_\Gamma\left(\tilde{Q}_{--}(t_2,t_1)\right) \quad.
\end{equation*}
\end{theo}
\begin{proof}
The boundary value operator of $FE^0_{\Gamma,\mathrm{APS}}(M,[t_1,t_2],\tilde{D}^{})$ coincides with the one of $FE^0_{\Gamma,\mathrm{aAPS}}(M,[t_1,t_2],D^{})$ and the boundary value operators of $FE^0_{\Gamma,\mathrm{aAPS}}(M,[t_1,t_2],\tilde{D}^{})$ and $FE^0_{\Gamma,\mathrm{APS}}(M,[t_1,t_2],D^{})$ are equal. The same proof strategy then shows that $\tilde{D}^{}_{\mathrm{APS}}$ and $\tilde{D}^{}_{\mathrm{aAPS}}$ are $\Gamma$-Fredholm with claimed $\Gamma$-indices.
\end{proof}

\section{Concluding remarks}\label{chap:concremarks}

We have restricted the proof to the classical spin Dirac operator and considered ordinary (anti) Atiyah-Patodi-Singer boundary conditions in order to keep the presentation of the $L^2$-$\Gamma$-Fredholm property simple. But the proof in fact works out in a more general setting and only little further modifications have to be taken into account. We first allow $M$ to have a $\text{spin}^c$-structure such that $M$ carries an associated Hermitian line bundle $L\rightarrow M$ which is either already defined on the covering or can be lifted from a corresponding line bundle $L_\Gamma \rightarrow M_\Gamma$ over the base $M_\Gamma$. The spinor bundles $\spinb^{\pm}(M)$ are then only locally defined. But if we consider the square root bundle of $L$, which is also only locally defined, the twisted spinor bundle $\spinb^{\pm}_L(M):=\spinb(M)\otimes L^{1/2}$ is a globally well-defined bundle and defined as $\text{spin}^c$ spinor bundle. Moreover, we can additionally twist the spin structure with any Hermitian $\Gamma$-vector bundle $E\rightarrow M$, leading to a twisted $\text{spin}^c$ spinor bundle $\spinb^{\pm}_{L,E}(M):=\spinb^{\pm}_{L}\otimes E$. We express both twistings with an overall twisting bundle $E_L$. Both $\Gamma$-vector bundles $L$ and $E$ are assumed to carry a metric connection. The Dirac operators $D^{E_L}_{\pm}: C^\infty(\spinb^{\pm}_{L,E}(M))\rightarrow C^\infty(\spinb^{\mp}_{L,E}(M))$ of our interest are supposed to commute or rather intertwine between the left action representation of $\Gamma$ on $M$.\\
\\
Another conceivable generalisation is applied on the boundary conditions. We already defined in \clef{gapsbc} and \clef{gaapsbc} or any two real numbers $a_1,a_2$ generalised (anti) Atiyah-Patodi-Singer boundary conditions. The presented proof of the $\Gamma$-Fredholmness works out as for ordinary APS or aAPS boundary conditions. We set as suitable boundary condition operators
\begin{equation*}
\begin{split}
\mathbb{P}_{+}&:=(P_{\intervallro{a_1}{\infty}{}}(t_1)\circ\rest{\Sigma_1})\oplus (P_{\intervalllo{-\infty}{a_2}{}}(t_2)\circ\rest{\Sigma_{2}}) \\
\mathbb{P}_{-}&:=(P_{\intervallo{-\infty}{a_1}{}}(t_1)\circ\rest{\Sigma_1})\oplus (P_{\intervallo{a_2}{\infty}{}}(t_2)\circ\rest{\Sigma_{2}}) \quad.
\end{split}
\end{equation*} 
The spaces of (twisted) finite energy spinor sections under generalised (a)APS boundary conditions in the $\Gamma$-setting, i.e.
\begin{equation*}
\begin{split}
FE^0_{\Gamma,\mathrm{APS}(a_1,a_2)}(M,[t_1,t_2],D_{\pm}^{E_L})&=\SET{u \in FE^0_{\Gamma}(M,[t_1,t_2],D_{\pm}^{E_L})\cap \kernel{\mathbb{P}_{\pm}}}\\ 
FE^0_{\Gamma,\mathrm{aAPS}(a_1,a_2)}(M,[t_1,t_2],D_{\pm}^{E_L})&=\SET{u \in FE^0_{\Gamma}(M,[t_1,t_2],D_{\pm}^{E_L})\cap \kernel{\mathbb{P}_{\mp}}}\\ 
\end{split}
\end{equation*}
then turn out to be suitable domains for $D^{E_L}_{\pm}$, defining the Dirac operators under generalised (anti) Atiyah-Patodi-Singer boundary conditions:
\begin{equation}\label{gammafredoperresnew}
\begin{split}
D^{E_L}_{\pm,\mathrm{APS}(a_1,a_2)}\,&:\,FE^0_{\Gamma,\mathrm{APS}(a_1,a_2)}(M,[t_1,t_2],D_{\pm}^{E_L})\,\rightarrow L^2(\spinb^{\mp}_{L,E}(M))\\
D^{E_L}_{\pm,\mathrm{aAPS}(a_1,a_2)}\,&:\,FE^0_{\Gamma,\mathrm{aAPS}(a_1,a_2)}(M,[t_1,t_2],D_{\pm}^{E_L})\,\rightarrow L^2(\spinb^{\mp}_{L,E}(M))\quad,
\end{split}
\end{equation}
i.e. $D^{E_L}_{\pm,\mathrm{APS}(a_1,a_2)}=D^{E_L}_{\pm}\oplus\mathbb{P}_{\pm}$ and $D^{E_L}_{\pm,\mathrm{aAPS}(a_1,a_2)}=D^{E_L}_{\pm}\oplus\mathbb{P}_{\mp}$. \\
\\
Considering all these modifications, our main result \Cref{fredDpmgaAPStwist} extends to the following formulation.

\begin{theo}\label{fredDpmgaAPStwistnew}
Let $a_1,a_2 \in \R$, $M$ a temporal compact, globally hyperbolic spatial $\Gamma$-manifold with compact base $M_\Gamma$ and time domain $[t_1,t_2]$, and $\spinb^{\pm}_{L,E}(M)\rightarrow M$ the twisted $\Gamma$-spin bundles. The $\Gamma$-invariant Dirac operators 
\begin{equation*}
D^{E_L}_{\pm,\mathrm{APS}(a_1,a_2)}\,:\,FE^0_{\Gamma,\mathrm{APS}(a_1,a_2)}(M,[t_1,t_2],D^{E_L}_{\pm})\,\rightarrow L^2(\spinb^{\mp}_{L,E}(M))
\end{equation*}
and
\begin{equation*}
D^{E_L}_{\pm,\mathrm{aAPS}(a_1,a_2)}\,:\,FE^0_{\Gamma,\mathrm{aAPS}(a_1,a_2)}(M,[t_1,t_2],D^{E_L}_{\pm})\,\rightarrow L^2(\spinb^{\mp}_{L,E}(M))
\end{equation*}
as lifts of Dirac operators on the base manifold $M_\Gamma$ are $\Gamma$-Fredholm. The $\Gamma$-indices satisfy
\begin{equation}\label{genQindexrel}
\Index_{\Gamma}(D^{E_L}_{\pm,\mathrm{APS}(a_1,a_2)})=-\Index_{\Gamma}(D^{E_L}_{\pm,\mathrm{aAPS}(a_1,a_2)})
\end{equation}
and are related to the Dirac operators $D^{E_L}_{\pm,\mathrm{APS}}:=D^{E_L}_{\pm,\mathrm{APS}(0,0)}$ for ordinary (a)APS boundary conditions as follows:

\begin{align}\label{indexDgaAPSformintronew}
\Index_{\Gamma}(D^{E_L}_{+,\mathrm{APS}(a_1,a_2)})&= \Index_{\Gamma}(D^{E_L}_{+,\mathrm{APS}}) \nonumber\\
&+ \Chi_{\SET{a_2 < 0}}\dim_\Gamma\left(L^2_{(a_2,0]}(\spinb_{L,E}(\Sigma_{1}))\right)-\Chi_{\SET{a_2 > 0}}\dim_\Gamma\left(L^2_{(0,a_2]}(\spinb_{L,E}(\Sigma_{2}))\right)\nonumber\\
&+ \Chi_{\SET{a_1 > 0}}\dim_\Gamma\left(L^2_{[0,a_1)}(\spinb_{L,E}(\Sigma_{1}))\right)-\Chi_{\SET{a_1 < 0}}\dim_\Gamma\left(L^2_{[a_1,0)}(\spinb_{L,E}(\Sigma_{2}))\right)\nonumber \\[5pt]
&\text{and}\\[5pt]
\Index_{\Gamma}(D^{E_L}_{-,\mathrm{APS}(a_1,a_2)})&=\Index_{\Gamma}(D^{E_L}_{-,\mathrm{APS}}) \quad . \nonumber \\
&+\Chi_{\SET{a_2 > 0}}\dim_\Gamma\left(L^2_{(0,a_2]}(\spinb_{L,E}(\Sigma_{1}))\right)-\Chi_{\SET{a_2 < 0}}\dim_\Gamma\left(L^2_{(a_2,0]}(\spinb_{L,E}(\Sigma_{2}))\right) \nonumber\\
&+ \Chi_{\SET{a_1 < 0}}\dim_\Gamma\left(L^2_{[a_1,0)}(\spinb_{L,E}(\Sigma_{1}))\right)-\Chi_{\SET{a_1 > 0}}\dim_\Gamma\left(L^2_{[0,a_1)}(\spinb_{L,E}(\Sigma_{2}))\right)\nonumber \\ \nonumber
\end{align} 
\end{theo}

The proof relies on the well-posedness results in the twisted setting which we also showed in \cite{OD1} under more general assumptions on the spacetime $M$. These results still hold for the twisting bundle $E_L$. The resulting wave evolution operators $Q^{E_L}$ and $\tilde{Q}^{E_L}$ for $D^{E_L}_{\pm}$ are again $\Gamma$-isomorphisms between $\Gamma$-Sobolev spaces, unitary between $L^2$-spinor sections and properly supported $\Gamma$-invariant Fourier integral operators of order $0$ between sections of the twisted spinor bundles, but with the same canonical relation. Their decompositions with respect to generalised (anti) Atiyah-Patodi-Singer boundary conditions are then similarly defined according to \clef{genL2splitgamma}. The important observation is that these new decomposits differ from $Q^{E_L}_{\pm\pm}$ and $Q^{E_L}_{\pm\mp}$ in compositions with $\Gamma$-Fredholm restricted operators (see \clef{restrictedprojections} and \Cref{fredrestrictedprojections}) modulo $\Gamma$-compact remainders. This affects the $\Gamma$-indices of the still retained $\Gamma$-Fredholm property of the diagonal matrix entries. We remark that in the following $\range{P_I(t_j)}$ stands for $L^2_I(\spinb_{L,E}(\Sigma_j))$.
\begin{theo}\label{qgengammafred}
\begin{equation*}
\begin{split}
Q^{> a_2}_{\geq a_1}(t_2,t_1) &\in \mathscr{F}_\Gamma(\range{P_{\geq a_1}(t_1)},\range{P_{> a_2}(t_2)})\\
Q^{\leq a_2}_{< a_1}(t_2,t_1) &\in \mathscr{F}_\Gamma(\range{P_{< a_1}(t_1)},\range{P_{\leq a_2}(t_2)}) \\
\end{split}
\end{equation*}
for $a_1,a_2 \in \R$ with $\Gamma$-indices
\begin{align}
\Index_\Gamma\left(Q^{> a_2}_{\geq a_1}(t_2,t_1)\right)=&\quad\Chi_{\SET{a_2 > 0}}\dim_\Gamma\left(\range{P_{(0,a_2]}}\right)-\Chi_{\SET{a_2 < 0}}\dim_\Gamma\left(\range{P_{(a_2,0]}}\right)\nonumber\\
&+ \Chi_{\SET{a_1 < 0}}\dim_\Gamma\left(\range{P_{[a_1,0)}}\right)-\Chi_{\SET{a_1 > 0}}\dim_\Gamma\left(\range{P_{[0,a_1)}}\right)\nonumber\\
&+ \Index_\Gamma\left(Q_{++}(t_2,t_1)\right)\label{qgengammaindexpospp} \\[2pt]
\Index_\Gamma\left(Q^{\leq a_2}_{< a_1}(t_2,t_1)\right)=&\quad\Chi_{\SET{a_2 < 0}}\dim_\Gamma\left(\range{P_{(a_2,0]}}\right)-\Chi_{\SET{a_2 > 0}}\dim_\Gamma\left(\range{P_{(0,a_2]}}\right)\nonumber\\
&+ \Chi_{\SET{a_1 > 0}}\dim_\Gamma\left(\range{P_{[0,a_1)}}\right)-\Chi_{\SET{a_1 < 0}}\dim_\Gamma\left(\range{P_{[a_1,0)}}\right)\nonumber \\
&+\Index_\Gamma\left(Q_{--}(t_2,t_1)\right) \label{qgengammaindexposmm} \quad.
\end{align}
\end{theo}
$\Chi_{\SET{a>0}}$ and $\Chi_{\SET{a<0}}$ are abbreviations for the characteristic functions $\Chi_{(0,\infty)}(a)$ respectively $\Chi_{(-\infty,0)}(a)$.
\begin{proof}
We show that $Q^{> a_2}_{\geq a_1}(t_2,t_1)$ and $Q^{\leq a_2}_{< a_1}(t_2,t_1)$ are compositions of $\Gamma$-Fredholm operators. We already know for $a_1=a_2=0$ that the entries $Q_{++}(t_2,t_1)$ and $Q_{--}(t_2,t_1)$ are $\Gamma$-Fredholm. We can rewrite the projections for the domains and ranges of $Q^{> a_2}_{\geq a_1}(t_2,t_1)$ and $Q^{\leq a_2}_{< a_1}(t_2,t_1)$ as spectral projections with spectral cut at $0$ with the help of \clef{restrictedprojections} such that $Q_{++}(t_2,t_1)$ and $Q_{--}(t_2,t_1)$ can be recovered. We have for all $a_1, a_2 \in \R$ 
\begin{align*}
P_{> a_2}(t_2) &=\overline{P}^{>0}_{>a_2}(t_2) P_{>0}(t_2)+\mathscr{S}^1_\Gamma \quad &,\quad P_{\geq a_1}(t_1) &= \overline{P}^{\geq a_1}_{\geq 0}(t_1)P_{\geq a_1}(t_1)+\mathscr{S}^1_\Gamma \quad, \\[3pt]
P_{\leq a_2}(t_2) &= \overline{P}^{\leq 0}_{\leq a_2}(t_2)P_{\leq 0}(t_2)+\mathscr{S}^1_\Gamma \quad &,\quad P_{< a_1}(t_1) &= \overline{P}^{< a_1}_{< 0}(t_1)P_{< a_1}(t_1)+\mathscr{S}^1_\Gamma \quad.
\end{align*}
The defects are of $\Gamma$-trace class as these are either spectral projections with zero or bounded spectral range and thus s-smoothing. We get
\begin{equation*}
\begin{split}
Q^{> a_2}_{\geq a_1}(t_2,t_1)&=\overline{P}^{>0}_{>a_2}(t) P_{>0}(t_2) Q(t_2,t_1) \overline{P}^{\geq a_1}_{\geq 0}(t_1)P_{\geq a_1}(t_1) + \mathscr{S}^1_\Gamma \\
&= \overline{P}^{>0}_{>a_2}(t)Q_{++}(t_2,t_1) \overline{P}^{\geq a_1}_{\geq 0}(t_1)P_{\geq a_1}(t_1) + \mathscr{S}^1_\Gamma\quad
\end{split}
\end{equation*}
for all $a_1,a_2 \in \R$ and likewise 
\begin{equation*}
\begin{split}
Q^{\leq  a_2}_{< a_1}(t_2,t_1)&=\overline{P}^{\leq 0}_{\leq a_2}(t) P_{\leq 0}(t_2) Q_(t_2,t_1) \overline{P}^{< a_1}_{< 0}(t_1)P_{< a_1}(t_1) + \mathscr{S}^1_\Gamma \\
&= \overline{P}^{\leq 0}_{\leq a_2}(t) Q_{--}(t_2,t_1) \overline{P}^{< a_1}_{< 0}(t_1)P_{< a_1}(t_1) + \mathscr{S}^1_\Gamma \quad.
\end{split}
\end{equation*}
If we consider $Q^{> a_2}_{\geq a_1}(t_2,t_1)$ on $\range{P_{\geq a_1}(t_1)}$ and $Q^{\leq  a_2}_{< a_1}(t_2,t_1)$ as operator on $\range{P_{< a_1}(t_1)}$, we gain
\begin{equation}\label{qgenasqnormal}
\begin{split}
Q^{> a_2}_{\geq a_1}(t_2,t_1)&=\overline{P}^{>0}_{>a_2}(t_2)Q_{++}(t_2,t_1) \overline{P}^{\geq a_1}_{\geq 0}(t_1)+ \mathscr{S}^1_\Gamma \\
Q^{\leq  a_2}_{< a_1}(t_2,t_1)&= \overline{P}^{\leq 0}_{\leq a_2}(t_2) Q_{--}(t_2,t_1) \overline{P}^{< a_1}_{< 0}(t_1) + \mathscr{S}^1_\Gamma\quad.
\end{split}
\end{equation}
As we would like to show $\Gamma$-Fredholmness of $Q^{>a_2}_{\leq a_1}(t_2,t_1)$ and $Q^{\geq a_2}_{<a_1}(t_2,t_1)$ with the help of the known $\Gamma$-Fredholmness of $Q_{\pm\pm}$ from \Cref{propQposnegfred}, it is left to check that the restricted projections $\overline{P}^{>0}_{>a_2}(t_2)$, $\overline{P}^{\geq a_1}_{\geq 0}(t_1)$, $\overline{P}^{\leq 0}_{\leq a_2}(t_2)$ and $\overline{P}^{< a_1}_{< 0}(t_1)$ are $\Gamma$-Fredholm. To do so, we calculate the $\Gamma$-dimensions and -codimensions according to \Cref{fredrestrictedprojections}:
\begin{eqnarray*}
\dim_\Gamma\left(\range{P_{>0}}\cap (\range{P_{> a_2}})^\perp\right) 
&=& \Chi_{\SET{a_2 > 0}}\dim_\Gamma\range{P_{(0,a_2]}} \, , \\[1pt]
\dim_\Gamma\left(\range{P_{\geq a_1}}\cap (\range{P_{\geq 0}})^\perp\right) 
&=&\Chi_{\SET{a_1 < 0}}\dim_\Gamma\range{P_{[a_1,0)}} \, , \\[1pt]
\dim_\Gamma\left(\range{P_{\leq 0}}\cap (\range{P_{\leq a_2}})^\perp\right) 
&=& \Chi_{\SET{a_2 < 0}}\dim_\Gamma\range{P_{(a_2,0]}} \, , \\[1pt]
\dim_\Gamma\left(\range{P_{<a_1}}\cap (\range{P_{< 0}})^\perp\right) 
&=&\Chi_{\SET{a_1 > 0}}\dim_\Gamma\range{P_{[0,a_1)}} \, , \\[1pt]
\codim_\Gamma\left(\range{P_{>0}}\cap \range{P_{> a_2}}\right) 
&=& \Chi_{\SET{a_2 < 0}}\dim_\Gamma\range{P_{(a_2,0]}}\, , \\[1pt]
\codim_\Gamma\left(\range{P_{\geq 0}}\cap \range{P_{\geq a_1}}\right)
&=& \Chi_{\SET{a_1 > 0}}\dim_\Gamma\range{P_{[0,a_1)}}\, , \\[1pt] 
\codim_\Gamma\left(\range{P_{\leq 0}}\cap \range{P_{\leq a_2}}\right) 
&=& \Chi_{\SET{a_2 > 0}}\dim_\Gamma\range{P_{(0,a_2]}} \, ,\\[1pt]
\codim_\Gamma\left(\range{P_{<0}}\cap \range{P_{< a_1}}\right) 
&=& \Chi_{\SET{a_1 < 0}}\dim_\Gamma\range{P_{[a_1,0)}}\quad .
\end{eqnarray*}
All occuring $\Gamma$-dimensions are finite because these spaces are images of spectral projections with bounded spectral ranges. Hence, the above restricted projections are $\Gamma$-Fredholm with $\Gamma$-indices
\begin{eqnarray*}
\Index_\Gamma\left(\overline{P}^{>0}_{>a_2}(t_2)\right) &=& \Chi_{\SET{a_2 > 0}}\dim_\Gamma\left(\range{P_{(0,a_2]}}\right)-\Chi_{\SET{a_2 < 0}}\dim_\Gamma\left(\range{P_{(a_2,0]}}\right)\\
\Index_\Gamma\left(\overline{P}^{\geq 0}_{\geq a_1}(t_1)\right) &=& \Chi_{\SET{a_1 < 0}}\dim_\Gamma\left(\range{P_{[a_1,0)}}\right)-\Chi_{\SET{a_1 > 0}}\dim_\Gamma\left(\range{P_{[0,a_1)}}\right)\\
\Index_\Gamma\left(\overline{P}^{\leq 0}_{\leq a_2}(t_2)\right) &=& \Chi_{\SET{a_2 < 0}}\dim_\Gamma\left(\range{P_{(a_2,0]}}\right)-\Chi_{\SET{a_2 > 0}}\dim_\Gamma\left(\range{P_{(0,a_2]}}\right)\\
\Index_\Gamma\left(\overline{P}^{<0}_{< a_1}(t_1)\right) &=& \Chi_{\SET{a_1 > 0}}\dim_\Gamma\left(\range{P_{[0,a_1)}}\right)-\Chi_{\SET{a_1 < 0}}\dim_\Gamma\left(\range{P_{[a_1,0)}}\right)\quad.
\end{eqnarray*}
Finally, $Q^{> a_2}_{\geq a_1}(t_2,t_1)$ and $Q^{\leq  a_2}_{< a_1}(t_2,t_1)$ in \clef{qgenasqnormal} becomes $\Gamma$-Fredholm as compositions of $\Gamma$-Fredholm operators with $\Gamma$-trace class pertubations. The invariance with respect to $\Gamma$-compact pertubations and the additivity of the $\Gamma$-index with respect to compositions imply to the given formulae.
\end{proof}
This result carries over to the case of negative chirality without further modifications. The proof of \Cref{indexDgaAPStwist} is purely functional algebraic and can be extended without further contentual modifications. It turns out that the $\Gamma$-indices of \clef{gammafredoperresnew} coincide with the $\Gamma$-Fredholm matrix entries with respect to generalised (anti) Atiyah-Patodi-Singer boundary conditions proving the following $\Gamma$-index expressions.
\begin{theo}\label{indexDgaAPStwist}
Let $a_1,a_2 \in \R$, $M$ a temporal compact, globally hyperbolic spatial $\Gamma$-manifold with compact base $M_\Gamma$ and time domain $[t_1,t_2]$, $\spinb^{+}_{L,E}(M)\rightarrow M$ the $\Gamma$-spin bundle of positive chirality which is twisted with a Hermitian $\Gamma$-vector bundle $E_L\rightarrow M$. The $\Gamma$-invariant Dirac operators $D^{E_L}_{\mathrm{APS}(a_1,a_2)}$ and $D^{E_L}_{\mathrm{aAPS}(a_1,a_2)}$ as lifts of Dirac operators on the base manifold are $\Gamma$-Fredholm with $\Gamma$-indices
\begin{equation*}
\Index_{\Gamma}(D^{E_L}_{\mathrm{APS}(a_1,a_2)})= \Index_\Gamma\left(Q^{\leq a_2}_{< a_1}(t_2,t_1)\right) \,\text{\&}\,\,\Index_{\Gamma}(D^{E_L}_{\mathrm{aAPS}(a_1,a_2)})=\Index_\Gamma\left(Q^{> a_2}_{\geq a_1}(t_2,t_1)\right)\, .
\end{equation*}
\end{theo}
An analogue statement for $\tilde{D}^{E_L}_{\mathrm{APS}(a_1,a_2)}$ and $\tilde{D}^{E_L}_{\mathrm{aAPS}(a_1,a_2)}$ can be formulated. The details and additional steps of the proof are explained and performed in \cite{ODT}.

\appendix

\section{Manifolds of bounded geometry}\label{chap:manbound}

We give a short description about manifolds of bounded geometry, based on material from \cite{shubinspec} and \cite{kordyu2}. Suppose $\Sigma$ is a connected Riemannian manifold of dimension $n$. The Riemannian metric $\met$ induces a distance function $d:\Sigma\times\Sigma\rightarrow \R$ by taking the infimum of lengths of arcs, connecting two points and measured with respect to the Riemannian metric. The exponential geodesic map $\exp_p:T_p\Sigma \rightarrow \Sigma$ is defined through $\exp_p(v)=\gamma(1)$ where $\gamma$ is a geodesic, given in such a parametrisation, that $\gamma(0)=p$ and $\dot{\gamma}(0)=v$. If we take a ball $\oball{n}{r}{0}$ of radius $r>0$ around the zero vector in $T_p\Sigma$, the exponential map becomes a diffeomorphism onto its image $U_{r}(p)=\exp(\oball{n}{r}{0})$ which is an open neighbourhood of $p\in \Sigma$. Let
$$ r(p):=\sup\SET{r>0\,\Big\vert\, \exp_p:\oball{n}{r}{0}\rightarrow U_r(p)\,\,\text{diffeomorphism}} $$
The infimum over all points $p\in \Sigma$ then defines the \textit{injectivity radius} $r_{\mathsf{inj}}$.
For $r_{\mathsf{inj}}>0$ and $r\in (0,r_{\mathsf{inj}})$ the exponential map $\exp_p:\oball{n}{r}{0}\rightarrow U_r(p)$ is a diffeomorphism for all $p\in \Sigma$ which motivates the following definition.
\begin{defi}
Let $\Sigma$ be a Riemannian manifold; we say that $\Sigma$ is a \textit{manifold of/with bounded geometry} if the following two conditions are satisfied:
\begin{itemize}
\item[(a)] $r_{\mathsf{inj}}>0$;
\item[(b)] the transition maps from $U_r(p)\cap U_r(q)$ to $\R^n$ are bounded for any fixed $r\in (0,r_{\mathsf{inj}})$ and any points $p,q$ with open neighbourhoods $U_r(p),U_r(q)$ such that $U_r(p)\cap U_r(q)\neq \emptyset$.
\end{itemize}
\end{defi}
(a) implies that $\Sigma$ is complete and (b) is equivalent to the boundedness of all covariant derivatives of the Riemannian curvature tensor. Examples of such spaces are homogenous spaces with invariant metrics, covering manifolds of compact manifolds\bnote{F16} as well as leaves of foliations on a compact manifold. We call a vector bundle $E$ to be of bounded geometry if all derivatives of the vector bundle transition functions on $U_r(p)\cap U_r(q)\neq \emptyset$ are bounded for all $p,q \in \Sigma$. Examples of such vector bundles are $\Sigma\times\C$ and complexifications of $T\Sigma$, $T^\ast\Sigma$ as well as any other complexified natural bundle. \\
\\
For a manifold of bounded geometry $\Sigma$ there exists a radius $R>0$ such that for all $r\in (0,R)$ the manifold can be covered by open balls $\oball{}{r}{p_i}$ around countably many points $p_i \in \Sigma$:
\begin{equation}\label{boundgeomcov}
\Sigma=\bigcup_{i}\oball{}{r}{p_i} \quad;
\end{equation}
this covering has the property that if we double the radii, but fixing the centres of the open balls in $\Sigma$, this covering has a finite number of non-empty intersections. This implies the existence of a suitable partition of unity (see \cite[Lem.1.2]{shubinspec} and \cite[Lem.1.4]{shubinspec}).
\begin{lem}\label{manboundgeopartun}
For every $R>0$ there exists a partition of unity $\SET{\phi_i}$ on $\Sigma$ such that
\begin{itemize}
\item[(1)] $\phi_i\in C^\infty_\comp(\Sigma,\R_{\geq 0})$ with $\mathrm{supp}(\phi_i)\subset \oball{}{2r}{p_i}$ for all $p_i$ from \clef{boundgeomcov}.
\item[(2)] all derivatives are uniformly bounded. 
\end{itemize}
\end{lem}
Having such a partition of unity enables to define Sobolev spaces $H^s(\Sigma)$ on manifolds with bounded geometry for any $s\in \R$ as follows:
\begin{equation}\label{sobolevboundedgeonorm}
\norm{u}{H^s(\Sigma)}^2:=\sum_{i}\norm{\phi u}{H^s(\oball{}{2r}{p_i}))}^2\quad.
\end{equation}
$H^s(\Sigma)$ is then defined to be the completion of $C^\infty_\comp(\Sigma)$ with respect to \clef{sobolevboundedgeonorm}. This can be easily extended to sections of a vector bundle of bounded geometry.\\
\\
An element $B$ in $\mathsf{B}\ydo{m}{\mathsf{prop}}(\Sigma,E)$ is a pseudo-differential operator of order $m \in \R$ with uniformly bounded symbol such that $B$ differs from the quantisation of the symbol in a smoothing operator on the open ball $\oball{n}{}{0}$ in any coordinate system. One also supposes in addition the following requirements:
\vspace*{0.5em}
\begin{itemize}
\item[a)] (c-locality): there exists a constant $c >0$ such that the Schwartz kernel $K_B$ of $B$ satisfies $K_B(x,y)=0$ if the geodesic distance is bigger than $c$.
\item[b)] For any $\epsilon > 0$ all covariant derivatives of the Schwartz kernel are uniformly bounded on $(\Sigma\times \Sigma)\setminus \mathcal{U}_\epsilon$ where $\mathcal{U}_\epsilon$ is an $\epsilon$-neighbourhood of the diagonal in $\Sigma$.
\end{itemize}
\vspace*{0.5em}
If the complete symbol admits a uniformly asymptotic expansion in any coordinate system, the related pseudo-differential operator is called classical. A smoothing operator is a continuous map between $C^\infty_\comp(M,E)$ and $C^{-\infty}(M,E)$ which extends to a continuous mapping between any Sobolev spaces $H^s(\Sigma,E)$ of bounded geometry. We denote the class of such operators with $\mathsf{B}\ydo{-\infty}{}(\Sigma,E)$. Any uniformly bounded pseudo-differential operator of order $m \in \R$ can be represented as a sum of a smoothing and a uniformly bounded, properly supported pseudo-differential operator of the same order.\\
\\
From now on, let $B \in \mathsf{B}\ydo{m}{\mathsf{prop,cl}}(\Sigma,E)$ be a uniformly elliptic and positive operator with scalar positive definite principal symbol ${\sigma}_m(B)(p,\xi)$ which lies inside a sector. Following \cite{kordyu}, one can construct a parametrix $C(\lambda)$ for the operator $(\lambda-B)$ for $\lambda \in \C\setminus\R_{+}$ as for these values $(\lambda-B)$ stays elliptic. It is shown that $C(\lambda)\in \mathsf{B}\ydo{-m}{}(\Sigma,E)$ and a remainder $R(\lambda)\in\mathsf{B}\ydo{-\infty}{}(\Sigma,E)$ are bounded maps between Sobolev spaces of bounded geometry for those $\lambda \in \Lambda_\theta:=\SET{\lambda\in \C\,\vert\, \absval{\mathrm{arg}(\lambda)} > \theta}$ with $\theta \in (0,\uppi/2)$.
$\Lambda_\theta$ is a sector in the complex plane with vertex at $0\in \C$. W.l.o.g. we assume that $\Lambda_\theta$ does not contain any point $\lambda^\ast$ such that the principal symbol of $B$ is vanishing for all $\xi\neq 0$. The construction of the parametrix works as usual by inverting the principal symbol in the asymptotic expansion of the complete symbol of $B$, giving the principal symbol of $C(\lambda)$, while the subprincipal symbols are defined recursively via transport equation. The asymptotic summation of these symbols yields a local parametrix and by gluing all coordinate patches with the partition of unity from \Cref{manboundgeopartun}, it gives a global parametrix. \\
\\
W.l.o.g. we assume that $\Lambda_\theta$ does not intersect with the spectrum $\upsigma(B)$ of $B$. Then $(\lambda-B)$ becomes invertible as unbounded operator with resolvent $\mathsf{R}(B,\lambda)=(\lambda-B)^{-1}$ which is bounded on $H^s(\Sigma,E)$. Thus, for any $\lambda \in \Lambda_{\theta,r}:=\Lambda_\theta\cap\SET{\lambda \in \C\,\vert\, r < \absval{\lambda}}$ with $r>0$ the resolvent can be expressed with the parametrix: $\mathsf{R}(B,\lambda) = C(\lambda)-\mathsf{R}(B,\lambda)\circ R(\lambda)$. The composition on the right-hand side is again a smoothing operator and finally $\mathsf{R}(B,\lambda)\in \mathsf{B}\ydo{-m}{}(\Sigma,E)$. In addition, the resolvent becomes a bounded operator from $H^s(\Sigma,E)$ to $H^{s+m}(\Sigma,E))$ for any $s\in \R$ and $\lambda \in \Lambda_{\theta,r}$ with $r>0$ and $\theta \in (0,\uppi/2)$. This makes $B$ a sectorial operator with sector $\Lambda_{\theta,r}^\complement$. The holomorphic functional calculus for sectorial operators is herewith justiefied to apply: let $f$ be a function, which can be extended to an entire function, such that for any $\eta \in \R$ the function $\R_{+}\ni x\mapsto f(x+\Imag y)$ is a Schwartz function with uniformly bounded seminorms on compact subsets in $\R$. $f(B)$ can be expressed as Cauchy integral 
\vspace*{-1em}
\begin{figure}[H]
\centering
\begin{minipage}{0.35\textwidth}
\begin{tikzpicture}

\draw[->,thick] (0,0) -- (4,0);	
\draw[->,thick] (0,-2.9) -- (0,2.9);
\draw[color=gray!40,thick] (0,0) -- (4,0);
\node (A) at (3.8,2.3) [draw] {$\mathbbm{C}$};

\filldraw[draw=gray!60,fill=gray!60] (0.5,0.833) arc (61.5:298.5:0.972) -- (0.5,-0.833)--(1.65,-2.75) -- (-3,-2.75) -- (-3,2.75) -- (1.65,2.75) -- (0.5,0.833);

(0,0)--(1.5,2.5)--(-3,2.5)--(-3,-2.5)--(1.5,-2.5)--(0,0);

\draw[thick,dashed] (0.5,0.833)--(1.65,2.75);
\draw[thick,dashed] (0.5,-0.833)--(1.65,-2.75);
\draw[thick,dashed] (0.5,0.833) arc (61.5:298.5:0.972);
\node (B) at (2.5,0.8) {$\upsigma(B)$};
\draw (2.5,0.6) -- (2,0);
\draw (2.5,0.6) -- (3.5,0);
\draw[dotted] (1.15,0) arc (0:61.5:1.15) ;
\draw[dotted] (0,0) -- (0.5,0.833);
\node (C) at (0.45,0.3) {$\theta$};






\node (D) at (-1.5,2.1) {$\Lambda_{\theta,r}$};
\fill (0,0) circle (1.5pt);

\draw (0,0)--(-0.687,0.687);
\node (H) at (-0.55,0.3) {$r$};

\end{tikzpicture}
\caption{Keyhole-sector $\Lambda_{\theta,r}$.}\label{sectorskordyukov}
\end{minipage}
\hspace{\fill}
\begin{minipage}{0.6\textwidth}
\begin{equation}\label{cauchyoperator}
f(B):=\frac{1}{2\uppi \Imag}\int_\gamma f(\lambda)\mathsf{R}(B,\lambda)\differ \lambda
\end{equation}
\vspace*{0.25em}
\noindent where $\gamma$ is a Hankel-like contour which acts as boundary curve of the keyhole-sector $\Lambda_{\theta,r}$. This enables to consider the functions $f(B)=\expe{-tB}$ for $t\in \R_{+}$ such that $B$ becomes a generator of a parabolic semigroup with holomorphic extension for $t\in \C$. Complex powers ($z\in \C$) can be either defined directly by setting $f(\lambda)=\lambda^z$ with the branch chosen, such that $\lambda^z=\expe{z\log{\lambda}}$ for $\lambda >0$, or with the parabolic semigroup via
\begin{equation}\label{complexpowers}
B^z:=\frac{1}{\Gamma(-z)}\int_0^\infty x^{-(z+1)}\expe{-xB}\differ x
\end{equation}
for $\Rep{z}<0$ and the Gamma function $\Gamma(-z)$ which is holomorphic for complex arguments with positive real part.
\end{minipage}

\end{figure}
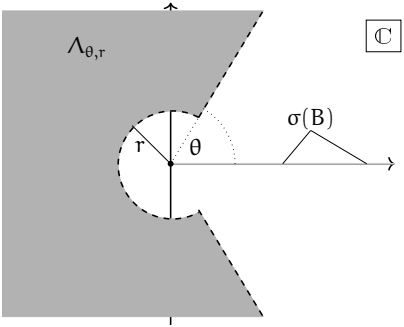
\noindent The semigroup property then carries over to the case with complex powers: let $z,w \in \C$ have negative real parts then $B^w\circ B^z=B^{z+w}$. 
This semigroup property can be extended to powers with positive real part: we choose for $\Rep{z}>0$ a $(-k) \in \N_0$ such that $\Rep{z+k}<0$. Hence, $B^{-k}$, $B^{z+k}$ and their composition are defined such that we can define complex powers of $B$ for any $z\in \C$ with positive real part via $B^z:=B^{z+k}\circ B^{-k}$.
It is clear from the definitions in \clef{complexpowers} that $B^z$ is a uniformly bounded pseudo-differential operator. The principal symbol of $B^z$ for $\Rep{z}<0$ is defined by \clef{cauchyoperator} where the resolvent is replaced with $({\sigma}_{m}(B)(p,\xi)-\lambda)^{-1}$:
\begin{equation}\label{princsymbcomplexpower}
{\sigma}_{m\Rep{z}}(B^z)(p,\xi):=\frac{1}{2\uppi\Imag}\int_\gamma \frac{\lambda^z}{\lambda-{\sigma}_{m}(B)(p,\xi)} \differ \lambda=({\sigma}_{m}(B)(p,\xi))^{z} \quad.
\end{equation}
The second equality follows with a contour integration argument. The factorising property of the principal symbol and the semigroup-property of complex powers allows to extend the equality to all $z\in \C$. 
We sum up the results with all necessary details.
\begin{prop}[cf. Proposition 1 in \cite{kordyu}]\label{kordyuresults}
Let $B \in \mathsf{B}\ydo{m}{\mathsf{prop,cl}}(\Sigma,E)$ be a positive and uniformly elliptic operator with scalar positive definite principal symbol ${\sigma}_m(B)(p,\xi)$ such that $({\sigma}_m(B)(p,\xi)-\lambda)$ is not vanishing for $\lambda$ in a sector $\Lambda_{\theta}\subset\uprho(B)$ for $\theta \in (0,\uppi/2)$ and $\xi\in \dot{T}_p\Sigma$. Then the following holds: 
\begin{itemize}
\item[(1)] there exists a value $r>0$ such that for $\lambda \in \Lambda_{\theta,r}\subset \C$ the operator $B$ becomes sectorial and the resolvent satisfies $\mathsf{R}(B,\lambda)\in \mathsf{B}\ydo{-m}{}(\Sigma,E)$ with principal symbol $({\sigma}_m(B)(p,\xi)-\lambda)^{-1}$; moreover $\mathsf{R}(B,\lambda)$ becomes a bounded operator from $H^s(\Sigma,E)$ to $H^{s+m}(\Sigma,E)$ for any $s\in \R$.
\item[(2)] $B$ generates a holomorphic semigroup $\expe{-zB}$ for $\Rep{z}>0$ in $L^2(\Sigma,E)$.
\item[(3)] $B^z \in \mathsf{B}\ydo{m\Rep{z}}{}(\Sigma,E)$ for all $z\in \C$ with principal symbol \clef{princsymbcomplexpower}.
\end{itemize}
\end{prop}
The last assertion is an equivalent of Seeley's theorem for complex powers in the setting of manifolds and vector bundles with bounded geometry.\\
\\ 
We have already pointed out in \Cref{remfun} (2) that a Galois covering with compact base can be viewed as manifold of bounded geometry. The role of the vector bundle of bounded geometry $E$ reduces to a $\Gamma$-vector bundle. The subordinated partition of unity is also uniformly bounded due to $\Gamma$-invariance such that the Sobolev spaces of bounded geometry transfer to the $\Gamma$-Sobolev spaces in the $\Gamma$-setting. Any pseudo-differential operator, which acts between sections of a $\Gamma$-vector bundles, can be considered as uniformly bounded pseudo-differential operator, acting between vector bundles of bounded geometry. Elements in $\ydo{\ast}{\Gamma,\mathsf{prop}}(\Sigma,E)$ correspond to uniformly bounded pseudo-differential operators which commute with $L^E_\gamma$ for all $\gamma \in \Gamma$. The $c$-locality property is implied by $\Gamma$-invariance and properly supportness. The space $\mathsf{B}\ydo{-\infty}{}(\Sigma,E)$ corresponds to s-smoothing operators $S\ydo{-\infty}{\Gamma}(\Sigma,E)$ such that $\mathsf{B}\ydo{m}{(\mathsf{cl})}(\Sigma,E)$ corresponds to (classical) s-regular operators $S\ydo{m}{\Gamma,(\mathrm{cl})}(\Sigma,E)$.
\begin{table}[H]
\centering
\begin{tabular}{|c c c|}
\hline
$\Sigma$ has bounded geometry & $\rightarrow$ & $\Sigma$ $\Gamma$-manifold \\
\hline
$E$ has bounded geometry &$\rightarrow$ & $E$ $\Gamma$-vector bundle \\
\hline
$H^s(\Sigma,E)$ &$\rightarrow$ & $H^s_\Gamma(\Sigma,E)$ \\
\hline
$\mathsf{B}\ydo{m}{\mathsf{prop}}(\Sigma,E)$ &$\rightarrow$ & $\ydo{m}{\Gamma,\mathsf{prop}}(\Sigma,E)$ \\
\hline
$\mathsf{B}\ydo{-\infty}{}(\Sigma,E)$ &$\rightarrow$ & $S\ydo{-\infty}{\Gamma}(\Sigma,E)$ \\
\hline
$\mathsf{B}\ydo{m}{}(\Sigma,E)$ &$\rightarrow$ & $S\ydo{m}{\Gamma}(\Sigma,E)$ \\
\hline
\end{tabular}
\caption{Correspondences between quantities in the bounded geometry setting to the one in $\Gamma$-setting.}
\end{table}

\end{document}